\documentclass[final,letterpaper,leqno, 11pt]{article}

\usepackage{secdot}
\usepackage[letterpaper,left=1in,right=1in, top = 1.2in, bottom=1.2in]{geometry}
\usepackage{titling}
\usepackage{amssymb, amsmath, amsfonts, mathrsfs, cite}
\usepackage{ntheorem}
\usepackage{graphicx} 
\usepackage{booktabs}
\usepackage{color, colortbl}
\usepackage[T1]{fontenc}
\renewcommand\theenumi{\roman{enumi}}

\definecolor{Maroon}{cmyk}{.4,1,.3,.2}
\definecolor{LightMaroon}{cmyk}{.04,0.1,.03,.02}
\definecolor{Gray}{cmyk}{0,0,0,.5}
\definecolor{Green}{cmyk}{1,0,1,0}
\definecolor{Red}{cmyk}{0,1,.8,0}
\definecolor{Orange}{cmyk}{0,.55,1,0}
\definecolor{LightOrange}{cmyk}{0,.12,0.3,0.01}
\definecolor{Blue}{cmyk}{.6,.1,.1,.1}

\newcommand{\dt}{\partial_t}
\newcommand{\dx}{\partial_x}
\newcommand{\dy}{\partial_y}
\newcommand{\dz}{\partial_z}
\newcommand{\dv}{\partial_v}
\newcommand{\inner}[3]{\big(#2  ,  #3 \big)\hspace{-0.3mm}_{#1}}
\newcommand{\norm}[2]{ \big \| #2 \big \|_{#1}}

\newcommand{\pair}[3]{\big \langle #2, #3 \big\rangle \hspace{-0.3mm}_{#1}}
\newcommand{\set}[2]{\left\lbrace #1 : #2  \right\rbrace}
\newcommand{\seq}[1]{\big\{ #1 \big\} }

\DeclareMathOperator*{\esssup}{ess\,sup}
\DeclareMathOperator*{\essinf}{ess\,inf}

\DeclareMathOperator{\dist}{dist}
\DeclareMathOperator{\diag}{diag}
\DeclareMathOperator{\trace}{tr}
\DeclareMathOperator{\range}{Range}
\DeclareMathOperator{\kernel}{Kernel}

\def\Xint#1{\mathchoice
	{\XXint\displaystyle\textstyle{#1}}%
	{\XXint\textstyle\scriptstyle{#1}}%
	{\XXint\scriptstyle\scriptscriptstyle{#1}}%
	{\XXint\scriptscriptstyle\scriptscriptstyle{#1}}%
	\!\int}
\def\XXint#1#2#3{{\setbox0=\hbox{$#1{#2#3}{\int}$}
		\vcenter{\hbox{$#2#3$}}\kern-.5\wd0}}
\def\dashint{\Xint-}

\def\YYint#1#2#3{{\setbox0=\hbox{$#1{#2#3}{\int}$}
		\lower1ex\hbox{$#2#3$}\kern-.46\wd0}}
\def\YYYint#1#2#3{{\setbox0=\hbox{$#1{#2#3}{\int}$}
		\lower0.35ex\hbox{$#2#3$}\kern-.48\wd0}}

\def\ZZint#1#2#3{{\setbox0=\hbox{$#1{#2#3}{\int}$}
		\raise1.15ex\hbox{$#2#3$}\kern-.57\wd0}}
\def\ZZZint#1#2#3{{\setbox0=\hbox{$#1{#2#3}{\int}$}
		\raise0.85ex\hbox{$#2#3$}\kern-.53\wd0}}

\numberwithin{equation}{section}

\newcommand{\rc}{\mathrm{c}}
\newcommand{\rd}{\mathrm{d}}
\newcommand{\re}{\mathrm{e}}
\newcommand{\rw}{\mathrm{w}}
\newcommand{\rr}{\mathrm{r}}
\newcommand{\rs}{\mathrm{s}}

\newcommand{\rF}{\mathrm{F}}
\newcommand{\rE}{\textsc{e}}
\newcommand{\rI}{\textsc{i}}

\newcommand{\rM}{\mathrm{M}}
\newcommand{\rN}{\mathrm{N}}
\newcommand{\rT}{\mathrm{T}}

\newcommand{\rV}{\mathrm{V}}
\newcommand{\rX}{\textsc{x}}
\newcommand{\rY}{\textsc{y}}

\newcommand{\bbR}{\mathbb{R}}

\newcommand{\cB}{\mathcal{B}}

\newcommand{\cE}{\mathcal{E}}
\newcommand{\cD}{\mathcal{D}}
\newcommand{\cH}{\mathcal{H}}
\newcommand{\cL}{\mathcal{L}}

\newcommand{\cU}{\mathcal{U}}

\newcommand{\cW}{\mathcal{W}}
\newcommand{\cX}{\mathcal{X}}
\newcommand{\cY}{\mathcal{Y}}

\newcommand{\sA}{\mathscr{A}}
\newcommand{\sB}{\mathscr{B}}

\newcommand{\sG}{\mathscr{G}}
\newcommand{\sH}{\mathscr{H}}

\newcommand{\sK}{\mathscr{K}}

\newcommand{\sM}{\mathscr{M}}

\newcommand{\sP}{\mathscr{P}}

\newcommand{\sX}{\mathscr{X}}
\newcommand{\sY}{\mathscr{Y}}

\newtheorem{theorem}      {Theorem}      [section]
\newtheorem{definition} [theorem] {Definition} 
\newtheorem{proposition} [theorem] {Proposition}  
\newtheorem{lemma}   [theorem]     {Lemma}       
\newtheorem{corollary}   [theorem]      {Corollary}  
\newtheorem{remark}  [theorem]     {Remark}      

\newenvironment{proof}{ 
	\par{
		\color{black}{
			\mbox{{\bf Proof.}}\color{black}
		}
	}
	\ignorespaces\color{black}
}

\begin{document}
		\title{On the Global Dynamics of an Electroencephalographic Mean Field Model of the Neocortex} 
		\author{Farshad Shirani\thanks{School of Aerospace Engineering, Georgia Institute of Technology, Atlanta, GA 30332} \and Wassim M. Haddad\thanks{School of Aerospace Engineering, Georgia Institute of Technology, Atlanta, GA 30332} \and Rafael de la Llave\thanks{School of Mathematics, Georgia Institute of Technology, Atlanta, GA 30332}
		}
	\maketitle
	\date{}

\begin{abstract}
	This paper investigates the global dynamics of a mean field model of the electroencephalogram developed by Liley \emph{et al.}, 2002. 
	The model is presented as a system of coupled ordinary and partial differential equations with periodic boundary conditions.
	Existence, uniqueness, and regularity of weak and strong solutions of the model are established in appropriate function spaces, and the associated initial-boundary value problems are proved to be well-posed. 
	Sufficient conditions are developed for the phase spaces of the model to ensure nonnegativity of certain quantities in the model, as required by their biophysical interpretation. 
	It is shown that the semigroups of weak and strong solution operators possess bounded absorbing sets for the entire range of biophysical values of the parameters of the model. 
	Challenges towards establishing a global attractor for the model are discussed and it is shown that there exist parameter values for which the constructed semidynamical systems do not possess a compact global attractor due to the lack of the asymptotic compactness property. 
	Finally, using the theoretical results of the paper, instructive insights are provided into the complexity of the behavior of the model and computational analysis of the model.
\end{abstract}

%


\section{Introduction}
Inspired by the seminal work of Alan Hodgkin and Andrew Huxley on modeling the flow of ionic currents through
the membrane of a giant nerve fiber, numerous biophysical and mathematical models have been developed towards
understanding the neurophysiology of the central nervous system and the underlying mechanism of the various phenomena 
that emerge during its vital operation in the body; many of which still remain a mystery to researchers 
\cite{HodgkinHuxley:JPhysiology:1952, Fitzhugh:Biophysical:1961, Nagumo:IRE:1962, Wilson:Biophysical:1972}.
In particular, exploring the core component of the central nervous system---the brain---substantial effort has been devoted to develop models at different levels of scope;  
from the \emph{molecular and intercellular} level dealing with the transportation of ions and the enzymatic kinetics of neurotransmitter-receptor binding at ion channels; 
to the \emph{single cell and intracellular} level dealing with the creation and transmission of action potential;
to the \emph{population and neuronal network} level dealing with the average behavior and synchronized activity of neuronal ensembles;
to the \emph{system level} dealing with the systematic operation and interaction between cortical and subcortical components of the brain; 
and eventually, to the \emph{behavioral and cognitive} level dealing with the integrated mental activity and the creation of mind \cite{Keener:MathematicalPhysiology:2009, Bear:Neuroscience:2016, Ermentrout:MathematicalNeuroscience:2010, Koch:BiophysicsComputation:2004, Gerstner:SpikingNeuron:2002, Wilson:PLOS:2012, SanzLeon:NeuroImage:2015, Sporns:Nature:2014}.


As an effective methodology for developing models at the population and network level, mean field theory has been employed to construct approximate models for interconnected populations of neurons by averaging the effect of all other neurons on a given individual neuron inside a population. 
The resulting \emph{averaged neuron} can be used to analyze the overall temporal behavior of a single population of neurons---leading to a \emph{neural mass} model---or can be considered as a locally averaged component of a  continuum of neural populations---leading to  a spatio-temporal \emph{mean field} model.  
These models are particularly useful in analyzing the electrophysiological activity of neuronal ensembles using local field potentials and electroencephalograms 
\cite{Pinotsis:Frontiers:2014, Moran:Frontiers:2013, Robinson:Neuropsychopharmacology:2003, Coombes:NeuralFields:2014}.

The evolution equations that describe a mean field model of neural activity in the cortex are in the form of a system of partial differential equations, or a system of coupled ordinary and partial differential equations. 
The theory of infinite-dimensional dynamical systems is hence used to analyze the global dynamics and long-term behavior of these systems. 
The classical approach to this problem follows several steps. 
\emph{First}, existence, uniqueness, and regularity of solutions are established for all positive time in appropriately chosen problem-dependent function spaces, and the well-posedness of the problem is confirmed.
\emph{Second}, a semidynamical framework is constructed over a positively invariant complete normed space---the phase space for the evolution of the solutions---and is shown to possess bounded absorbing sets. 
Asymptotic compactness of the semigroup of solution operators is then ensured to guarantee the existence of a global attractor, which is a compact strictly invariant attracting set and contains all the information regarding the asymptotic behavior of the model.  
\emph{Third}, the Hausdorff or fractal dimension of the global attractor is estimated to show that the attractor is finite dimensional, so that the asymptotic dynamics of the system is determined by a finite number of degrees of freedom.
\emph{Fourth}, the existence of an inertial manifold is established, which is a smooth finite-dimensional invariant manifold containing the global attractor. 
Consequently, the dynamics on the attractor can be presented by a finite set of ordinary differential equations and further characterized to give the overall picture of the long-term behavior of the system
\cite{Guckenheimer:AMS:1989, Temam:InfiniteDimensional:1997, Chepyzhov:Attractors:2002, Robinson:InfiniteDimensional:2001, Hale:AsymptoticBehavior:1988}.

In this paper, we investigate the mean field model proposed in \cite{Liley:Network:2002} for understanding the electrical activity in the neocortex as observed in the electroencephalogram (EEG). 
This model, which is comprised of a system of coupled ordinary and partial differential equations in a two-dimensional space, has been widely used in the literature to study the
alpha- and gamma-band rhythmic activity in the cortex \cite{Bojak:Neurocomputing:2004, Bojak:Neurocomputing:2007},
phase transition and burst suppression in cortical neurons during general anesthesia \cite{Liley:Frontiers:2013, Bojak:Frontiers:2015, SteynRoss:ProBiophysics:2004}, 
the effect of anesthetic drugs on the EEG \cite{Bojak:Frontiers:2013, Foster:Frontiers:2008},
and epileptic seizures \cite{Liley:ClinicNeurophys:2005,Kramer:PNAS:2012,  Kramer:JRCInterface:2005, Kramer:JCN:2006}.
Open-source tools for numerical implementation of the model and computation of equilibria and time-periodic solutions are developed in \cite{Green:JCS:2014}.
Complexity of the dynamics of the model, including periodic and pseudo-periodic solutions, chaotic behavior, multistability, and bifurcation are studied in \cite{Frascoli:ProcSPIE:2008, Frascoli:PhysicaD:2011, Dafilis:Chaos:2001, Dafilis:Chaos:2013, Dafilis:JMN:2015, VanVeen:EPJST:2014, VanVeen:PhysRevLett:2006}.

The above results, however, are mainly computational or use approximate versions of the model. 
A rigorous analysis of the dynamics of the model in an infinite-dimensional dynamical system framework as outlined above is not available in the literature. 
In particular, the basic problems of well-posedness of the initial-boundary value problem associated with the model and regularity of the solutions remain uninvestigated. 
It is not known under what conditions, if any, the components of the solutions of the model that are associated with nonnegative biophysical quantities remain nonnegative for all time. 
The solutions that take negative values for such quantities---even for a small interval of time in distant future---cannot represent a biophysically plausible dynamics of the electrical activity in the neocortex. 

The aim of this paper is to study the global dynamics of the mean field model discussed above, to ensure its biophysical plausibility, and to provide the basic analytical results required for characterization of the long-term dynamics of the model. 
Specifically, we follow the first two steps of the classical analysis approach to investigate the problem of existence or nonexistence of a global attractor.   

This paper is organized as follows. 
In Section \ref{sec:Notation}, we introduce notation and recall key definitions that are necessary for developing the results in this paper. 
In Section \ref{sec:ModelDiscription}, we give a description of the anatomical structure of the neocortex and the physiological interactions that underly the construction of the model. 
Moreover, we present the mathematical structure of the model as a system of coupled ordinary-partial differential equations with initial values and periodic boundary conditions. 
In Section \ref{sec:ExistenceUniqueness}, following the first step of the classical analysis approach,
we prove the existence and uniqueness of weak and strong solutions for the proposed initial value problem and analyze the regularity of these solutions. 

As in the second step of the classical analysis approach, in Section \ref{sec:SemidynamicalSystems} we define semigroups of weak and strong solution operators and show their continuity properties. 
Moreover, we establish sufficient conditions on the phase spaces that ensure biophysical plausibility of the evolution of the solutions under the associated semidynamical systems. 
In Section \ref{sec:AbsorbingSets}, we show that the semigroups of solution operators possess bounded absorbing sets for all possible values of the biophysical parameters of the model. 
In Section \ref{sec:Attractor}, we discuss challenges towards establishing a global attractor for the model, and in particular, we show that there exist sets of values for the biophysical parameters of the model such that the associated semigroups of solution operators do not possess a compact global attractor. 
We conclude the paper in Section \ref{sec:Conclusion} with a discussion on the results developed in the paper and their application to computational analysis of the model.          

\section{Notation and Preliminaries} \label{sec:Notation}
The notation used in this paper is fairly standard. Specifically, $\bbR^n$ denotes the $n$-dimensional real Euclidean space and $\bbR^{m \times n}$ denotes the space of real $m \times n$ matrices. 
A point $x \in \bbR^n$ is presented by the $n$-tuple $x=(x_1, \dots, x_n)$ or, 
when it appears in matrix operations, by the column vector 
$x= \left[ 
\begin{array}{ccc}
x_1 & \cdots & x_n
\end{array}
\right]^{\rT}$,
where $(\cdot)^{\rT}$ denotes transpose.
The nonnegative cone $\lbrace x\in \bbR^n : x_j \geq 0 \text{ for } j=1,\dots,n  \rbrace$ is denoted by $\bbR_+^n$.
A sequence of points in $\bbR^n$ is denoted by $\seq{x^{(l)}}_{l=1}^{\infty}$,
with the $j$th component of $x^{(l)}$ denoted by $x_j^{(l)}$.
Moreover, the trace of a square matrix $A \in \bbR^{n \times n}$ is denoted by $\trace A$ 
and a block-diagonal matrix $D$ with $k$ blocks $D_1, \dots, D_k$ is denoted by 
$\diag (D_1, \dots, D_k)$.
For $x, y \in \bbR^n$, we write $x \geq y$ to denote component-wise inequality, that is,  $x_j \geq y_j$, $j=1,\dots, n$. 
For $A, B \in \bbR^{n \times n}$ we write $A\geq B$ to denote $A-B$ is positive semidefinite. 
Finally, we denote by $0_{n \times n}$ and $I_{n \times n}$ the zero and identity matrices in $\bbR^{n \times n}$, respectively. 
We write $I$ for the identity operator in other vector spaces. 

For an inner product space $\cU$, we denote the associated inner product by $\inner{\cU}{\cdot}{\cdot}$ and the norm generated by the inner product by $\norm{\cU}{\cdot}$. 
For a Hilbert space $\cU$, we denote the pairing of $\cU$ with its dual space $\cU^{\ast}$ by $\pair{\cU}{\cdot}{\cdot}$.
In particular, for $\cU = \bbR^n$ we write $\inner{\bbR^n}{\cdot}{\cdot}$ and  $\norm{\bbR^n}{\cdot}$ for the standard inner product and the Euclidean norm, respectively.
Similarly, for $\cU = \bbR^{m \times n}$ we write $\inner{\bbR^{m \times n}}{\cdot}{\cdot}$ for the standard inner product and $\norm{\bbR^{m \times n}}{\cdot}$ 
for the associated inner product norm.  
Moreover, we denote the vector 1-, 2-, and $\infty$-norms in $\bbR^n$ by $\norm{1}{\cdot}$, $\norm{2}{\cdot}=\norm{\bbR^n}{\cdot}$, and $\norm{\infty}{\cdot}$, respectively. 
The matrix 1-, 2-, and $\infty$-norms in $\bbR^{m \times n}$ induced, respectively, by the vector 1-, 2-, and $\infty$-norms in $\bbR^n$ are denoted by $\norm{1}{\cdot}$, $\norm{2}{\cdot}$ and $\norm{\infty}{\cdot}$. 

Let $\Omega$ be an open subset of $\bbR^n$ denoting the \emph{space} domain
of a given dynamical system, with $x \in \Omega$ denoting a \emph{spatial} point in $\Omega$. 
The \emph{time} domain of the system  is given
by the closed interval $[0,T] \subset \bbR$, $T>0$, with the \emph{temporal} point $t$. 
For a function $u: [0,T] \rightarrow \bbR $, the $k$th-order total derivative with respect to $t$ at $t_0$ is denoted by $\rd_t^k u(t_0)$.
For $k=1$, we write $\rd_t u(t_0)$.
For a function $u(x, t): \Omega \times [0,T] \rightarrow \bbR$, the $k$th-order partial derivative with respect to $t$ at $(x_0,t_0)$ is denoted by $\dt^k u(x_0,t_0)$ 
and the $k$th-order partial derivative with respect to $x_j$ at $(x_0,t_0)$ is denoted by $\partial_{x_j}^k u(x_0,t_0)$, $j=1,\dots,n$.
For $k=1$, we write $\dt u(x_0,t_0)$ and $\partial_{x_j} u(x_0,t_0)$. 
The gradient of $u$ in $\Omega$ is denoted by $\dx u$ and is given by $\dx u := (\partial_{x_1} u, \dots, \partial_{x_n} u)$.
The Laplacian of  $u$ in $\Omega$ is denoted by $\Delta u$ and is given by $\Delta u := (\partial_{x_1}^2 + \dots + \partial_{x_n}^2)$.
For a vector-valued function $u(x, t): \Omega \times [0,T] \rightarrow \bbR^m$ we
interpret $u(x, t)$ as the $m$-tuple $u(x, t) = (u_1(x, t), \dots, u_m(x, t))$,
where each component $u_j(x,t)$, $j=1,\dots,m$, is a scalar-valued function on $\Omega \times [0,T]$.
In this case, $\dx u(x,t) \in \bbR^{m \times n}$ is the gradient of $u$ and the vector Laplacian $\Delta u$ is given by $\Delta u(x,t) := (\Delta u_1(x,t), \dots, \Delta u_m(x,t)) \in \bbR^m$, assuming
Cartesian coordinates. 

For every integer $k \geq 0$, the space of $k$-times continuously differentiable real-valued functions on $\Omega$ is denoted by $C^k(\Omega)$. 
The space $C^k(\overline{\Omega})$ consists of all functions in $C^k(\Omega)$ that, together with all of their  partial derivatives up to the order $k$, are uniformly continuous in bounded subsets of $\Omega$. 
Moreover, for $0 <\lambda \leq 1$, the H\"{o}lder space $C^{k,\lambda}(\overline{\Omega})$ is a subspace of $C^k(\overline{\Omega})$ consisting of functions whose partial derivatives of order $k$ are  H\"{o}lder continuous with exponent $\lambda$; 
see \cite[Sec. 1.18]{Ciarlet:FunctionalAnalysis:2013} for details.
We use $C_{\rc}^{\infty} (\Omega)$ to denote the space of infinitely differentiable real-valued functions with compact support in $\Omega$.  
Moreover, we denote by $L_{\rm loc}^1 (\Omega)$ the space of locally integrable real-valued functions on $\Omega$. 
Then, for every function $u \in L_{\rm loc}^1 (\Omega)$ and any multi index $\alpha$ with $|\alpha| \geq 1$,
the \emph{weak partial derivative} of $u$ in $L_{\rm loc}^1 (\Omega)$, of order $|\alpha|$, is defined by the 
distribution $u^{\alpha}$
that stisfies
\begin{equation*}
	\int_{\Omega} u^{\alpha} \phi \, \rd x = (-1)^{|\alpha|} \int_{\Omega} u \partial^{\alpha} \phi \, \rd x \quad 
	\text{for all } \phi \in C_{\rc}^{\infty} (\Omega),  
\end{equation*} 
where $\rd x = \rd x_1 \cdots \rd x_n$ is the Lebesgue measure on $\bbR^n$; see \cite[Sec. 6.3]{Ciarlet:FunctionalAnalysis:2013} for details. 
With a minor abuse of notation, we use $\dt^k$ and $\dx^k$ to denote the $k$th-order weak---as well as classical---partial derivatives with respect to $t$ and $x$, respectively. 
The distinction will be clear from the context, or will otherwise  be explicitly specified.  

The Hilbert space of vector-valued Lebesgue measurable functions $u: \Omega \rightarrow \bbR^m$ with
finite $L^2$-norm is denoted by $L^2(\Omega; \bbR^m)$, with associated inner product and norm given by
\begin{align*}
	\inner{L^2(\Omega; \bbR^m)}{u}{v} := \int_{\Omega} \inner{\bbR^m}{u(x)}{v(x)} \rd x, \quad
	\norm{L^2(\Omega; \bbR^m)}{u} := \left[ \int_{\Omega} \norm{\bbR^m}{u(x)}^2 \rd x \right]^{\frac{1}{2}}.
\end{align*}
The Banach space of vector-valued Lebesgue measurable functions $u: \Omega \rightarrow \bbR^m$ 	with
finite $L^{\infty}$-norm is denoted by $L^{\infty}(\Omega; \bbR^m)$, with the norm
\begin{align*}
	\norm{L^{\infty}(\Omega; \bbR^m)}{u} := \esssup_{x \in \Omega} \norm{\infty}{u(x)}.
\end{align*}
The Sobolev space of vector-valued functions $u \in L^p(\Omega; \bbR^m)$ whose all $l$th-order weak derivatives $\dx^l u$, $l\leq k$,
exist and belong to $L^p(\Omega; \bbR^{m\times n^l})$ is denoted by $W^{k,p}(\Omega; \bbR^m)$.
When $p=2$, the Sobolev spaces $W^{k,2}(\Omega; \bbR^m)$ are Hilbert spaces for all $k \in [0, \infty)$, and are denoted by $H^k(\Omega; \bbR^m):=W^{k,2}(\Omega; \bbR^m)$.
Specifically, $H^0(\Omega; \bbR^m) = L^2(\Omega; \bbR^m)$, and $H^1(\Omega; \bbR^m)$ is a Hilbert space with the inner product 
\begin{align*}
	\inner{H^1(\Omega; \bbR^m)}{u}{v} &:= \inner{L^2(\Omega; \bbR^m)}{u}{v} + \inner{L^2(\Omega; \bbR^{m \times n})}{\dx u}{\dx v}. 
\end{align*}
Moreover, $H^2(\Omega; \bbR^m)$ is a Hilbert space with the  inner product 
\begin{align*}
	\inner{H^2(\Omega; \bbR^m)}{u}{v} &:= \inner{L^2(\Omega; \bbR^m)}{u}{v} + \inner{L^2(\Omega; \bbR^{m \times n})}{\dx u}{\dx v} + \inner{L^2(\Omega; \bbR^{m \times n^2})}{\dx^2 u}{\dx^2 v}.
\end{align*}

Let $\Omega = (0,\omega_1) \times \cdots \times (0,\omega_n)$, where $\omega_j >0$, $j=1, \dots, n$, be an open rectangle in $\bbR^n$. A function $u: \bbR^n \rightarrow \bbR$ is called \emph{$\Omega$-periodic} if it is periodic in each direction, that is,
\begin{equation*}
	u(x + \omega_j e_j) = u(x), \quad j=1, \dots, n, \quad x \in \bbR^n,
\end{equation*}
where $e_j$ is the unit vector in the $j$th direction. 
Define the space $C_{\rm per}^{\infty}(\Omega)$ as the restriction to $\Omega$ of the space of infinitely differentiable $\Omega$-periodic functions. 
Then, the Sobolev space $H_{\rm per}^k(\Omega)$, $k \geq 0$, is defined by the completion of $C_{\rm per}^{\infty}(\Omega)$ in $H^k(\Omega)$; see \cite[Definition 5.37]{Robinson:InfiniteDimensional:2001} or, for an equivalent definition, \cite[p. 50]{Temam:InfiniteDimensional:1997} .   
A vector-valued function  $u: \bbR^n \rightarrow \bbR^m$ is $\Omega$-periodic if each of its components $u_j: \bbR^n \rightarrow \bbR$, $j=1, \dots, m$, is $\Omega$-periodic. The spaces $C_{\rm per}^{\infty}(\Omega; \bbR^m)$ and $H_{\rm per}^k(\Omega; \bbR^m)$ are then defined accordingly. It follows from Green's formula that
\begin{align} \label{eq:DeltaFormulas}
	\inner{L_{\rm per}^2(\Omega;\bbR^m)}{-\Delta u}{v} &= \inner{L_{\rm per}^2(\Omega; \bbR^{m \times n})}{\dx u}{\dx v},\\  
	\inner{L_{\rm per}^2(\Omega;\bbR^m)}{(-\Delta+I) u}{v} &= \inner{H_{\rm per}^1(\Omega;\bbR^m)}{u}{v}, \nonumber \\
	\inner{L_{\rm per}^2(\Omega;\bbR^m)}{-\Delta u}{(-\Delta+I) u} &= \norm{H_{\rm per}^2(\Omega;\bbR^m)}{u}^2
	-\norm{L_{\rm per}^2(\Omega;\bbR^m)}{u}^2, \nonumber \\
	\norm{L_{\rm per}^2(\Omega;\bbR^m)}{(-\Delta+I)u}^2 &= \norm{H_{\rm per}^2(\Omega;\bbR^m)}{u}^2 + \norm{L_{\rm per}^2(\Omega; \bbR^{m \times n})}{\dx u}^2  \nonumber\\
	&=\norm{H_{\rm per}^1(\Omega;\bbR^m)}{u}^2 + \norm{H_{\rm per}^1(\Omega; \bbR^{m \times n})}{\dx u}^2. \nonumber
\end{align}

In this paper, we interchangeably view the function $u(x,t)$, $x \in \Omega$, $t \in [0,T]$, as a composite function of $x$ and $t$, as well as a mapping $u$ of $t$ to a function of $x$, that is,
\begin{equation*}
	[u(t)](x) := u(x,t), \quad x \in \Omega, \quad t \in [0,T].
\end{equation*}    
With a minor abuse of notation, the same symbol is used to denote both the original form of the function and the mapping.
The distinction becomes evident in the way we define the space of such mappings or, equivalently, Banach space-valued functions; see for example \cite[Appx. E.5]{Evans:PDE:2010}. 
For a Banach space $\cU$, the space $L^2(0,T;\cU)$ is composed of all strongly measurable Banach space-valued functions $u: [0,T] \rightarrow \cU$ 	with
the finite $L^2$-norm defined by
\begin{equation*}
	\norm{L^2(0,T;\cU)}{u} := \left[  \int_0^T \norm{\cU}{u(t)}^2 \rd t \right]^{\frac{1}{2}}.
\end{equation*}  
The space $C^0([0,T];\cU)$ is composed of all continuous Banach space-valued functions $u: [0,T] \rightarrow \cU$ 	with the finite uniform norm defined by
\begin{equation*}
	\norm{C^0([0,T];\cU)}{u} := \max_{t\in[0,T]} \norm{\cU}{u(t)}.
\end{equation*}  
Accordingly, the spaces $C^k([0,T];\cU)$ and $C^{k,\lambda}([0,T];\cU)$, $k\geq 0$, $0 < \lambda \leq 1$, are defined as the space of $k$-times continuously differentiable Banach space-valued functions and its H\"{o}lder continuous subspace. 
The Sobolev spaces $H^k(0,T;\cU)$, $k \geq 0$, are composed of all functions $u \in L^2(0,T;\cU)$ whose $l$th-order weak derivatives $\rd_t^l u$ exist for $l \leq k$ and belong to $L^2(0,T;\cU)$. 
In particular, for $k=1$ we have
\begin{equation*}
	\norm{H^1(0,T;\cU)}{u} := \left[ \int_0^T \left( \norm{\cU}{u(t)}^2 + \norm{\cU}{\rd_t u(t)}^2 \right) \, \rd t \right]^{\frac{1}{2}}.
\end{equation*}
For further details on these spaces; see \cite[Sec. 5.9.2]{Evans:PDE:2010} and \cite[Sec. 7.1]{Robinson:InfiniteDimensional:2001}. 

When $P:\cU \rightarrow \cY$ is a mapping between the Banach spaces $\cU$ and $\cY$, we denote the $k$th order Fr\'{e}chet derivative of $P$ at $u_0$ by $\rd_u P(u_0)$. 
The space $C^k(\cU;\cY)$ is then composed of all $k$-times continuously differentiable mappings from $\cU$ into $\cY$.
For a mapping $P:\cU_1 \times \dots \times \cU_m \rightarrow \cY$, where $\cY$ and $\cU_j$, $j=1,\dots m$, are Banach spaces, $\partial_{u_j} P(u_0)$ is the $j$th partial Fr\'{e}chet derivative of $P$ at  $u_0 = ({u_0}_1,\dots,{u_0}_m)$. 
The gradient of $P$ at $u_0$ is then written as $\partial_u P(u_0)$;   
see \cite[Sec. 7.1]{Ciarlet:FunctionalAnalysis:2013} for details.

Finally, we denote the symmetric difference of two sets $\sX$ and $\sY$ by $\sX \bigtriangleup \sY$. In a topological space $\cX$, we denote the closure of a set $\sX \subset \cX$ by $\overline{\sX}$, its interior by $\sX^\circ$, and its boundary by $\partial \sX$. The characteristic function of $\sX$ is denoted by $\chi(\sX)$.
When $\cX$ is a measure space, $|\sX|$ denotes the measure of the set $\sX \subset \cX$. 
For normed vector spaces $\cX$ and $\cY$, we write $\cX \hookrightarrow \cY$ for continuous embedding of $\cX$ in $\cY$, and $\cX \Subset \cY$ for compact embedding of $\cX$ in $\cY$; see \cite[Sec. 6.6]{Ciarlet:FunctionalAnalysis:2013} for details.
When $\cX$ is a metric space and the topology on $\cX$ is induced by the given metric, $B(x,R)$ denotes the open ball centered at $x \in \cX$ with radius $R>0$, which is a basis element for the topology.
For every bounded measurable set in $\cX$ and, in particular for $B(x,R)$, we denote by $\dashint_{B(x,R)}$ the averaging operator over $B(x,R)$, that is, $\dashint_{B(x,R)} := \frac{1}{|B(x,R)|}\int_{B(x,R)}$.

\section{Model Description} \label{sec:ModelDiscription}

\begin{figure}
	\centering
	\includegraphics[width=1\linewidth]{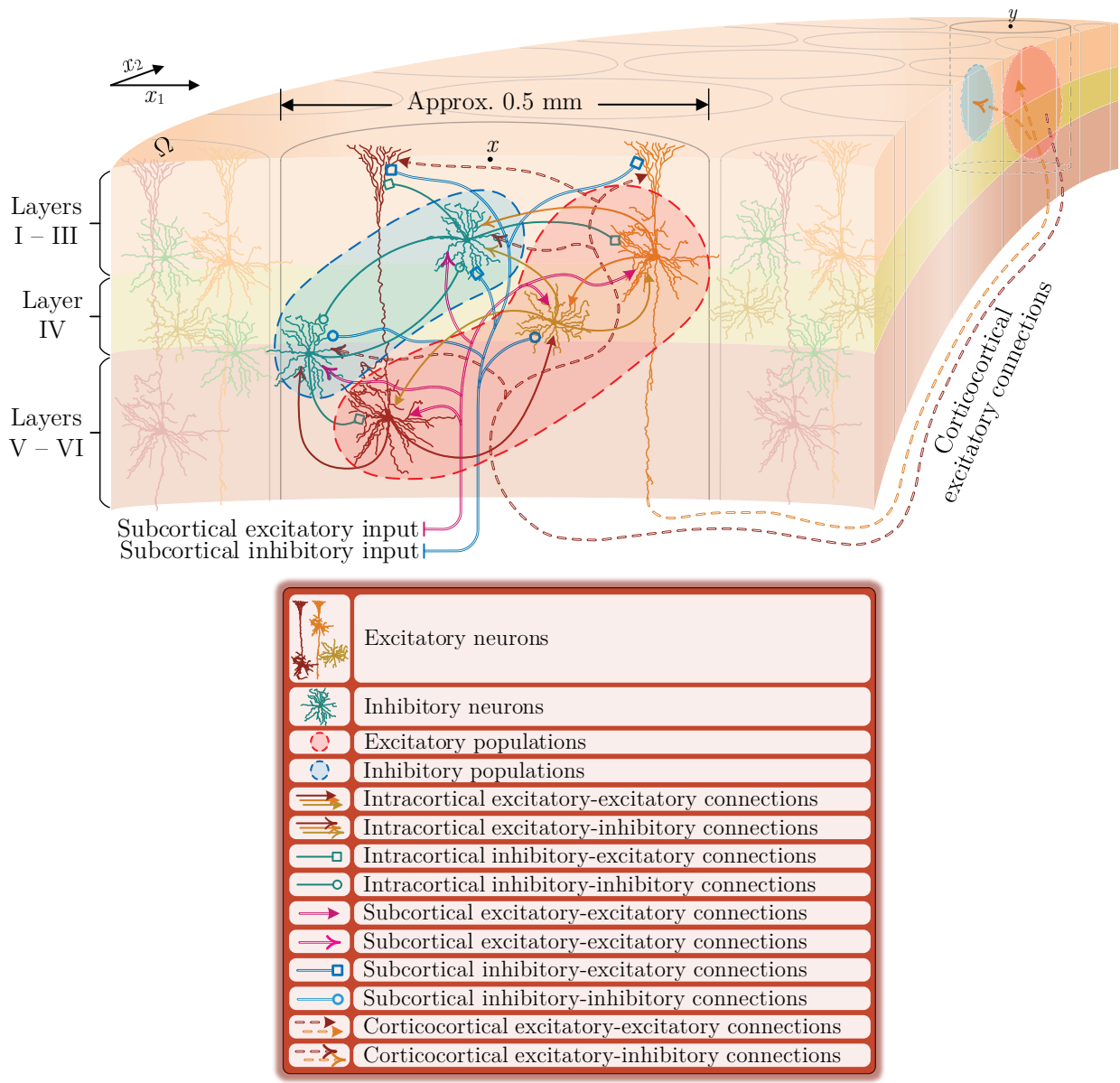}
	\caption[Cortex Structure]{Schematic of the structure of the neocortex with intracortical and corticocortical connections.}
	\label{fig:CortexStructure}
\end{figure}

The neocortex has a layered columnar structure consisting mostly of six distinctive layers. 
Neurons in the neocortex are organized in vertical columns, usually referred to as \emph{cortical columns} or \emph{macrocolumns}, 
which are a fraction of a millimeter wide and traverse all the layers of the neocortex from the white matter to the pial surface \cite{Kandel:NeuralScience:2013, Mountcastle:Brain:1997, Horton:PTRSL:2005}. 
Depending on their type of action, neurons are mainly classified as \emph{excitatory} or \emph{inhibitory}, 
wherein this distinction depends on whether they increase the firing rate in the destination neurons they are communicating with, or they essentially suppress them. 
Inhibitory neurons are located in all layers and usually have axons that remain within the same area where their cell body resides, and hence, they have a local range of action. 
Layers III, V, and VI contain pyramidal excitatory neurons whose axons can provide long-range communication (projection) throughout the neocortex. 
Layer IV contains primarily star-shaped excitatory interneurons that receive sensory inputs from the thalamus. 
Figure \ref{fig:CortexStructure} shows a schematic of the structure of the neocortex, including the intracortical and corticocortical neuronal connections; 
see \cite[Ch. 15]{Kandel:NeuralScience:2013} for further details.         

On a local scale, within a cortical column, neurons are densely interconnected and involve all types of feedforward  and feedback intracortical connections. 
Such a dense and relatively homogeneous local structure of the neocortex suggests modeling a local population of functionally similar neurons by a single \emph{space-averaged neuron}, which preserves enough physiological information to understand the temporal patterns observed in spatially smoothed (averaged) EEG signals without creating excessive theoretical complicacies in the mathematical analysis of the model. 
On a global scale, in the exclusively excitatory corticocortical communication throughout the neocortex, two major patterns of connectivity are observed. Namely, a homogeneous, symmetrical, and translation invariant pattern of connections, versus a heterogeneous, patchy, and asymmetrical distribution of connections. 
For modeling simplicity and due to unavailability of detailed anatomical data, in the model that we investigate in this paper the corticocortical connectivity is assumed to be isotropic, homogeneous, symmetric, and translation invariant \cite{Liley:Network:2002}.

To establish the mathematical framework of the model, let $\Omega = (0, \omega) \times (0, \omega)$, $\omega>0 $, be an open rectangle in $\bbR^2$ that defines 
the domain of the neocortex. Each point $x=(x_1,x_2) \in \Omega$ indicates the location of a local network---possibly representing a cortical column---modeled by a space-averaged excitatory neuron and a space-averaged inhibitory neuron.
Let $\rE$ denote a population of excitatory neurons and $\rI$ denote a population of inhibitory neurons. 
For  $x \in \Omega$, $t \in [0, T]$, $T>0$, and $\rX, \rY \in \{ \rE, \rI \}$, we denote by 
$v_{\rX}(x,t)$, measured in $\text{mV}$, the spatially mean soma membrane potential of a population of type $\rX$ centered at $x$. 
Moreover, we denote by $i_{\rX \rY}(x,t)$, measured in $\text{mV}$, the spatially mean postsynaptic activation of synapses of a population of type $\rX$ centered at $x$, onto a population of type $\rY$ centered at the same point $x$. 
In addition, we denote by $w_{\rE \rX}(x,t)$, measured in $\rs^{-1}$, the mean rate of corticocortical excitatory input pulses from the entire domain of the neocortex to a population of type $\rX$ centered at $x$. 
Finally, we denote by  $g_{\rX \rY}(x,t)$, measured in $\rs^{-1}$, the mean rate of subcortical input pulses of type $\rX$ to a population of type $\rY$ centered at $x$. 
Note that, by definition, $ i_{\rX \rY}(x,t)$, $w_{\rE \rX}(x,t)$, and $g_{\rX \rY}(x,t)$ are nonnegative quantities.

Then, as developed in \cite{Liley:Network:2002}, the system of coupled ordinary and partial differential equations 
\begin{align}     \label{eq:Model}
	(\tau_{\rE} \dt + 1) v_{\rE}(x,t) &= \frac{\rV_{\rE \rE}-v_{\rE}(x,t)}{|\rV_{\rE \rE}|} i_{\rE \rE}(x,t)
	+ \frac{\rV_{\rI \rE}-v_{\rE}(x,t)}{|\rV_{\rI \rE}|} i_{\rI \rE}(x,t), \\
	(\tau_{\rI} \dt + 1) v_{\rI}(x,t) &= \frac{\rV_{\rE \rI}-v_{\rI}(x,t)}{|\rV_{\rE \rI}|} i_{\rE \rI}(x,t)
	+ \frac{\rV_{\rI \rI}-v_{\rI}(x,t)}{|\rV_{\rI \rI}|} i_{\rI \rI}(x,t),	\nonumber\\
	(\dt +\gamma_{\rE \rE} )^2 i_{\rE \rE}(x,t) &= e \Upsilon_{\rE \rE} \gamma_{\rE \rE}
	\left[ \rN_{\rE \rE} f_{\rE}\big( v_{\rE}(x,t) \big) + w_{\rE \rE}(x,t) + g_{\rE \rE}(x,t) \right],	\nonumber\\
	(\dt +\gamma_{\rE \rI} )^2 i_{\rE \rI}(x,t) &= e \Upsilon_{\rE \rI} \gamma_{\rE \rI}
	\left[ \rN_{\rE \rI} f_{\rE}\big( v_{\rE}(x,t) \big) + w_{\rE \rI}(x,t) + g_{\rE \rI}(x,t) \right],	\nonumber\\
	(\dt +\gamma_{\rI \rE} )^2 i_{\rI \rE}(x,t) &= e \Upsilon_{\rI \rE} \gamma_{\rI \rE}
	\left[ \rN_{\rI \rE} f_{\rI}\big( v_{\rI}(x,t) \big) + g_{\rI \rE}(x,t) \right],	\nonumber\\
	(\dt + \gamma_{\rI \rI} )^2 i_{\rI \rI}(x,t) &= e \Upsilon_{\rI \rI} \gamma_{\rI \rI}
	\left[ \rN_{\rI \rI} f_{\rI}\big( v_{\rI}(x,t) \big) + g_{\rI \rI}(x,t) \right],	 \nonumber  \nonumber\\
	\big[ ( \dt + \nu \Lambda_{\rE \rE} )^2 - \tfrac{3}{2} \nu^2 \Delta  \big]  w_{\rE \rE}(x,t) &=
	\nu^2 \Lambda_{\rE \rE}^2 \rM_{\rE \rE} f_{\rE}\big( v_{\rE}(x,t) \big),	\nonumber\\
	\big[ ( \dt + \nu \Lambda_{\rE \rI} )^2  - \tfrac{3}{2} \nu^2 \Delta  \big]  w_{\rE \rI}(x,t) &=
	\nu^2 \Lambda_{\rE \rI}^2 \rM_{\rE \rI} f_{\rE}\big( v_{\rE}(x,t) \big), 
	\hspace{1.65cm}(x,t) \in \Omega \times (0,T],	\nonumber
\end{align}
with periodic boundary conditions provides a mean field model for the electrocortical activity in the neocortex. 
Here, $e$ is the Napier constant and    
$f_{\rX}(\cdot)$ is the mean firing rate function of a population of type $\rX$ and is given by 
\begin{equation} \label{eq:FiringRateFunction}
	f_{\rX}\big( v_{\rX}(x,t) \big) := \frac{\rF_{{\rX}}}
	{1+\exp\left( -\sqrt{2}\, \dfrac{v_{\rX}(x,t) - \mu_{\rX} }{\sigma_{\rX}} \right)},
	\quad \rX \in \{ \rE, \rI \}.
\end{equation}
The definition of the biophysical parameters of the model and the ranges of the values they may take are given in Table \ref{tb:Parameters}. 
For the range of values given in Table \ref{tb:Parameters}, we have 
$|\rV_{\rE \rE}| = \rV_{\rE \rE}$, $|\rV_{\rE \rI}| = \rV_{\rE \rI}$, $|\rV_{\rI \rE}|= -\rV_{\rI \rE}$ , and $|\rV_{\rI \rI}| = -\rV_{\rI \rI}$, 
which we use to simplify \eqref{eq:Model}.  
Note that other than notational changes to the original equations given in \cite{Liley:Network:2002}, we have changed the reference of electric potential to the \emph{resting potential} to avoid the constant terms that would otherwise appear in \eqref{eq:Model}.
Figure \ref{fig:CorticalInputs} shows a schematic of intracortical, corticocortical, and subcortical inputs to two local networks located at points $x$ and $y$, along with their contribution to the global corticocortical activation as modeled by \eqref{eq:Model}. 
The particular coupling between the equations of the model is depicted by the block diagram shown in Figure \ref{fig:BlockDiagram}.

\begin{table}[t]
	\vspace{0.4cm}
	\caption{ Definition and range of values for the biophysical parameters of the mean field model \eqref{eq:Model}. All electric potentials are given with respect to the mean resting soma membrane potential $v_{\rm rest} = -70 
		\text{ mV}$ \cite{Bojak:PhysRev:2005}. } \label{tb:Parameters}
	\begin{center} \small 
		\renewcommand{\arraystretch}{1.3}
		\begin{tabular}{|llll|}\hline
			\rowcolor{LightMaroon}
			\textbf{Parameter} & \textbf{Definition} & \textbf{Range} & \textbf{Unit} \\ \hline
			$\tau_{\rE}$ & Passive excitatory membrane decay time constant & $[0.005,0.15]$ & s\\ 
			$\tau_{\rI}$ & Passive inhibitory membrane decay time constant & $[0.005,0.15]$ & s\\ 
			$\rV_{\rE \rE}$, $\rV_{\rE \rI}$ & Mean excitatory Nernst potentials & $[50, 80]$ &  mV\\ 
			$\rV_{\rI \rE}$, $\rV_{\rI \rI}$ & Mean inhibitory Nernst potentials & $[-20, -5]$ &  mV\\ 
			$\gamma_{\rE \rE}$, $\gamma_{\rE \rI}$ & Excitatory postsynaptic potential rate constants & $[100, 1000]$ &  s$^{-1}$\\ 
			$\gamma_{\rI \rE}$, $\gamma_{\rI \rI}$ & Inhibitory postsynaptic potential rate constants & $[10, 500]$ &  s$^{-1}$\\ 
			$\Upsilon_{\rE \rE}$, $\Upsilon_{\rE \rI}$ & Amplitude of excitatory postsynaptic potentials & $[0.1, 2.0]$ &  mV\\ 
			$\Upsilon_{\rI \rE}$, $\Upsilon_{\rI \rI}$ & Amplitude of inhibitory postsynaptic potentials & $[0.1, 2.0]$ &  mV\\ 
			$\rN_{\rE \rE}$, $\rN_{\rE \rI}$ & Number of intracortical excitatory connections & $[2000, 5000]$ &  ---\\ 
			$\rN_{\rI \rE}$, $\rN_{\rI \rI}$ & Number of intracortical inhibitory connections & $[100, 1000]$ &  ---\\ 
			$\nu$ & Corticocortical conduction velocity & $[100,1000]$ & cm/s\\
			$\Lambda_{\rE \rE}$, $\Lambda_{\rE \rI}$ & Decay scale of corticocortical excitatory connectivities  & $[0.1,1.0]$ & cm$^{-1}$\\
			$\rM_{\rE \rE}$, $\rM_{\rE \rI}$ & Number of corticocortical excitatory connections & $[2000, 5000]$ &  ---\\ 
			$\rF_{{\rE}}$ & Maximum mean excitatory firing rate & $[50, 500]$ & s$^{-1}$\\
			$\rF_{{\rI}}$ & Maximum mean inhibitory firing rate & $[50, 500]$ & s$^{-1}$\\
			$\mu_{\rE}$ & Excitatory firing threshold potential & $[15, 30]$ & mV\\
			$\mu_{\rI}$ & Inhibitory firing threshold potential & $[15, 30]$ & mV\\
			$\sigma_{\rE}$ & Standard deviation of excitatory firing threshold potential & $[2, 7]$ & mV\\
			$\sigma_{\rI}$ & Standard deviation of inhibitory firing threshold potential & $[2, 7]$ & mV\\
			\hline
		\end{tabular}
	\end{center}
\end{table}

\begin{figure}[t]
	\centering
	\includegraphics[width=0.9\linewidth]{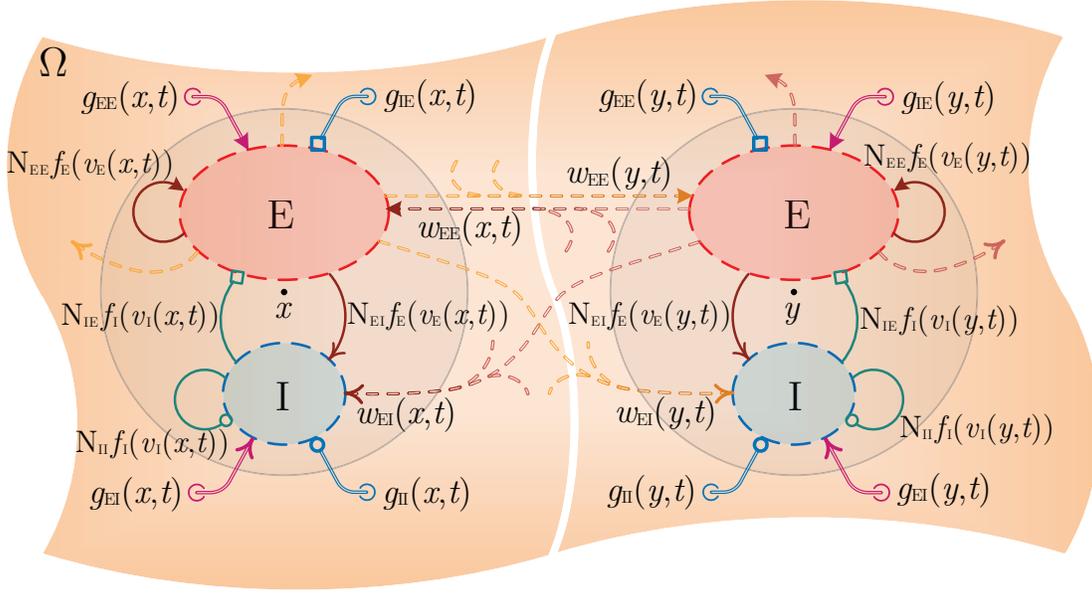}
	\caption[Cortical Inputs]{Cortical inputs to two local networks located at points $x$ and $y$ as modeled by \eqref{eq:Model}.}
	\label{fig:CorticalInputs}
\end{figure}

\begin{figure}[t]
	\centering \vspace{0.5cm}
	\includegraphics[width=1\linewidth]{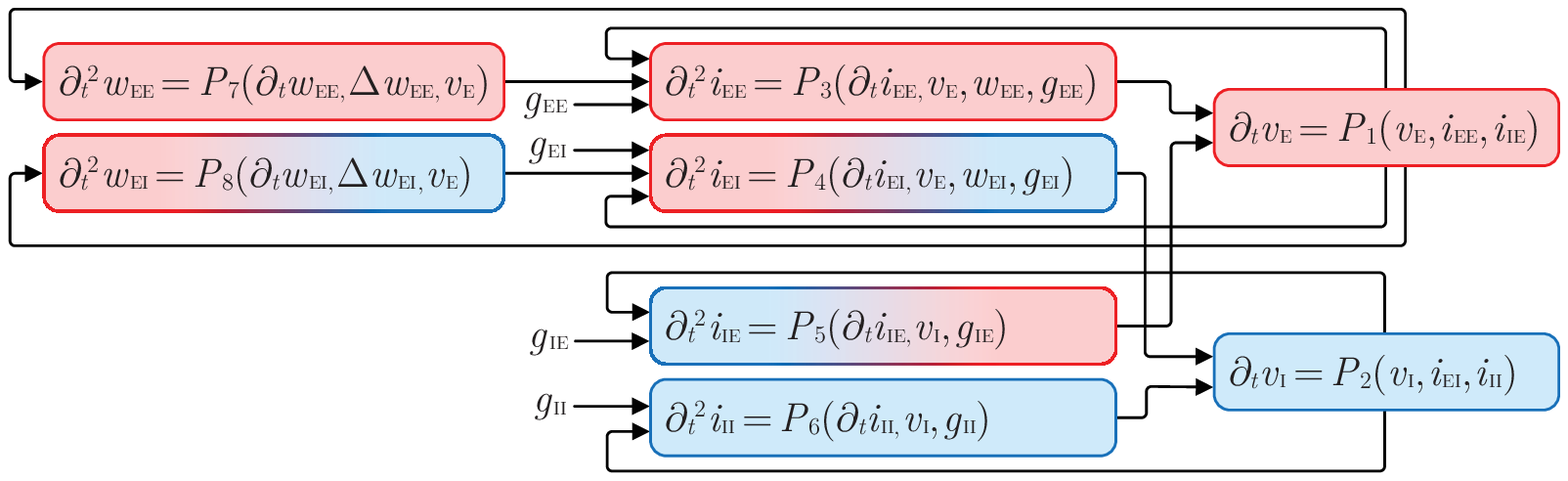}
	\caption[Block Diagram]{Block diagram of the mean field model \eqref{eq:Model}. The operators $P_1, \dots,P_8$ represent the eight equations in \eqref{eq:Model}, respectively.
		Similar to Figures \ref{fig:CortexStructure} and \ref{fig:CorticalInputs}, the blocks associated with excitatory populations are shown in red,
		and the blocks associated with inhibitory populations are shown in blue.}
	\label{fig:BlockDiagram}
\end{figure}

The first two equations in \eqref{eq:Model}, that is, the $v$-equations, model the dynamics of the resistive-capacitive membrane of the space-averaged neurons located at $x$.  
In the absence of postsynaptic $i$-inputs, the mean membrane potential decays exponentially to the resting potential.
The fractions appearing in the equations weight the postsynaptic inputs 
to incorporate the effect of transmembrane diffusive ion flows into the model.
Specifically, the depolarizing effect of excitatory inputs on the membrane is linearly decreased by the weights as
the membrane potential rises to the Nernst (reversal) potential.
When the membrane potential exceeds the Nernst potential, the effect is reversed and further excitation tends to
hyperpolarize the membrane.
The weights associated with the inhibitory postsynaptic inputs have opposite signs at the resting potential, and hence, they have an opposite reversal effect.    

The critically damped second order dynamics of the four $i$-equations in \eqref{eq:Model} generates
a synaptic $\alpha$-function---as in the classical dendritic cable theory---in response to an impulse.  
As shown in Figure \ref{fig:CorticalInputs}, these second order dynamical systems are driven by three different
sources of presynaptic spikes, namely, the inputs $\rN_{\rX \rY} f_{\rX}(v_{\rX})$ from local neuronal populations, the excitatory inputs $w_{\rE \rX}$ form corticocortical fibers, and the inputs $g_{\rX \rY}$ from 
subcortical regions.
As a result, these four equations generate the postsynaptic responses that modulate the polarization of the cell membranes according to the $v$-equations discussed before.	 

Unlike the conduction through short-range intracortical fibers, the conduction through long-range corticocortical fibers cannot be assumed to be instantaneous. 
The $w$-equations in \eqref{eq:Model} form a system of telegraph equations that effectively models the propagation of the excitatory axonal pulses through corticocortical fibers.
To derive these equations, it is assumed in \cite{Liley:Network:2002} that the strength of corticocortical connections onto a local population decays exponentially with distance, with the characteristic scale $\Lambda_{\rE \rX}$.
Moreover, it is assumed that the spatial distribution of connections is isotropic and homogeneous all over the neocortex.

In practical applications, the key variable in the model presented by \eqref{eq:Model} is the mean membrane potential of excitatory populations $v_{\rE}(x,t)$ 
that is presumed to be linearly proportional to EEG recordings from the scalp
\cite{Liley:Network:2002, Liley:Frontiers:2013}.
For further details of the model see \cite{Liley:Network:2002}, or the introductory sections of 
\cite{Bojak:Frontiers:2015, Liley:Frontiers:2013, Frascoli:PhysicaD:2011}.

Now, let 
\begin{align*}
	v(x,t) &:=\big(v_{\rE}(x,t), v_{\rI}(x,t) \big) \in \bbR^2,  \\
	i(x,t) &:=\big( i_{\rE \rE}(x,t), i_{\rE \rI}(x,t), i_{\rI \rE}(x,t), i_{\rI \rI}(x,t) \big) \in \bbR^4 , \\
	w(x,t) &:= \big( w_{\rE \rE}(x,t), w_{\rE \rI}(x,t) \big) \in \bbR^2, \\
	g(x,t) &:= \big( g_{\rE \rE}(x,t), g_{\rE \rI}(x,t), g_{\rI \rE}(x,t), g_{\rI \rI}(x,t) \big) \in \bbR^4,
\end{align*}
and note that \eqref{eq:Model} can be represented in vector form in $\Omega \times (0,T]$ as
\begin{align}
	\Phi \dt v + v -  J_1 i +  J_2 v i^{\rT} \Psi J_4 +  J_3 v i^{\rT} \Psi J_5 &= 0,   \label{eq:Voltage}\\
	\dt^2 i + 2 \Gamma \dt i + \Gamma^2 i 
	- e \Upsilon \Gamma J_6 w - e \Upsilon \Gamma \rN J_7 f(v) &= e \Upsilon\Gamma g, \label{eq:Current}\\
	\dt^2 w + 2 \nu \Lambda \dt w	- \tfrac{3}{2} \nu^2 \Delta w + \nu^2 \Lambda^2 w - \nu^2 \Lambda^2 \rM J_8 f\big(v \big) &=0, \label{eq:Wave}
\end{align}
where $v$, $i$, and $w$ are $\Omega$-periodic vector-valued functions with the initial values 
\begin{equation}
	v \big|_{t=0} = v_0, \quad i \big|_{t=0} = i_0, \quad (\dt i) \big|_{t=0} = i_0', \quad w \big|_{t=0} = w_0, \quad (\dt w) \big|_{t=0} = w_0', \label{eq:InitialValues}
\end{equation}
and
\begin{align} \label{eq:Parameters}
	\Phi &= \mathrm{diag}\big( \tau_{\rE}, \tau_{\rI} \big),	 \quad
	&\Psi &= \mathrm{diag}\big( \tfrac{1}{|\rV_{\rE \rE}|}, \tfrac{1}{|\rV_{\rE \rI}|},  \tfrac{1}{|\rV_{\rI \rE}|}, \tfrac{1}{|\rV_{\rI \rI}|} \big),  \\
	\Gamma &= \mathrm{diag} (\gamma_{\rE \rE}, \gamma_{\rE \rI}, \gamma_{\rI \rE}, \gamma_{\rI \rI}),  \quad
	&\Upsilon &= \mathrm{diag} (	\Upsilon_{\rE \rE}, \Upsilon_{\rE \rI}, \Upsilon_{\rI \rE}, \Upsilon_{\rI \rI}),	\nonumber \\		
	\rN &= \mathrm{diag} (\rN_{\rE \rE}, \rN_{\rE \rI}, \rN_{\rI \rE}, \rN_{\rI \rI}),		\quad
	&\rM &= \mathrm{diag} (\rM_{\rE \rE}, \rM_{\rE \rI}),	 	\nonumber \\
	\Lambda &= \mathrm{diag} ( \Lambda_{\rE \rE}, \Lambda_{\rE \rI} ),		\quad
	&J_1 &= \left[ \begin{array}{cc}
		I_{2 \times 2} & -I_{2\times 2}
	\end{array} \right], \nonumber \\	
	J_2 &= \mathrm{diag} (1, 0), \quad
	&J_3 &= \mathrm{diag} (0, 1), \nonumber \\
	J_4 &= \left[ \begin{array}{cccc}
		1 &	0 & 1 & 0 
	\end{array} \right]^{\rT},  	\quad
	&J_5 &= \left[ \begin{array}{cccc}
		0 &	1 & 0 & 1 
	\end{array} \right]^{\rT},  	 \nonumber \\
	J_6 &= \left[ \begin{array}{cccc}
		1 &	0 & 0 & 0 \\
		0 &	1 & 0 & 0
	\end{array} \right]^{\rT},  	\quad
	&J_7 &= \left[ \begin{array}{cccc}
		1 &	1 & 0 & 0 \\
		0 &	0 & 1 & 1
	\end{array} \right]^{\rT},  \nonumber \\
	J_8 &= \left[ \begin{array}{cc}
		1 &	0 \\
		1 &	0 
	\end{array} \right],			\quad
	&f (v) &= \left[ \begin{array}{c}
		f_{\rE} \big(\left[ \begin{array}{cc}
			1 & 0 \\
		\end{array} \right] v \big) \\ 
		f_{\rI} \big(\left[ \begin{array}{cc}
			0 & 1 \\
		\end{array} \right] v \big)
	\end{array} \right].			 \nonumber	
\end{align}

For simplicity of exposition, the dependence of the functions $v$, $i$, $w$, and $g$ on the arguments $(x,t)$ 
is not explicitly shown in \eqref{eq:Voltage}--\eqref{eq:Wave}. 
Note that \eqref{eq:Voltage} and \eqref{eq:Current}, which model the local dynamics of the neocortex, are essentially systems of ordinary differential equations. 
These equations do not possess any spatial smoothing component, and hence, their solutions are expected to evolve in less regular function spaces \cite{Sell:EvolutionaryEquations:2002, Marion:SIMA:1989}. 
The system of partial differential equations \eqref{eq:Wave} consists of two telegraph equations coupled indirectly through \eqref{eq:Voltage} and \eqref{eq:Current}; see Figure \ref{fig:BlockDiagram}. 

\begin{remark}[Variations in the parameters]: In the analysis that follows in the rest of the paper, we assume that all the parameters of the model are constant.
However, in practical applications, certain parameters may be considered to vary in time or space 
to model specific physiological situations in the brain. 
The variations can occur independently, or can be modeled using additional ODE's or PDE's coupled with the existing equations.
We give all the details of the results---even though some of which may be found fairly standard---along with a careful inclusion of all parameters. 
Therefore, in applications, it should be possible to easily observe where the parameters of interest appear in the analysis, and whether or not their particular variations can affect the validity of the results. 
\end{remark}

\section{Existence and Uniqueness of Solutions} \label{sec:ExistenceUniqueness}
In this section, we investigate the problem of existence, uniqueness, and regularity of solutions for \eqref{eq:Voltage}--\eqref{eq:Wave} with the initial values \eqref{eq:InitialValues} and periodic boundary conditions.
We set appropriate spaces of $\Omega$-periodic functions as the functional framework of the problem by which we include the boundary conditions in the solution spaces. 
We view $v(x,t)$, $i(x,t)$, and $w(x,t)$ as Banach space-valued functions and follow the standard technique of Galerkin approximations \cite{Robinson:InfiniteDimensional:2001, Evans:PDE:2010, Temam:InfiniteDimensional:1997} to construct weak and strong solutions in Theorems \ref{th:WeakExistenceUniqueness} and \ref{th:StrongExistenceUniqueness}.
The details of the proof of these results can be skipped in the case the reader is proficient in the analysis of Galerkin method. 

First, define the function spaces 
\begin{alignat}{3} \label{eq:FunctionSpaces}
	\cL^2_v &:= L_{\rm per}^2(\Omega; \bbR^2), &\quad
	\cL^2_i &:= L_{\rm per}^2(\Omega; \bbR^4), &\quad
	\cL^2_w &:= L_{\rm per}^2(\Omega; \bbR^2), \\
	\cL^{\infty}_v &:= L_{\rm per}^{\infty}(\Omega; \bbR^2), &\quad
	\cL^{\infty}_i &:= L_{\rm per}^{\infty}(\Omega; \bbR^4), &\quad 
	\cL^{\infty}_w &:= L_{\rm per}^{\infty}(\Omega; \bbR^2), \nonumber \\
	\cH_w^1 &:= H_{\rm per}^1(\Omega; \bbR^2), &\quad
	\cH_w^2 &:= H_{\rm per}^2(\Omega; \bbR^2), &&\nonumber \\
	\cL_{\partial w}^2 &:= L_{\rm per}^2(\Omega; \bbR^{2 \times 2}), &\quad
	\cH_{\partial w}^1 &:= H_{\rm per}^1(\Omega; \bbR^{2 \times 2}), &&\nonumber\\
	\cW_w^{1,\infty} &:= W_{\rm per}^{1, \infty}(\Omega; \bbR^2), &&&& \nonumber
\end{alignat} 
and denote by ${\cL^2_v}^{\ast}$, ${\cL^2_i}^{\ast}$, and ${\cH^1_w}^{\!\!\ast}$ the dual spaces of $\cL^2_v$, $\cL^2_i$, and $\cH_w^1$, respectively. 
Note that $\cL^2_v$ and $\cL^2_i$ are, respectively, isometrically isomorphic to ${\cL^2_v}^{\ast}$ and ${\cL^2_i}^{\ast}$ \cite[Th. 6.15]{Folland:RealAnalysis:1999}, which we denote by ${\cL^2_v}^{\ast} = \cL^2_v$ and ${\cL^2_i}^{\ast} = \cL^2_i$. 
By the Rellich-Kondrachov compact embedding theorems we have $\cH_w^1 \Subset \cL^2_w \hookrightarrow {\cH^1_w}^{\!\!\ast}$; see \cite[Th. 6.6-3]{Ciarlet:FunctionalAnalysis:2013} and \cite[Th. A.4]{Robinson:InfiniteDimensional:2001}. 
Moreover, there exists a dual orthogonal basis of $\cH_w^1$ and $\cL_w^2$ given by the following lemma. 
The proof of this lemma is fairly standard and follows the general results given in \cite[Sec. II.2.1]{Temam:InfiniteDimensional:1997}.

\begin{lemma}[Dual orthogonal basis] \label{lem:DualOrthogonalBasis}
	There exists an orthonormal basis of $\cL^2_w$ that is also an orthogonal basis of $\cH_w^1$, and can be constructed by the eigenfunctions of the linear operator $A:=-\Delta + I$.
\end{lemma}

Now, before proceeding to the main results of this section, we define the notions of \emph{weak} and \emph{strong} solutions of \eqref{eq:Voltage}--\eqref{eq:InitialValues} as used in this paper. 

\begin{definition}[Weak solution] \label{def:WeakSolution}
	A solution $(v,i,w)$ is called an $\Omega$-periodic \emph{weak solution} of the initial value problem \eqref{eq:Voltage}--\eqref{eq:InitialValues} if it solves the weak version of the problem wherein the partial differential equations are understood as equalities in the space of duals $L^2(0,T;{\cL^2_v}^{\ast} \times {\cL^2_i}^{\ast} \times {\cH^1_w}^{\!\!\ast})$. 
	That is, the functions 
	\begin{equation*}
		v \in L^2(0,T;\cL_v^2), \quad i \in L^2(0,T;\cL_i^2), \quad w \in L^2(0,T;\cH_w^1),
	\end{equation*}
	with 
	\begin{align*}
		\rd_t v &\in L^2(0,T;{\cL^2_v}^{\ast}), \quad 
		&\rd_t i &\in L^2(0,T;\cL_i^2), \quad 
		&\rd_t^2 i \in L^2(0,T;{\cL^2_i}^{\ast}), \\
		\rd_t w &\in L^2(0,T;\cL_w^2), \quad
		&\rd_t^2 w &\in L^2(0,T;{\cH^1_w}^{\!\!\ast}), &
	\end{align*}
	construct an $\Omega$-periodic weak solution for \eqref{eq:Voltage}--\eqref{eq:InitialValues} if for every $\ell_v \in \cL_v^2$,  $\ell_i \in \cL_i^2$, $h_w \in \cH_w^1$, and almost every $t \in [0,T]$, $T>0$,
	\begin{align}
		\pair{\cL_v^2}{\Phi \rd_t v}{\ell_v} + \inner{\cL_v^2}{v}{\ell_v} - \inner{\cL_v^2}{J_1 i}{\ell_v} + \inner{\cL_v^2}{ J_2 v i^{\rT} \Psi J_4 + J_3 v i^{\rT} \Psi J_5}{\ell_v} &= 0,  \label{eq:WeakVoltage}\\
		\pair{\cL_i^2}{\rd_t^2i}{\ell_i} + 2 \inner{\cL_i^2}{\Gamma \rd_t i}{\ell_i} + \inner{\cL_i^2}{\Gamma^2 i}{\ell_i} - e \inner{\cL_i^2}{\Upsilon \Gamma J_6 w}{\ell_i} \label{eq:WeakCurrent} \hspace{2.18cm} \\	 
		-	e \inner{\cL_i^2}{\Upsilon \Gamma \rN J_7 f( v)}{\ell_i}	&= e \inner{\cL_i^2}{\Upsilon\Gamma g}{\ell_i},  \nonumber \\
		\pair{\cH_w^1}{\rd_t^2 w}{h_w} + 2 \nu \inner{\cL_w^2}{\Lambda \rd_t w}{h_w} + \tfrac{3}{2} \nu^2  \inner{\cL_{\partial w}^2}{\dx w}{\dx h_w} 
		\hspace{2.53cm}\label{eq:WeakWave} \\
		+ \nu^2 \inner{\cL_w^2}{\Lambda^2 w}{h_w}  
		- \nu^2 \inner{\cL_w^2}{\Lambda^2 \rM J_8 f(v)}{h_w} &=0, \nonumber 
	\end{align}
	with the initial values
	\begin{equation}
		v (0) = v_0, \quad i (0) = i_0, \quad \rd_t i (0) = i_0', \quad w (0) = w_0, \quad \rd_t w(0) = w_0'. \label{eq:WeakInitial}
	\end{equation}
\end{definition}

\begin{definition}[Strong solution] \label{def:StrongSolution}
	A solution $(v,i,w)$ is called an $\Omega$-periodic \emph{strong solution} of the initial value problem \eqref{eq:Voltage}--\eqref{eq:InitialValues} if it solves the strong version of the problem wherein the partial differential equations are understood as equalities in $L^2(0,T;\cL^2_v \times \cL^2_i\times \cL^2_w)$. 
	That is, the functions 
	\begin{equation*}
		v \in H^1(0,T;\cL_v^2), \quad i \in H^2(0,T;\cL_i^2), \quad w \in L^2(0,T;\cH_w^2),
	\end{equation*}
	with 
	\begin{alignat*}{3}
		\rd_t v &\in L^2(0,T;\cL_v^2), &\quad 
		\rd_t i &\in H^1(0,T;\cL_i^2), &\quad 
		\rd_t^2 i &\in L^2(0,T;\cL_i^2), \\
		\rd_t w &\in L^2(0,T;\cH_w^1), &\quad
		\rd_t^2 w &\in L^2(0,T;\cL_w^2), &&
	\end{alignat*}
	construct an $\Omega$-periodic strong solution for \eqref{eq:Voltage}--\eqref{eq:InitialValues} wherein they solve the equations for almost every $x \in \Omega$ and almost every $t \in [0,T]$, $T>0$.
\end{definition}

Now, let  $\cB_v = \seq{\ell_v^{(l)}}_{l=1}^{\infty}$ be a basis of $\cL_v^2$ such that  $\seq{\Phi^{\frac{1}{2}} \ell_v^{(l)}}_{l=1}^{\infty}$ is orthonormal in $\cL_v^2$.
Note that \eqref{eq:Parameters}, with the range of values given in Table \ref{tb:Parameters}, implies that 
$\Phi$ is a positive-definite diagonal matrix, and hence, such a basis exists.
Moreover, let $\cB_i =\seq{ \ell_i^{(l)}}_{l=1}^{\infty}$ be an orthonormal basis of $\cL_i^2$,
and 
$\cB_w =\seq{ h_w^{(l)}}_{l=1}^{\infty}$ be an orthogonal basis of $\cH_w^1$ that is orthonormal in $\cL_w^2$;
see Lemma \ref{lem:DualOrthogonalBasis} for the existence and structure of $\cB_w$.  
Finally, construct the set $\cB = \seq{b^{(k)}}_{k=1}^{\infty} \subset \cL^2_v \times \cL^2_i \times \cH_w^1$ as 
\begin{align} \label{eq:Basis}
	\cB := \cB_v \times \cB_i \times \cB_w = \left\{ b^{(k)} = ( \ell_v^{(k)}, \ell_i^{(k)}, h_w^{(k)} ) : 
	\ell_v^{(k)} \in \cB_v,  \ell_i^{(k)} \in \cB_i, h_w^{(k)} \in \cB_w \right \}_{k=1}^{\infty}.
\end{align} 
For each positive integer $m$, we seek approximations 
$v^{(m)}:[0,T]\rightarrow \cL_v^2$, 
$i^{(m)}:[0,T]\rightarrow \cL_i^2$, and
$w^{(m)}:[0,T]\rightarrow \cH_w^1$ 
of the form
\begin{align}
	v^{(m)} (t) &:= \sum \nolimits_{k=1}^m  c_{v_k}^{(m)}(t) \ell_v^{(k)},    \label{eq:VoltageExpansion} \\
	i^{(m)} (t) &:= \sum \nolimits_{k=1}^m  c_{i_k}^{(m)}(t) \ell_i^{(k)},    \label{eq:CurrentExpansion}  \\
	w^{(m)} (t) &:= \sum \nolimits_{k=1}^m  c_{w_k}^{(m)}(t) h_w^{(k)},    \label{eq:WaveExpansion}
\end{align}
constructed by the first $m$ components of $\cB$ and sufficiently smooth  scalar-valued functions $c_{v_k}^{(m)}$, $c_{i_k}^{(m)}$, and $c_{w_k}^{(m)}$ on $[0,T]$, such that these approximations satisfy 
\begin{align}
	\inner{\cL_v^2}{\Phi \rd_t v^{(m)}}{\ell_v^{(k)}} + \inner{\cL_v^2}{v^{(m)}}{\ell_v^{(k)}} - \inner{\cL_v^2}{J_1 i^{(m)}}{\ell_v^{(k)}} \hspace{2.58cm} & \label{eq:VoltageApprx}\\
	+ \inner{\cL_v^2}{ J_2 v^{(m)} i^{{(m)}^{\rT}} \Psi J_4 + J_3 v^{(m)} i^{{(m)}^{\rT}} \Psi J_5}{\ell_v^{(k)}} &= 0,  \nonumber \\
	\inner{\cL_i^2}{\rd_t^2i^{(m)}}{\ell_i^{(k)}} + 2 \inner{\cL_i^2}{\Gamma \rd_t i^{(m)}}{\ell_i^{(k)}} + \inner{\cL_i^2}{\Gamma^2 i^{(m)}}{\ell_i^{(k)}} \hspace{2.21cm} \label{eq:CurrentApprx}&\\
	- e \inner{\cL_i^2}{\Upsilon \Gamma J_6 w^{(m)}}{\ell_i^{(k)}} - e \inner{\cL_i^2}{\Upsilon \Gamma \rN J_7 f( v^{(m)})}{\ell_i^{(k)}} \nonumber
	&= e \inner{\cL_i^2}{\Upsilon\Gamma g}{\ell_i^{(k)}},  \nonumber \\
	\inner{\cL_w^2}{\rd_t^2 w^{(m)}}{h_w^{(k)}} + 2 \nu \inner{\cL_w^2}{\Lambda \rd_t w^{(m)}}{h_w^{(k)}} + \tfrac{3}{2} \nu^2  \inner{\cL_{\partial w}^2}{\dx w^{(m)}}{\dx h_w^{(k)}} \hspace{0.0cm}&\label{eq:WaveApprx}\\
	+ \nu^2 \inner{\cL_w^2}{\Lambda^2 w^{(m)}}{h_w^{(k)}} 	- \nu^2 \inner{\cL_w^2}{\Lambda^2 \rM J_8 f(v^{(m)} )}{h_w^{(k)}} &=0, \nonumber
\end{align}
for all $t \in [0,T]$ and $k=1,\dots,m$,  
subject to the initial conditions 
\begin{alignat}{3} \label{eq:ApprxInitial}
	c_{v_k}^{(m)} (0) &= \inner{\cL_v^2}{v_0}{\ell_v^{(k)}}, &\quad 
	c_{i_k}^{(m)} (0) &= \inner{\cL_i^2}{i_0}{\ell_i^{(k)}}, &\quad
	\rd_t c_{i_k}^{(m)}(0) &= \inner{\cL_i^2}{i_0'}{\ell_i^{(k)}}, \\
	c_{w_k}^{(m)} (0) &= \inner{\cL_w^2}{w_0}{h_w^{(k)}}, &\quad
	\rd_t c_{w_k}^{(m)}(0) &= \inner{\cL_w^2}{w_0'}{h_w^{(k)}}, &&\nonumber
\end{alignat}
on the coefficients
$c_k^{(m)}(t) := (c_{v_k}^{(m)}(t), c_{i_k}^{(m)}(t), c_{w_k}^{(m)}(t)) \in \bbR^3$. 

Equations \eqref{eq:VoltageApprx}--\eqref{eq:ApprxInitial} are equivalent to a system of nonlinear $3m$-dimensional ordinary differential equations
on coefficients $c^{(m)}(t)=(c_1^{(m)}(t), \dots, c_m^{(m)}(t)) \in \bbR^{3m}$.
Therefore, by the standard theory of ordinary differential equations \cite[Th. 2.1]{Taylor:PDEI:2011}, 
there exists a unique function $c^{(m)}(t)$ that solves \eqref{eq:VoltageApprx}--\eqref{eq:ApprxInitial} 
for $t \in [0,T_m)$, $T_m >0$, with the approximations \eqref{eq:VoltageExpansion}--\eqref{eq:WaveExpansion}.  Moreover, $T_m = T$ for all positive integers $m$, which follows from Proposition \ref{prp:EnergyEstimates}.

The standard Galerkin approximation method involves providing energy estimates that are uniform in $m$ for all the approximations
$(v^{(m)}, i^{(m)}, w^{(m)})$. 
Such \textit{a priori} energy estimates then allow construction of solutions by passing to the limits as $m \rightarrow \infty$.
The following proposition gives the desired estimates for the approximations \eqref{eq:VoltageExpansion}--\eqref{eq:WaveExpansion}.

\begin{proposition} [Energy estimates] \label{prp:EnergyEstimates}
	Suppose $g \in L^2(0,T; \cL_i^2)$ and for every positive integer $m$ let $v^{(m)}$, $i^{(m)}$, and $w^{(m)}$ be 
	functions of the form \eqref{eq:VoltageExpansion}--\eqref{eq:WaveExpansion}, respectively,
	satisfying \eqref{eq:VoltageApprx}--\eqref{eq:WaveApprx} with the initial conditions \eqref{eq:ApprxInitial}. 
	Then there exist positive constants $\alpha_v$, $\beta_v$, $\alpha_i$, and $\alpha_w$, 
	dependent only on the parameters of the model, such that for every positive integer $m$,
	\begin{align} 
		\sup_{t\in [0,T]}\left( \norm{\cL_v^2}{v^{(m)}(t)}^2 \right) + \norm{L^2(0,T;{\cL^2_v}^{\ast})}{\rd_t v^{(m)}}^2 &\leq \kappa_v, \label{eq:BoundVoltage} \\
		\sup_{t\in [0,T]}\left(\norm{\cL_i^2}{\rd_t i^{(m)}(t)}^2 + \norm{\cL_i^2}{i^{(m)}(t)}^2 \right) 
		+\norm{L^2(0,T; {\cL^2_i}^{\ast})}{\rd_t^2 i^{(m)}}^2 &\leq \kappa_i,  \label{eq:BoundCurrent} \\
		\sup_{t\in[0,T]} \left(\norm{\cL_w^2}{\rd_t w^{(m)}(t)}^2  + \norm{\cH_w^1}{w^{(m)}(t)}^2 \right) 
		+ \norm{L^2(0,T;{\cH^1_w}^{\!\!\ast})}{\rd_t^2 w^{(m)}}^2 &\leq \kappa_w, 	 \label{eq:BoundWave}
	\end{align}
	where $\kappa_v$, $\kappa_i$, and $\kappa_w$ are positive constants given, independently of $m$, by 
	\begin{align} 
		\kappa_v &:= \alpha_v \left( \big(1+(1+\sqrt{\kappa_i})^2T \big) \exp \left(\beta_v \sqrt{\kappa_i} T \right) \left[ \norm{\cL_v^2}{v_0}^2 
		+ \sqrt{\kappa_i} \right] + \kappa_i T \right),  \label{eq:KappaV}\\
		\kappa_i &:= \alpha_i \left( (1+T)\left[ \norm{\cL_i^2}{i'_0}^2 + \norm{\cL_i^2}{i_0}^2 \right]
		+(2+T)\Big[ T \left( \kappa_w + |\Omega| \right) + \norm{L^2(0,T;\cL_i^2)}{g}^2 \Big]\right), \label{eq:KappaI}\\
		\kappa_w &:= \alpha_w \left(
		(1+T) \left[ \norm{\cL_w^2}{w_0'}^2  + \norm{\cH_w^1}{w_0}^2 \right] + (2+T) T |\Omega|  \right). \label{eq:KappaW}
	\end{align}
\end{proposition}
\begin{proof}
	Multiplying \eqref{eq:WaveApprx} by $\rd_t c_{w_k}^{(m)}$ and summing over $k=1, \dots, m$ yields
	\begin{multline*}
		\inner{\cL_w^2}{\rd_t^2 w^{(m)}}{\rd_t w^{(m)}} + 2 \nu \inner{\cL_w^2}{\Lambda \rd_t w^{(m)}}{\rd_t w^{(m)}} + \tfrac{3}{2} \nu^2  \inner{\cL_{\partial w}^2}{\dx w^{(m)}}{\rd_t \dx w^{(m)}}  \\
		+ \nu^2 \inner{\cL_w^2}{\Lambda^2 w^{(m)}}{\rd_t w^{(m)}} 	- \nu^2 \inner{\cL_w^2}{\Lambda^2 \rM J_8 f(v^{(m)} )}{\rd_t w^{(m)}} =0, 
	\end{multline*} 
	or, equivalently,
	\begin{multline*}
		\tfrac{1}{2} \rd_t \left[ \norm{\cL_w^2}{\rd_t w^{(m)}}^2  + \tfrac{3}{2} \nu^2  \norm{\cL^2_{\partial w}}{\dx w^{(m)}}^2 
		+ \nu^2 \norm{\cL_w^2}{\Lambda w^{(m)}}^2 \right]	+ 2 \nu \norm{\cL_w^2}{\Lambda ^{\frac{1}{2}}\rd_t w^{(m)}}^2  \\
		- \nu^2 \inner{\cL_w^2}{\Lambda^2 \rM J_8 f(v^{(m)} )}{\rd_t w^{(m)}} =0. 
	\end{multline*}
	Now, Young's inequality implies that, for every $\varepsilon_1>0$,
	\begin{align*}
		\nu^2 \inner{\cL_w^2}{\Lambda^2 \rM J_8 f(v^{(m)} )}{\rd_t w^{(m)}} &\leq
		\varepsilon_1 \nu^2 \norm{\cL_w^2}{\rd_t w^{(m)}}^2 + \frac{\nu^2}{4\varepsilon_1}\norm{\cL_w^2}{\Lambda^2 \rM J_8 f(v^{(m)} )}^2  \\
		&= \varepsilon_1 \nu^2 \norm{\cL_w^2}{\rd_t w^{(m)}}^2 +\frac{\nu^2}{4\varepsilon_1}\trace (\Lambda^4 \rM^2)  \int_{\Omega} \big |f_{\rE}(v_{\rE}^{(m)}) \big |^2 \rd x \\
		& \leq  \varepsilon_1 \nu^2 \norm{\cL_w^2}{\rd_t w^{(m)}}^2  +\frac{\nu^2}{4\varepsilon_1} |\Omega| \rF_{\rE}^2 \trace (\Lambda^4 \rM^2).
	\end{align*}
	Therefore,
	\begin{multline*}
		\rd_t \left[ \norm{\cL_w^2}{\rd_t w^{(m)}}^2  + \tfrac{3}{2} \nu^2  \norm{\cL^2_{\partial w}}{\dx w^{(m)}}^2 + \nu^2 \norm{\cL_w^2}{\Lambda w^{(m)}}^2 \right]	
		+ 2\nu(2\Lambda_{\min} - \varepsilon_1 \nu) \norm{\cL_w^2}{\rd_t w^{(m)}}^2   \\ 
		\leq \frac{\nu^2}{2\varepsilon_1} |\Omega| \rF_{\rE}^2 \trace (\Lambda^4 \rM^2), 
	\end{multline*}
	where $\Lambda_{\min}:= \min \{\Lambda_{\rE \rE}, \Lambda_{\rE \rI} \}$ is the smallest eigenvalue of $\Lambda$. 
	
	Next, setting $\varepsilon_1 = \frac{2}{\nu}\Lambda_{\min}$ and integrating with respect to time over $[0,t]$
	yields
	\begin{multline*}
		\norm{\cL_w^2}{\rd_t w^{(m)}(t)}^2  + \tfrac{3}{2} \nu^2  \norm{\cL_{\partial w}^2}{\dx w^{(m)}(t)}^2 + \nu^2 \norm{\cL_w^2}{\Lambda w^{(m)}(t)}^2 \\
		\leq \left( \norm{\cL_w^2}{\rd_t w^{(m)}}^2  + \tfrac{3}{2} \nu^2  \norm{\cL_{\partial w}^2}{\dx w^{(m)}}^2 + \nu^2 \norm{\cL_w^2}{\Lambda w^{(m)}}^2 \right)  \hspace{-1mm}\Big|_{t=0}
		+ \tfrac{1}{4} \frac{\nu^3}{\Lambda_{\min}} |\Omega| \rF_{\rE}^2 \trace (\Lambda^4 \rM^2) t,
	\end{multline*}
	which, using \eqref{eq:ApprxInitial}, implies 
	\begin{align*}
		\norm{\cL_w^2}{\rd_t w^{(m)}(t)}^2  + \norm{\cH_w^1}{w^{(m)}(t)}^2 
		\leq \hat{\alpha}_w \left( \norm{\cL_w^2}{w_0'}^2  + \norm{\cH_w^1}{w_0}^2 + 
		\tfrac{1}{4}\frac{\nu^3}{\Lambda_{\min}} |\Omega| \rF_{\rE}^2 \trace (\Lambda^4 \rM^2) t \right)
	\end{align*}
	for all $t\in[0,T]$ and some $\hat{\alpha}_w > 0$.
	Since this inequality holds for all $t\in[0,T]$, it follows that
	\begin{align} \label{eq:BoundWaveHat}
		\sup_{t\in[0,T]} \left(\norm{\cL_w^2}{\rd_t w^{(m)}(t)}^2  + \norm{\cH_w^1}{w^{(m)}(t)}^2 \right) \leq \hat{\kappa}_w,
	\end{align}
	where
	\begin{align*}
		\hat{\kappa}_w := \hat{\alpha}_w \left(
		\norm{\cL_w^2}{w_0'}^2  + \norm{\cH_w^1}{w_0}^2 + \tfrac{1}{4}\frac{\nu^3}{\Lambda_{\min}} |\Omega| \rF_{\rE}^2 \trace (\Lambda^4 \rM^2) T \right).
	\end{align*}
	
	Now, fix $\bar{h} \in \cH_w^1$ such that $\norm{\cH_w^1}{\bar{h}} \leq 1$ and decompose $\bar{h}$ as $\bar{h} = h + h^{\perp}$, where 
	$h \in \mathrm{span} \seq{ h_w^{(k)}}_{k=1}^m$ and
	$\inner{\cL_w^2}{h_w^{(k)}}{h^{\perp}} = 0$, $k=1,\dots, m$.
	Since the basis $\cB_w$ used to construct $\cB$ in \eqref{eq:Basis} is orthonormal in $\cL_w^2$, it follows from  \eqref{eq:WaveExpansion} that
	\begin{equation*}
		\pair{\cH_w^1}{\rd_t^2  w^{(m)}}{\bar{h}} = \inner{\cL_w^2}{\rd_t^2 w^{(m)}}{\bar{h}} = \inner{\cL_w^2}{\rd_t^2 w^{(m)}}{h},
	\end{equation*}
	where the first equality holds since $\rd_t^2 w^{(m)} \in \cH_w^1$; see the proof of \cite[Th. 5.9-1]{Evans:PDE:2010}.
	Therefore, \eqref{eq:WaveApprx} gives
	\begin{multline*} 
		\pair{\cH_w^1}{\rd_t^2  w^{(m)}}{\bar{h}} = 
		-2 \nu \inner{\cL_w^2}{\Lambda \rd_t w^{(m)}}{h} - \tfrac{3}{2} \nu^2  \inner{\cL_{\partial w}^2}{\dx w^{(m)}}{\dx h} \\
		- \nu^2 \inner{\cL_w^2}{\Lambda^2 w^{(m)}}{h} + \nu^2 \inner{\cL_w^2}{\Lambda^2 \rM J_8 f(v^{(m)} )}{h}. 
	\end{multline*}
	Since $\cB_w$ is orthogonal in $\cH_w^1$ we have 
	$\norm{\cH_w^1}{h} \leq \norm{\cH_w^1}{\bar{h}} \leq 1$, and hence, the Cauchy-Schwarz inequality gives
	\begin{align*}
		\big| \pair{\cH_w^1}{\rd_t^2 w^{(m)}&}{\bar{h}} \big | \\&\leq
		2 \nu \norm{\cL_w^2}{\rd_t w^{(m)}} \! \! + \tfrac{3}{2} \nu^2  \norm{\cL_{\partial w}^2}{\dx w^{(m)}}
		\!+ \nu^2 \norm{\cL_w^2}{\Lambda^2 w^{(m)}} \! \! \!+ \nu^2 \norm{\cL_w^2}{\Lambda^2 \rM J_8 f(v^{(m)} )}  \\
		&\leq \alpha_1 \left( \norm{\cL_w^2}{\rd_t w^{(m)}} + \norm{\cH_w^1}{w^{(m)}}
		+ \nu^2 \big(|\Omega| \rF_{\rE}^2 \trace (\Lambda^4 \rM^2) \big)^{\frac{1}{2}}  \right)
	\end{align*}
	for some $\alpha_1>0$. 
	Therefore, there exists $\alpha_2>0$ such that
	\begin{align*}
		\int_0^T \norm{{\cH^1_w}^{\!\!\ast}}{\rd_t^2 w^{(m)}}^2 \rd t &\leq
		\alpha_2 \int_0^T \left( \norm{\cL_w^2}{\rd_t w^{(m)}}^2 + \norm{\cH_w^1}{w^{(m)}}^2
		+ \nu^4 |\Omega| \rF_{\rE}^2 \trace (\Lambda^4 \rM^2) \right) \rd t, 
	\end{align*} 
	which, using \eqref{eq:BoundWaveHat}, yields
	\begin{align*}
		\norm{L^2(0,T;{\cH^1_w}^{\!\!\ast})}{\rd_t^2 w^{(m)}}^2 \leq \alpha_2 \left( \hat{\kappa}_w + \nu^4 |\Omega| \rF_{\rE}^2 \trace (\Lambda^4 \rM^2)  \right)T.
	\end{align*}
	This inequality, together with \eqref{eq:BoundWaveHat}, establishes the bound \eqref{eq:BoundWave} 
	with \eqref{eq:KappaW} for some $\alpha_w >0$. 
	
	Next, multiplying \eqref{eq:CurrentApprx} by $\rd_t c_{i_k}^{(m)}$ and summing over $k=1, \dots, m$ yields
	\begin{multline} \label{eq:CurrentApprx2}
		\inner{\cL_i^2}{\rd_t^2i^{(m)}}{\rd_t i^{(m)}} + 2 \inner{\cL_i^2}{\Gamma \rd_t i^{(m)}}{\rd_t i^{(m)}} + \inner{\cL_i^2}{\Gamma^2 i^{(m)}}{\rd_t i^{(m)}} \\
		\hspace{1.1cm} - e \inner{\cL_i^2}{\Upsilon \Gamma J_6 w^{(m)}}{\rd_t i^{(m)}} - e \inner{\cL_i^2}{\Upsilon \Gamma \rN J_7 f( v^{(m)})}{\rd_t i^{(m)}} 
		= e \inner{\cL_i^2}{\Upsilon\Gamma g}{\rd_t i^{(m)}}.
	\end{multline}
	For the second term we have
	\begin{align*}
		\inner{\cL_i^2}{\Gamma \rd_t i^{(m)}}{\rd_t i^{(m)}} \geq \gamma_{\min} \norm{\cL_i^2}{\rd_t i^{(m)}}^2, \quad
	\end{align*}
	where $\gamma_{\min}:= \min\{ \gamma_{\rE \rE}, \gamma_{\rE \rI}, \gamma_{\rI \rE}, \gamma_{\rI \rI} \}$
	is the smallest eigenvalue of $\Gamma$.
	Now, using Young's inequality and recalling \eqref{eq:BoundWave} we obtain, for every $\varepsilon_2, \dots, \varepsilon_4 > 0$,  
	\begin{align*}
		e \inner{\cL_i^2}{\Upsilon \Gamma J_6 w^{(m)}}{\rd_t i^{(m)}} 
		&  \leq \varepsilon_2 \norm{\cL_i^2}{\rd_t i^{(m)}}^2 + \frac{e^2}{4 \varepsilon_2} \norm{\cL_i^2}{\Upsilon \Gamma J_6 w^{(m)}}^2 \\
		& \leq \varepsilon_2 \norm{\cL_i^2}{\rd_t i^{(m)}}^2 
		+ \frac{e^2}{4 \varepsilon_2} \norm{2}{\Upsilon \Gamma J_6}^2 \norm{\cL_w^2}{ w^{(m)}}^2 \\
		&  \leq 	\varepsilon_2 \norm{\cL_i^2}{\rd_t i^{(m)}}^2 
		+ \frac{e^2 \kappa_w}{4 \varepsilon_2} \norm{2}{\Upsilon \Gamma J_6}^2, \\
		e \inner{\cL_i^2}{\Upsilon \Gamma \rN J_7 f( v^{(m)})}{\rd_t i^{(m)}}
		& \leq \varepsilon_3 \norm{\cL_i^2}{\rd_t i^{(m)}}^2 + \frac{e^2}{4 \varepsilon_3} \norm{\cL_i^2}{\Upsilon \Gamma \rN J_7 f( v^{(m)})}^2 \\
		& \leq \varepsilon_3 \norm{\cL_i^2}{\rd_t i^{(m)}}^2 
		+ \frac{e^2}{4 \varepsilon_3} \norm{2}{\Upsilon \Gamma \rN J_7}^2 \norm{\cL_v^2} {f( v^{(m)})}^2  \\
		& \leq  \varepsilon_3 \norm{\cL_i^2}{\rd_t i^{(m)}}^2
		+ \frac{e^2 |\Omega|}{4 \varepsilon_3}  (\rF_{\rE}^2 + \rF_{\rI}^2 ) \norm{2}{\Upsilon \Gamma \rN J_7}^2,\\
		e \inner{\cL_i^2}{\Upsilon\Gamma g}{\rd_t i^{(m)}} 
		& \leq \varepsilon_4 \norm{\cL_i^2}{\rd_t i^{(m)}}^2 + \frac{e^2}{4 \varepsilon_4} \norm{\cL_i^2}{\Upsilon\Gamma g}^2 \\
		& \leq \varepsilon_4 \norm{\cL_i^2}{\rd_t i^{(m)}}^2 + \frac{e^2}{4 \varepsilon_4} \norm{2}{\Upsilon\Gamma}^2 \norm{\cL_i^2}{g}^2.
	\end{align*} 
	Hence, with the above inequalities, \eqref{eq:CurrentApprx2} implies
	\begin{multline*} 
		\rd_t\left[ \norm{\cL_i^2}{\rd_t i^{(m)}}^2 + \norm{\cL_i^2}{\Gamma i^{(m)}}^2 \right] 
		+ 2(2 \gamma_{\min} - \varepsilon_2 - \varepsilon_3 -\varepsilon_4) \norm{\cL_i^2}{\rd_t i^{(m)}}^2\\
		\leq \frac{e^2 \kappa_w}{2 \varepsilon_2}  \norm{2}{\Upsilon \Gamma J_6}^2
		+ \frac{e^2 |\Omega|}{2 \varepsilon_3}  (\rF_{\rE}^2 + \rF_{\rI}^2 ) \norm{2}{\Upsilon \Gamma \rN J_7}^2
		+  \frac{e^2}{2 \varepsilon_4} \norm{2}{\Upsilon\Gamma}^2 \norm{\cL_i^2}{g}^2.
	\end{multline*}
	
	Now, setting $\varepsilon_2 = \varepsilon_3 = \frac{1}{2}\gamma_{\min}$ and $\varepsilon_4 = \gamma_{\min}$,
	integrating with respect to time over $[0,t]$, and taking the supremum over $t \in [0,T]$  we have
	\begin{align} \label{eq:BoundCurrentHat}
		\sup_{t\in [0,T]}\left(\norm{\cL_i^2}{\rd_t i^{(m)}(t)}^2 + \norm{\cL_i^2}{i^{(m)}(t)}^2 \right) \leq \hat{\kappa}_i ,
	\end{align}
	where, for some $\hat{\alpha}_i >0$,
	\begin{multline*}
		\hat{\kappa}_i = \hat{\alpha}_i \left( \norm{\cL_i^2}{i'_0}^2 + \norm{\cL_i^2}{i_0}^2
		+\left[ \frac{e^2 \kappa_w}{\gamma_{\min}} \norm{2}{\Upsilon\Gamma J_6}^2 
		+ \frac{e^2 |\Omega|}{\gamma_{\min}}  (\rF_{\rE}^2 + \rF_{\rI}^2 ) \norm{2}{\Upsilon \Gamma \rN J_7}^2 \right]T \right. \\ \left.
		+ \frac{e^2}{2 \gamma_{\min}} \norm{2}{\Upsilon\Gamma}^2 \norm{L^2(0,T;\cL_i^2)}{g}^2 \right).
	\end{multline*}
	Fix $\bar{\ell} \in \cL_i^2$ such that $\norm{\cL_i^2}{\bar{\ell}} \leq 1$ and decompose $\bar{\ell}$ as $\bar{\ell} = \ell + \ell^{\perp}$, where 
	$\ell \in \mathrm{span} \seq{ \ell_i^{(k)}}_{k=1}^m$ and
	$\inner{\cL_i^2}{\ell_i^{(k)}}{\ell^{\perp}} = 0$, $k=1,\dots, m$.
	Using \eqref{eq:CurrentExpansion} and \eqref{eq:CurrentApprx} we obtain
	\begin{align*}
		\pair{\cL_i^2}{\rd_t^2i^{(m)}}{\bar{\ell}\,} &= \,\inner{\cL_i^2}{\rd_t^2i^{(m)}}{\bar{\ell}} = \inner{\cL_i^2}{\rd_t^2i^{(m)}}{\ell} \\
		&= -2 \inner{\cL_i^2}{\Gamma \rd_t i^{(m)}}{\ell} - \inner{\cL_i^2}{\Gamma^2 i^{(m)}}{\ell} 
		+ e \inner{\cL_i^2}{\Upsilon \Gamma J_6 w^{(m)}}{\ell} + e \inner{\cL_i^2}{\Upsilon \Gamma \rN J_7 f( v^{(m)})}{\ell} \\
		&\quad + e \inner{\cL_i^2}{\Upsilon\Gamma g}{\ell}.
	\end{align*}
	The orthogonality of the basis $\cB_i$ in \eqref{eq:Basis} implies $\norm{\cL_i^2}{\ell} \leq 1$, and hence,
	\begin{align*}
		\big |\pair{\cL_i^2}{\rd_t^2 i^{(m)}}{\bar{\ell}} \big | &\leq  
		2 \norm{2}{\Gamma} \norm{\cL_i^2}{\rd_t i^{(m)}} +\; \norm{2}{\Gamma^2} \norm{\cL_i^2}{i^{(m)}} \\
		&\quad
		+ e \norm{\cL_i^2}{\Upsilon \Gamma J_6 w^{(m)}} + e \norm{\cL_i^2}{\Upsilon \Gamma \rN J_7 f( v^{(m)})}
		+ e \norm{\cL_i^2}{\Upsilon\Gamma g}.
	\end{align*}
	Therefore, it follows from the same inequalities used to derive \eqref{eq:BoundCurrentHat} that, for some $\alpha_3 >0$,
	\begin{multline*}
		\norm{L^2(0,T; {\cL^2_i}^{\ast})}{\rd_t^2 i^{(m)}}^2 \leq \alpha_3 \left( \left[\hat{\kappa}_i
		+ e^2 \kappa_w  \norm{2}{\Upsilon\Gamma J_6}^2 
		+ e^2 |\Omega|  (\rF_{\rE}^2 + \rF_{\rI}^2 ) \norm{2}{\Upsilon \Gamma \rN J_7}^2 \right] T \right.\\ \left.
		+ e^2 \norm{2}{\Upsilon\Gamma}^2 \norm{L^2(0,T;\cL_i^2)}{g}^2 
		\right).
	\end{multline*}
	This, together with \eqref{eq:BoundCurrentHat}, establishes the bound \eqref{eq:BoundCurrent} with
	\eqref{eq:KappaI} for some $\alpha_i >0$.
	
	Finally, multiplying \eqref{eq:VoltageApprx} by $c_{v_k}^{(m)}$ and summing over $k=1, \dots, m$ yields
	\begin{multline} \label{eq:VoltageApprx2}
		\inner{\cL_v^2}{\Phi \rd_t v^{(m)}}{v^{(m)}} + \inner{\cL_v^2}{v^{(m)}}{v^{(m)}} - \inner{\cL_v^2}{J_1 i^{(m)}}{v^{(m)}} \\
		+ \inner{\cL_v^2}{ J_2 v^{(m)} i^{{(m)}^{\rT}} \Psi J_4 + J_3 v^{(m)} i^{{(m)}^{\rT}} \Psi J_5}{v^{(m)}} = 0. 
	\end{multline}
	Now, using Young's inequality and recalling \eqref{eq:BoundCurrent} we obtain, for every $\varepsilon_5 >0$,
	\begin{align*}
		\inner{\cL_v^2}{J_1 i^{(m)}}{v^{(m)}}  &\leq \varepsilon_5 \norm{\cL_v^2}{v^{(m)}}^2 +\frac{1}{4 \varepsilon_5} \norm{\cL_v^2}{J_1 i^{(m)}}^2 \\
		&\leq \varepsilon_5 \norm{\cL_v^2}{v^{(m)}}^2 +\frac{1}{2 \varepsilon_5} \norm{\cL_v^2}{i^{(m)}}^2 \\
		& \leq \varepsilon_5 \norm{\cL_v^2}{v^{(m)}}^2 +\frac{\kappa_i}{2 \varepsilon_5}.
	\end{align*}
	Moreover, using H\"{o}lder's inequality in $\bbR^2$ and the Cauchy-Schwarz inequality in $\bbR^4$ we obtain
	\begin{align*}
		-\inner{\cL_v^2}{ J_2 v^{(m)} i^{{(m)}^{\rT}} \Psi J_4 + J_3 v^{(m)}& i^{{(m)}^{\rT}} \Psi J_5 }{v^{(m)}}\\
		&= -\int_{\Omega} \left( (v_1^{(m)})^2 i^{{(m)}^{\rT}} \Psi J_4 + (v_2^{(m)})^2 i^{{(m)}^{\rT}} \Psi J_5 \right)\rd x\\
		&\leq \int_{\Omega} \norm{\bbR^2}{v^{(m)}}^2 \max \left \{|i^{{(m)}^{\rT}} \Psi J_4|, |i^{{(m)}^{\rT}} \Psi J_5| \right\} \rd x\\
		&\leq \int_{\Omega} \norm{\bbR^2}{v^{(m)}}^2 \norm{\bbR^4}{i^{(m)}} 
		\max \left \{\norm{\bbR^4}{\Psi J_4}, \norm{\bbR^4}{\Psi J_5} \right\} \rd x  \\
		&\leq \sqrt{2\kappa_i} \norm{2}{\Psi} \norm{\cL_v^2}{v^{(m)}}^2.
	\end{align*}
	Therefore, \eqref{eq:VoltageApprx2} implies
	\begin{align*} 
		\rd_t \norm{\cL_v^2}{\Phi^{\frac{1}{2}} v^{(m)}}^2 + 2 \left( 1-\varepsilon_5 - \sqrt{2\kappa_i} \norm{2}{\Psi} \right)
		\norm{\cL_v^2}{v^{(m)}}^2 \leq \frac{\kappa_i}{\varepsilon_5}.
	\end{align*}
	Next, setting $\varepsilon_5 = 1$ and using Gr\"{o}nwall's inequality \cite[Sec. III.1.1.3.]{Temam:InfiniteDimensional:1997} yields
	\begin{align} \label{eq:BoundVoltageHat}
		\sup_{t\in [0,T]}\left( \norm{\cL_v^2}{v^{(m)}(t)}^2 \right) \leq \hat{\kappa}_v,
	\end{align}
	where, for some $\hat{\alpha}_v >0$ and $\hat{\beta}_v >0$,
	\begin{align*}
		\hat{\kappa}_v = \hat{\alpha}_v \exp \left( \hat{\beta}_v \sqrt{2\kappa_i} \norm{2}{\Psi} T \right) \left( \norm{\cL_v^2}{v_0}^2 + \frac{\kappa_i}{\sqrt{2 \kappa_i} \norm{2}{\Psi}}  \right).
	\end{align*}
	
	Now, fix $\bar{\ell} \in \cL_v^2$ such that $\norm{\cL_v^2}{\bar{\ell}} \leq 1$ and decompose $\bar{\ell}$ as  $\bar{\ell} = \ell + \ell^{\perp}$, where 
	$\ell \in \mathrm{span} \seq{ \ell_v^{(k)}}_{k=1}^m$ and
	$\inner{\cL_v^2}{\Phi \ell_v^{(k)}}{\ell^{\perp}} = 0$, $k=1,\dots, m$.
	Note that this decomposition exists due to the way we construct the basis $\cB_v$ in \eqref{eq:Basis}, wherein the elements, weighted by 
	$\Phi^{\frac{1}{2}}$, are orthonormal in $\cL_v^2$.
	Then, it follows from \eqref{eq:VoltageExpansion} and \eqref{eq:VoltageApprx} that
	\begin{align*} 
		\pair{\cL_v^2}{\Phi \rd_t v^{(m)}}{\bar{\ell}} &= \, \inner{\cL_v^2}{\Phi \rd_t v^{(m)}}{\bar{\ell}} = \inner{\cL_v^2}{\Phi \rd_t v^{(m)}}{\ell}\\
		&= -\inner{\cL_v^2}{v^{(m)}}{\ell} + \inner{\cL_v^2}{J_1 i^{(m)}}{\ell} 
		- \inner{\cL_v^2}{ J_2 v^{(m)} i^{{(m)}^{\rT}} \Psi J_4 + J_3 v^{(m)} i^{{(m)}^{\rT}} \Psi J_5}{\ell}. 
	\end{align*}
	Since $\cB_v$  is a $\Phi^{\frac{1}{2}}$-weighted orthonormal set in $\cL_v^2$, 
	it follows that 
	\begin{equation*}
		\norm{\cL_v^2}{\ell} \leq \norm{2}{\Phi^{-\frac{1}{2}}} \norm{\cL_v^2}{\Phi^{\frac{1}{2}}\ell}
		\leq \norm{2}{\Phi^{-\frac{1}{2}}} \norm{\cL_v^2}{\Phi^{\frac{1}{2}}\bar{\ell}} 
		\leq \norm{2}{\Phi^{-\frac{1}{2}}} \norm{2}{\Phi^{\frac{1}{2}}} \norm{\cL_v^2}{\bar{\ell}}
		\leq \norm{2}{\Phi^{-\frac{1}{2}}} \norm{2}{\Phi^{\frac{1}{2}}}
	\end{equation*}
	and hence, letting $\alpha_4:=\norm{2}{\Phi^{-\frac{1}{2}}} \norm{2}{\Phi^{\frac{1}{2}}}$ and using Cauchy-Schwarz inequality we have
	\begin{align*} 
		\big| \pair{\cL_v^2}{\Phi \rd_t v^{(m)}}{\bar{\ell}} \big| &\leq \alpha_4 \left(
		\norm{\cL_v^2}{v^{(m)}} + \norm{\cL_v^2}{J_1 i^{(m)}} 
		+ \norm{\cL_v^2}{ J_2 v^{(m)} i^{{(m)}^{\rT}} \Psi J_4 + J_3 v^{(m)} i^{{(m)}^{\rT}} \Psi J_5}  \right)\\
		& \leq \alpha_4 \left( \norm{\cL_v^2}{v^{(m)}} + \sqrt{2} \norm{\cL_i^2}{i^{(m)}}
		+ 2\sqrt{2} \norm{\cL_v^2}{v^{(m)}} \norm{\cL_i^2}{i^{(m)}} \norm{2}{\Psi} \right) \\
		& \leq \alpha_4 \left(\left( 1+ 2\sqrt{2 \kappa_i} \norm{2}{\Psi} \right)\norm{\cL_v^2}{v^{(m)}} + \sqrt{2 \kappa_i} \right),
	\end{align*}
	which, along with \eqref{eq:BoundVoltageHat} implies that, for some $\alpha_5 >0$,
	\begin{align*}
		\norm{L^2(0,T;{\cL^2_v}^{\ast})}{\rd_t v^{(m)}}^2 \leq \alpha_5 
		\left(\left( 1+ 2\sqrt{2 \kappa_i} \norm{2}{\Psi} \right)^2 \hat{\kappa}_v + 2 \kappa_i \right) T.
	\end{align*}
	This, together with \eqref{eq:BoundVoltageHat}, establishes the bound \eqref{eq:BoundVoltage}
	with \eqref{eq:KappaV} for some $\alpha_v >0$. 
	Note that constants $\alpha_1, \dots, \alpha_5$, $\hat{\alpha}_v$, $\hat{\beta}_v$, $\hat{\alpha}_i$, and $\hat{\alpha}_w$ depend only on the parameters of the model, which further implies that the constants $\alpha_v$, $\beta_v$, $\alpha_i$, and $\alpha_w$ also depend only on the parameters of the model and completes the proof.
	\qquad
\end{proof}

\begin{theorem} [Existence and uniqueness of weak solutions] \label{th:WeakExistenceUniqueness}
	Suppose that $g \in L^2(0,T;\cL_i^2)$, $v_0 \in \cL_v^2$, $i_0 \in \cL_i^2$, $i'_0 \in \cL_i^2$, $w_0 \in \cH_w^1$, and $w'_0 \in \cL_w^2$. Then there exists a unique $\Omega$-periodic weak solution $(v,i,w)$ of the initial value problem \eqref{eq:Voltage}--\eqref{eq:InitialValues}.
\end{theorem}
\begin{proof}
	The energy estimate \eqref{eq:BoundVoltage} implies that the sequence $\seq{v^{(m)}}_{m=1}^{\infty}$ is bounded in $L^2(0,T; \cL_v^2)$ and the sequence $\seq{\rd_t v^{(m)}}_{m=1}^{\infty}$ is bounded in $L^2(0,T;{\cL^2_v}^{\ast})$. 
	Since ${\cL^2_v}^{\ast} = \cL_v^2$, it follows that $\seq{v^{(m)}}_{m=1}^{\infty}$ is bounded in $H^1(0,T;\cL_v^2)$ and $\seq{\rd_t v^{(m)}}_{m=1}^{\infty}$ is bounded in $L_2(0,T;\cL_v^2)$. 
	Similarly, since ${\cL^2_i}^{\ast} = \cL_i^2$, the energy estimate \eqref{eq:BoundCurrent} implies that the sequence $\seq{i^{(m)}}_{m=1}^{\infty}$ is bounded in $H^2(0,T;\cL_i^2)$, the sequence $\seq{\rd_t i^{(m)}}_{m=1}^{\infty}$ is bounded in $H^1(0,T;\cL_i^2)$, and the sequence $\seq{\rd_t^2 i^{(m)}}_{m=1}^{\infty}$ is bounded in $L^2(0,T;\cL_i^2)$. 
	Finally, the energy estimate \eqref{eq:BoundWave} implies that the sequence $\seq{w^{(m)}}_{m=1}^{\infty}$ is bounded in $L^2(0,T;\cH_w^1)$, the sequence  $\seq{\rd_t w^{(m)}}_{m=1}^{\infty}$ is bounded in $L^2(0,T;\cL_w^2)$, and the sequence $\seq{\rd_t^2 w^{(m)}}_{m=1}^{\infty}$ is bounded in $L^2(0,T;{\cH^1_w}^{\!\!\ast})$. 
	Now, it follows from the Rellich-Kondrachov compact embedding theorems \cite[Th. 6.6-3]{Ciarlet:FunctionalAnalysis:2013} that  $H^1(0,T;\cL_v^2) \Subset L^2(0,T;\cL_v^2)$ and $H^1(0,T;\cL_i^2) \Subset L^2(0,T;\cL_i^2)$. 
	Therefore, by \cite[Th. 2.10-1b]{Ciarlet:FunctionalAnalysis:2013}, there exist subsequences $\seq{v^{(m_k)}}_{k=1}^{\infty}$, $\seq{i^{(m_k)}}_{k=1}^{\infty}$, and $\seq{\rd_t i^{(m_k)}}_{k=1}^{\infty}$ such that 
	\begin{alignat}{2} \label{eq:StrongConvergence}
		v^{(m_k)} &\rightarrow v &\quad  \text{strongly in } & L^2(0,T;\cL_v^2),  \\
		i^{(m_k)} &\rightarrow i &\quad  \text{strongly in } & L^2(0,T;\cL_i^2),  \nonumber\\
		\rd_t i^{(m_k)} &\rightarrow i' &\quad  \text{strongly in } & L^2(0,T;\cL_i^2).\nonumber
	\end{alignat}
	Moreover, by the Banach-Eberlein-\v Smulian theorem \cite[Th. 5.14-4]{Ciarlet:FunctionalAnalysis:2013}, there exist subsequences  $\seq{\rd_t v^{(m_k)}}_{k=1}^{\infty}$, $\rd_t^2 \seq{i^{(m_k)}}_{k=1}^{\infty}$, $\seq{w^{(m_k)}}_{k=1}^{\infty}$, $\seq{\rd_t w^{(m_k)}}_{k=1}^{\infty}$, and $\seq{\rd_t^2 w^{(m_k)}}_{k=1}^{\infty}$ such that
	\begin{alignat}{2} \label{eq:WeakConvergence}
		\rd_t v^{(m_k)} &\rightharpoonup v' &\quad  \text{weakly in } & L^2(0,T;\cL_v^2),  \\
		\rd_t^2 i^{(m_k)} &\rightharpoonup i'' &\quad  \text{weakly in } & L^2(0,T;\cL_i^2),  \nonumber\\
		w^{(m_k)} &\rightharpoonup w &\quad  \text{weakly in } & L^2(0,T;\cH_w^1),\nonumber\\
		\rd_t w^{(m_k)} &\rightharpoonup w' &\quad  \text{weakly in } & L^2(0,T;\cL_w^2),\nonumber\\
		\rd_t^2 w^{(m_k)} &\rightharpoonup w'' &\quad  \text{weakly in } & L^2(0,T;{\cH^1_w}^{\!\!\ast}),\nonumber
	\end{alignat}
	where the time derivatives in the above analysis  are derivatives in the weak sense. 
	
	Next, we show that 
	\begin{equation*} 
		v'= \rd_t v, \quad i'= \rd_t i, \quad i''= \rd_t^2 i, \quad w'= \rd_t w, \quad w''= \rd_t^2 w.
	\end{equation*}
	Since $L^2(0,T; \cH_w^1)$ is reflexive, the weak and weak* convergence coincide. Recalling the definition of weak* convergence and weak derivatives, it follows that for every $h \in \cH_w^1$ and $\phi \in C_{\rc}^{\infty}([0,T])$, 
	\begin{align*}
		\pair{\cH_w^1}{\int_0^T w'' \phi  \rd t}{h} &= \int_0^T \pair{\cH_w^1}{w''\phi}{h}  \rd t  
		= \lim\limits_{k \rightarrow \infty} \int_0^T \pair{\cH_w^1}{\rd_t^2 w^{(m_k)} \phi}{ h}  \rd t \\
		&= \lim\limits_{k \rightarrow \infty} \pair{\cH_w^1}{\int_0^T \rd_t^2 w^{(m_k)} \phi \rd t}{h}  
		=   \lim\limits_{k \rightarrow \infty} \pair{\cH_w^1}{(-1)^2\int_0^T w^{(m_k)} \rd_t^2 \phi  \rd t }{h} \\
		&= \lim\limits_{k \rightarrow \infty} (-1)^2\int_0^T \pair{\cH_w^1}{w^{(m_k)} \rd_t^2 \phi}{h}  \rd t 
		= (-1)^2 \int_0^T \pair{\cH_w^1}{w \rd_t^2 \phi}{h} \rd t  \\
		&= \pair{\cH_w^1}{(-1)^2\int_0^T w \rd_t^2 \phi  \rd t}{h}, 
	\end{align*}
	which implies $w''=\rd_t^2 w$ in the weak sense. The other identities 
	are proved similarly.
	
	Now, recall \eqref{eq:FiringRateFunction} and \eqref{eq:Parameters} and note that the nonlinear map $f:\bbR^2 \rightarrow \bbR^2$ is bounded and smooth, and in particular, is Lipschitz continuous. Therefore, it follows from the strong convergence of $\seq{v^{(m_k)}}_{k=1}^{\infty}$ in \eqref{eq:StrongConvergence} that
	\begin{equation} \label{eq:NonlinearityLimit}
		f(v^{(m_k)}) \rightarrow f(v) \quad \text{strongly in } L^2(0,T;\cL_v^2).
	\end{equation} 
	For the bilinear term $J_2 v i^{\rT} \Psi J_4$, use \eqref{eq:BoundVoltage} and \eqref{eq:BoundCurrent} to write
	\begin{align*}
		\norm{L^2(0,T;\cL_v^2)}{  J_2  \big(& v i^{\rT} - v^{(m_k)} {i^{(m_k)}}^{\rT} \big) \Psi J_4} \\
		& \leq \norm{L^2(0,T;\cL_v^2)}{J_2 (v-v^{(m_k)}) i^{\rT} \Psi J_4}
		+\norm{L^2(0,T;\cL_v^2)}{J_2 v^{(m_k)} (i-i^{(m_k)})^{\rT}  \Psi J_4} \\
		& \leq \sqrt{2} \norm{2}{\Psi} \left[ 
		\norm{L^2(0,T;\cL_v^2)}{v-v^{(m_k)}} \norm{L^2(0,T;\cL_i^2)}{i}  
		+ \norm{L^2(0,T;\cL_v^2)}{v^{(m_k)}} \norm{L^2(0,T;\cL_i^2)}{i-i^{(m_k)}}   \right]\\
		& \leq \sqrt{2} \norm{2}{\Psi} \left[ \sqrt{\kappa_i}\, \norm{L^2(0,T;\cL_v^2)}{v-v^{(m_k)}} 
		+ \sqrt{\kappa_v}\, \norm{L^2(0,T;\cL_i^2)}{i-i^{(m_k)}}   \right].
	\end{align*}
	The same inequality holds for the bilinear term $J_3 v i^{\rT} \Psi J_5$ as well. Therefore,  \eqref{eq:StrongConvergence} gives
	\begin{alignat}{2} \label{eq:QuadraticLimit}
		J_2 v^{(m_k)} {i^{(m_k)}}^{\rT} \Psi J_4 &\rightarrow J_2 v i^{\rT} \Psi J_4 &\quad  \text{strongly in } & L^2(0,T;\cL_v^2),  \\
		J_3 v^{(m_k)} {i^{(m_k)}}^{\rT} \Psi J_5 &\rightarrow J_3 v i^{\rT} \Psi J_5 &\quad  \text{strongly in } & L^2(0,T;\cL_v^2).\nonumber
	\end{alignat} 
	
	Next, fix a positive integer $K$ and choose the functions
	\begin{alignat*}{2}
		\hat{v} &= \sum \nolimits_{k=1}^K  c_{v_k}(t) \ell_v^{(k)} & &\in C^1([0,T];\cL_v^2),  \\
		\hat{i} &= \sum \nolimits_{k=1}^K  c_{i_k}(t) \ell_i^{(k)} & &\in C^1([0,T];\cL_i^2),   \\
		\hat{w} &= \sum \nolimits_{k=1}^K  c_{w_k}(t) h_w^{(k)} & &\in C^1([0,T];\cH_w^1),   
	\end{alignat*}
	where $c_{v_k}$, $c_{i_k}$, and $c_{w_k}$ are sufficiently smooth functions on $[0,T]$,  and $(\ell_v^{(k)},\ell_i^{(k)},h_w^{(k)})$, $k =1, \dots, K$, are the first $K$ components of $\cB$ given by \eqref{eq:Basis}. 
	Set $m = m_k$ in \eqref{eq:VoltageApprx}--\eqref{eq:WaveApprx} and choose $m_k \geq K$. 
	Then, multiplying \eqref{eq:VoltageApprx}--\eqref{eq:WaveApprx} by $c_{v_k}$, $c_{i_k}$, and $c_{w_k}$, respectively, summing over $k=1,\dots, K$, and integrating over $t \in [0,T]$ yields
	\begin{align} \label{eq:ApproximateEquations}
		\int_0^T \left[\pair{\cL_v^2}{\Phi \rd_t v^{(m_k)}}{\hat{v}} + \inner{\cL_v^2}{v^{(m_k)}}{\hat{v}} - \inner{\cL_v^2}{J_1 i^{(m_k)}}{\hat{v}} \right. \hspace{4 cm} \\ \left.
		+ \inner{\cL_v^2}{ J_2 v^{(m_k)} i^{{(m_k)}^{\rT}} \Psi J_4 + J_3 v^{(m_k)} i^{{(m_k)}^{\rT}} \Psi J_5}{\hat{v}} \right] \rd t &= 0, \nonumber \\
		\int_0^T \left[\pair{\cL_i^2}{\rd_t^2i^{(m_k)}}{\hat{i}} + 2 \inner{\cL_i^2}{\Gamma \rd_t i^{(m_k)}}{\hat{i}} + \inner{\cL_i^2}{\Gamma^2 i^{(m_k)}}{\hat{i}} \right. \hspace{3.83 cm} \nonumber \\ \left.
		- e \inner{\cL_i^2}{\Upsilon \Gamma J_6 w^{(m_k)}}{\hat{i}} - e \inner{\cL_i^2}{\Upsilon \Gamma \rN J_7 f( v^{(m_k)})}{\hat{i}} - e \inner{\cL_i^2}{\Upsilon\Gamma g}{\hat{i}} \right]\rd t &=0, \nonumber \\
		\int_0^T \left[\pair{\cL_w^2}{\rd_t^2 w^{(m_k)}}{\hat{w}} + 2 \nu \inner{\cL_w^2}{\Lambda \rd_t w^{(m_k)}}{\hat{w}} + \tfrac{3}{2} \nu^2  \inner{\cL_{\partial w}^2}{\dx w^{(m_k)}}{\dx \hat{w}} \right. \hspace{1.37 cm} \nonumber\\ \left.
		+ \nu^2 \inner{\cL_w^2}{\Lambda^2 w^{(m_k)}}{\hat{w}} - \nu^2 \inner{\cL_w^2}{\Lambda^2 \rM J_8 f(v^{(m)} )}{\hat{w}}	\right] \rd t &=0. \nonumber
	\end{align}
	Note that the families of functions $\hat{v}$, $\hat{i}$, and $\hat{w}$ chosen above are dense in the spaces 
	$L^2(0,T;\cL_v^2)$, $L^2(0,T;\cL_i^2)$, and $L^2(0,T;\cH_w^1)$, respectively. Therefore, \eqref{eq:ApproximateEquations} holds for all functions  $\hat{v} \in L^2(0,T;\cL_v^2)$, $\hat{i} \in L^2(0,T;\cL_i^2)$, and $\hat{w} \in L^2(0,T;\cH_w^1)$. 
	Now, use \eqref{eq:StrongConvergence}--\eqref{eq:QuadraticLimit} to pass to the limits in \eqref{eq:ApproximateEquations}, which implies that \eqref{eq:WeakVoltage}--\eqref{eq:WeakWave} hold for all $\ell_v \in \cL_v^2$, $\ell_i \in \cL_i^2$, $h_w \in \cH_w^1$, and almost every $t \in [0,T]$.
	
	It remains to verify the initial conditions \eqref{eq:WeakInitial}. Choose the functions 
	\begin{equation*}
		\hat{v} \in C^1([0,T];\cL_v^2),  \quad
		\hat{i} \in C^2([0,T];\cL_i^2),  \quad 
		\hat{w} \in C^2([0,T];\cH_w^1),   
	\end{equation*}
	such that these functions vanish at the end point $t=T$. Integrating by parts in \eqref{eq:ApproximateEquations} yields
	\begin{align} \label{eq:Z1}
		\int_0^T \left[-\inner{\cL_v^2}{\Phi v^{(m_k)}}{\rd_t \hat{v}} + \cdots \right] \rd t 
		&= \inner{\cL_v^2}{\Phi v^{(m_k)}(0)}{\hat{v}(0)},  \\
		\int_0^T \left[\inner{\cL_i^2}{i^{(m_k)}}{\rd_t^2 \hat{i}} + \cdots \right] \rd t
		&= \cdots 
		+ \inner{\cL_i^2}{\rd_t i^{(m_k)}(0)}{\hat{i}(0)} 
		- \inner{\cL_i^2}{i^{(m_k)}(0)}{\rd_t\hat{i}(0)}, \nonumber  \\
		\int_0^T \left[\inner{\cH_w^1}{w^{(m_k)}}{\rd_t^2 \hat{w}} + \cdots \right] \rd t 
		&=	\inner{\cL_w^2}{\rd_t w^{(m_k)}(0)}{\hat{w}(0)} 
		- \inner{\cL_w^2}{w^{(m_k)}(0)}{\rd_t\hat{w}(0)}, \nonumber
	\end{align}
	where ``$\cdots$'' denotes terms that are not pertinent to the analysis. Similarly, integrating by parts in the limit of \eqref{eq:ApproximateEquations} yields
	\begin{align} \label{eq:Z2}
		\int_0^T \left[-\inner{\cL_v^2}{\Phi v}{\rd_t \hat{v}} + \cdots \right] \rd t 
		&= \inner{\cL_v^2}{\Phi v(0)}{\hat{v}(0)},  \\
		\int_0^T \left[\inner{\cL_i^2}{i}{\rd_t^2 \hat{i}} + \cdots \right] \rd t
		&= \cdots 
		+ \inner{\cL_i^2}{\rd_t i(0)}{\hat{i}(0)} 
		- \inner{\cL_i^2}{i(0)}{\rd_t\hat{i}(0)} ,  \nonumber \\
		\int_0^T \left[\inner{\cH_w^1}{w}{\rd_t^2 \hat{w}} + \cdots \right] \rd t 
		&=	\inner{\cL_w^2}{\rd_t w(0)}{\hat{w}(0)} 
		- \inner{\cL_w^2}{w(0)}{\rd_t\hat{w}(0)}. \nonumber
	\end{align}
	Now, consider the initial conditions \eqref{eq:ApprxInitial}, pass to the limits in \eqref{eq:Z1} through \eqref{eq:StrongConvergence}--\eqref{eq:QuadraticLimit}, and compare the results with \eqref{eq:Z2}.
	Since $\hat{v}$, $\hat{i}$, and $\hat{w}$ are arbitrary, the initial condition \eqref{eq:WeakInitial} holds and this completes the proof of existence.
	
	To prove uniqueness, assume, by contradiction, that there exist two weak solutions $(\tilde{v},\tilde{i},\tilde{w})$ and $(\hat{v},\hat{i},\hat{w})$ for \eqref{eq:Model}, initiating from the same initial values, such that
	$(\tilde{v},\tilde{i},\tilde{w}) \neq (\hat{v},\hat{i},\hat{w})$. 
	Then,  $(v,i,w):=(\tilde{v},\tilde{i},\tilde{w})-(\hat{v},\hat{i},\hat{w})$ is a weak solution initiating from the zero initial condition $(v_0,i_0,i'_0,w_0,w'_0)=0$. 
	Now, fix $s \in [0,T]$ and define, for $0 \leq t \leq T$, the functions
	\begin{equation} \label{eq:pqDefinition}
		p(t) := \int_{0}^{t} w(r) \rd r, \quad \quad  
		q(t) := \begin{cases}
			\int_{t}^{s} w(r) \rd r, & \text{if } 0\leq t \leq s,\\
			0, & \text{if } s< t \leq T.
		\end{cases}
	\end{equation}
	Note that $p(t)\in \cH_w^1$ and $q(t)\in \cH_w^1$ for all $t\in [0,T]$, and hence, $p$ and $q$ are regular enough to be used as the test function $h_w$ in \eqref{eq:WeakWave}. 
	Moreover, $q(s) = 0$, $q(0) = p(s)$, and $p(0)=0$. 
	Let $\tilde{u}$ and $\hat{u}$ satisfy \eqref{eq:WeakVoltage}--\eqref{eq:WeakWave} with the same test functions $\ell_v = v(t)$, $\ell_i = \rd_t i(t)$, and $h_w = q(t)$. 
	Subtracting the two sets of equations and integrating over $t\in [0,s]$ yields
	\begin{align}
		\int_0^s \Big[ \pair{\cL_v^2}{\Phi \rd_t v}{v} + \inner{\cL_v^2}{v}{v} - \inner{\cL_v^2}{J_1 i}{v} \label{eq:VoltageWithInt} \hspace{6.62 cm} \\
		+ \inner{\cL_v^2}{ J_2 (\tilde{v} \tilde{i}^{\rT} - \hat{v} \hat{i}^{\rT}) \Psi J_4 + J_3 (\tilde{v} \tilde{i}^{\rT} -\hat{v} \hat{i}^{\rT}) \Psi J_5}{v} \Big] \rd t &= 0,\nonumber  			\\
		\int_0^s \Big[ \pair{\cL_i^2}{\rd_t^2i}{\rd_t i} + 2 \inner{\cL_i^2}{\Gamma \rd_t i}{\rd_t i} + \inner{\cL_i^2}{\Gamma^2 i}{\rd_t i} - e \inner{\cL_i^2}{\Upsilon \Gamma J_6 w}{\rd_t i} \label{eq:CurrentWithInt} \hspace{2.23cm} \\	 
		-	e \inner{\cL_i^2}{\Upsilon \Gamma \rN J_7 (f( \tilde{v}) - f( \hat{v}) )}{\rd_t i} \Big] \rd t &= 0,  \nonumber \\
		\int_0^s \Big[  \pair{\cH_w^1}{\rd_t^2 w}{q} + 2 \nu \inner{\cL_w^2}{\Lambda \rd_t w}{q} + \tfrac{3}{2} \nu^2  \inner{\cL_{\partial w}^2}{\dx w}{\dx q} + \nu^2 \inner{\cL_w^2}{\Lambda^2 w}{q} \label{eq:WaveWithInt}\hspace{1.5cm} \\ 
		- \nu^2 \inner{\cL_w^2}{\Lambda^2 \rM J_8 ( f(\tilde{v}) - f(\hat{v}) )}{q} \Big] \rd t &= 0. \nonumber
	\end{align}
	
	Next, integrating by parts in the first and second terms in \eqref{eq:WaveWithInt} yields
	\begin{multline*}
		\int_0^s \Big[ -\inner{\cL_w^2}{\rd_t w}{\rd_t q} - 2 \nu \inner{\cL_w^2}{\Lambda w}{\rd_t q} + \tfrac{3}{2} \nu^2  \inner{\cL_{\partial w}^2}{\dx w}{\dx q} + \nu^2 \inner{\cL_w^2}{\Lambda^2 w}{q}\Big] \rd t\\  
		=\int_0^s \nu^2 \inner{\cL_w^2}{\Lambda^2 \rM J_8 ( f(\tilde{v}) - f(\hat{v}) )}{q} \rd t. 
	\end{multline*}	
	Note that $\pair{\cH_w^1}{\rd_t w}{\rd_t q} = \inner{\cL_w^2}{\rd_t w}{\rd_t q}$ since $\rd_t w \in \cL_w^2$ for almost every $t \in [0,T]$; see the proof of \cite[Th. 5.9-1]{Evans:PDE:2010}.	
	Now, it follows from the definition of $q(t)$ that $\rd_t q = -w$ for all $t \in [0, s]$. Therefore,
	\begin{multline} \label{eq:52816754}
		\int_0^s \left[\tfrac{1}{2}\rd_t \left( \norm{\cL_w^2}{w}^2 - \tfrac{3}{2} \nu^2  \norm{\cL_{\partial w}^2}{\partial_x q}^2 \right) 
		+ 2 \nu \norm{\cL_w^2}{\Lambda^{\frac{1}{2}} w}^2  + \nu^2 \inner{\cL_w^2}{\Lambda^2 w}{q} \right] \rd t \\
		=\int_0^s \nu^2 \inner{\cL_w^2}{\Lambda^2 \rM J_8 ( f(\tilde{v}) - f(\hat{v}) )}{q}  \rd t. 
	\end{multline}	 
	Using Young's inequality, 
	\begin{align*}
		\nu^2 \inner{\cL_w^2}{\Lambda^2 \rM J_8 ( f(\tilde{v}) - f(\hat{v}) )}{q}
		&\leq  \tfrac{1}{4} \nu^2 \norm{\cL_w^2}{q}^2 +\nu^2 \trace(\Lambda^4 \rM^2 ) \left[\sup_{v_{\rE}(x,t) \in \bbR} |\partial_{v_{\rE}} f_{\rE}(v_{\rE})| \right]^2 \norm{\cL_v^2}{v}^2\\
		&\leq \tfrac{1}{4} \nu^2 \norm{\cL_w^2}{q}^2 + \tfrac{1}{8}\nu^2 \frac{\rF_{\rE}^2}{\sigma_{\rE}^2} \trace(\Lambda^4 \rM^2 ) \norm{\cL_v^2}{v}^2, \\
		-\nu^2 \inner{\cL_w^2}{\Lambda^2 w}{q} &\leq \tfrac{1}{4} \nu^2 \norm{\cL_w^2}{q}^2 + \nu^2 \norm{2}{\Lambda }^4 \norm{\cL_w^2}{w}^2,
	\end{align*}
	where the second inequality follows, for $\rX  = \rE$, from differentiating \eqref{eq:FiringRateFunction} as 
	\begin{equation} \label{eq:FireingFuncDerivative}
		\partial_{v_{\rX}} f_{\rX}(v_{\rX}) =  
		\dfrac{\sqrt{2}}{\sigma_{\rX}}\, \rF_{{\rX}}\exp\left(\! -\sqrt{2}\, \dfrac{v_{\rX} - \mu_{\rX} }{\sigma_{\rX}} \right)\! \left[ 1+\exp\left(\! -\sqrt{2}\, \dfrac{v_{\rX} - \mu_{\rX} }{\sigma_{\rX}} \right) \right]^{-2}\!\!, 
		\quad \rX \in \{\rE,\rI \},
	\end{equation}
	which implies $\sup_{v_{\rm X}(x,t)\in \bbR}|\partial_{v_{\rX}} f_{\rX}(v_{\rX})| \leq \frac{\rF_{\rX}}{2\sqrt{2} \sigma_{\rX}}$.
	
	Now, \eqref{eq:52816754} implies
	\begin{align*}
		\tfrac{1}{2} \norm{\cL_w^2}{w(s)}^2 + \tfrac{3}{4} \nu^2  \norm{\cH_w^1}{q(0)}^2 
		&\leq	\int_0^s \left[\Big(\!-2\nu \Lambda_{\min} + \nu^2 \norm{2}{\Lambda}^4 \Big)  \norm{\cL_w^2}{w}^2  
		+\tfrac{1}{2}\nu^2 \norm{\cL_w^2}{q}^2 \right.\\&\left.
		\quad + \tfrac{1}{8}\nu^2 \frac{\rF_{\rE}^2}{\sigma_{\rE}^2} \trace(\Lambda^4 \rM^2 ) \norm{\cL_v^2}{v}^2
		\right] \rd t
		+\tfrac{3}{4} \nu^2 \norm{\cL_w^2}{q(0)}^2. 
	\end{align*}
	where $\Lambda_{\min}:= \min\{ \Lambda_{\rE \rE}, \Lambda_{\rE \rI} \}$ is the smallest eigenvalue of $\Lambda$.
	Noting from \eqref{eq:pqDefinition} that $q(t) = p(s) - p(t)$ for all $t \in [0,s]$, it follows that the above inequality can be written as
	\begin{align*}
		\tfrac{1}{2} \norm{\cL_w^2}{w(s)}^2 + \tfrac{3}{4} \nu^2  \norm{\cH_w^1}{p(s)}^2
		&\leq	\int_0^s \left[\Big(\!-2\nu \Lambda_{\min} + \nu^2 \norm{2}{\Lambda}^4 \Big)  \norm{\cL_w^2}{w(t)}^2  
		+\tfrac{1}{2}\nu^2 \norm{\cL_w^2}{p(s)-p(t)}^2  \right. \\& \left.
		\quad + \tfrac{1}{8}\nu^2 \frac{\rF_{\rE}^2}{\sigma_{\rE}^2} \trace(\Lambda^4 \rM^2 ) \norm{\cL_v^2}{v(t)}^2
		\right] \rd t 
		+\tfrac{3}{4} \nu^2 \norm{\cL_w^2}{p(s)}^2. 
	\end{align*}	
	Using the Cauchy-Schwarz inequality, it follows from the definition of $p(t)$ given by \eqref{eq:pqDefinition} that
	$\norm{\cL_w^2}{p(s)}^2 \leq s \int_0^s \norm{\cL_w^2}{w(t)}^2 \rd t$.
	Moreover, 
	\begin{equation*}
		\norm{\cL_w^2}{p(s)-p(t)}^2  \leq 2 \norm{\cL_w^2}{p(s)}^2  + 2\norm{\cL_w^2}{p(t)}^2
		\leq 2 \norm{\cH_w^1}{p(s)}^2  + 2\norm{\cH_w^1}{p(t)}^2.
	\end{equation*} 
	Therefore,
	\begin{align} \label{eq:WaveWithIntBound}
		\tfrac{1}{2} \norm{\cL_w^2}{w(s)}^2 + \nu^2  (\tfrac{3}{4} - s)\norm{\cH_w^1}{p(s)}^2 
		&\leq	\int_0^s \left[\Big(\!-2\nu \Lambda_{\min} + \nu^2 \norm{2}{\Lambda}^4 +\tfrac{3}{4} \nu^2 s\Big)  \norm{\cL_w^2}{w(t)}^2  \right. \\& \left.
		\quad +\nu^2 \norm{\cH_w^1}{p(t)}^2 
		+ \tfrac{1}{8}\nu^2 \frac{\rF_{\rE}^2}{\sigma_{\rE}^2} \trace(\Lambda^4 \rM^2 ) \norm{\cL_v^2}{v(t)}^2
		\right] \rd t. \nonumber
	\end{align}	 
	
	Next, recalling \eqref{eq:BoundVoltage} and \eqref{eq:BoundCurrent} and using the Cauchy-Schwarz and Young inequalities, it follows that the fourth term in \eqref{eq:VoltageWithInt} satisfies, for every $\varepsilon_1 >0$, 
	\begin{align*}
		\inner{\cL_v^2}{ J_2 (\tilde{v} \tilde{i}^{\rT} - \hat{v} \hat{i}^{\rT}) \Psi J_4}{v} &= \inner{\cL_v^2}{ J_2 v \tilde{i}^{\rT} \Psi J_4}{v} + \inner{\cL_v^2}{ J_2 \hat{v} i^{\rT} \Psi J_4}{v} \\
		& \geq -\sqrt{2 \kappa_{\tilde{i}}} \; \norm{2}{\Psi} \norm{\cL_v^2}{v}^2 -\varepsilon_1 \norm{\cL_v^2}{v}^2 - \frac{2 \kappa_{\hat{v}}}{4 \varepsilon_1} \norm{2}{\Psi}^2 \norm{\cL_i^2}{i}^2,
	\end{align*} 
	where $\kappa_{\hat{v}}$ and $\kappa_{\tilde{i}}$ are in the form of \eqref{eq:KappaV} and \eqref{eq:KappaI}, respectively. The same inequality holds for $\inner{\cL_v^2}{ J_3 (\tilde{v} \tilde{i}^{\rT} - \hat{v} \hat{i}^{\rT}) \Psi J_5}{v}$. 
	Similarly, using Young's inequality and \eqref{eq:FireingFuncDerivative}, 
	\begin{align*}
		e \inner{\cL_i^2}{\Upsilon \Gamma \rN J_7 (f( \tilde{v}) - f( \hat{v}) )}{\rd_t i}
		&\leq  \varepsilon_2 \norm{\cL_i^2}{\rd_t i}^2 + \frac{e^2}{4 \varepsilon_2} \norm{2}{\Upsilon \Gamma \rN J_7}^2 \sup_{v(x,t) \in \bbR^2} \norm{2}{\partial_v f(v)}^2 \norm{\cL_v^2}{v}^2\\
		&\leq \varepsilon_2 \norm{\cL_i^2}{\rd_t i}^2 + \frac{e^2}{32 \varepsilon_2} \norm{2}{\Upsilon \Gamma \rN J_7}^2  \max \left \{
		\frac{\rF_{\rE}^2}{\sigma_{\rE}^2} , \frac{\rF_{\rI}^2}{\sigma_{\rI}^2} \right \} \norm{\cL_v^2}{v}^2,
	\end{align*}
	for every $\varepsilon_2>0$. 
	Moreover, for every $\varepsilon_3 >0$ and $\varepsilon_4 >0$,
	\begin{align*}
		\inner{\cL_v^2}{J_1 i}{v} &\leq \varepsilon_4 \norm{\cL_v^2}{v}^2 + \frac{1}{2\varepsilon_4} \norm{\cL_i^2}{i}^2, \\
		e \inner{\cL_i^2}{\Upsilon \Gamma J_6 w}{\rd_t i} &\leq \varepsilon_4 \norm{\cL_i^2}{\rd_t i}^2 +  \frac{e^2}{4\varepsilon_4} \norm{2}{\Upsilon \Gamma J_6}^2 \norm{\cL_w^2}{w}^2.
	\end{align*}
	
	Substituting the above inequalities into \eqref{eq:VoltageWithInt} and \eqref{eq:CurrentWithInt}, and adding the resulting inequalities to \eqref{eq:WaveWithIntBound} yields, for some $\alpha >0$,
	\begin{multline*}
		\norm{\cL_v^2}{\Phi^{\frac{1}{2}} v(s)}^2 + \norm{\cL_i^2}{\rd_t i(s)}^2 + \norm{\cL_i^2}{\Gamma i(s)}^2 
		+\norm{\cL_w^2}{w(s)}^2 + \nu^2  (\tfrac{3}{2} - 2s)\norm{\cH_w^1}{p(s)}^2  \\
		\leq
		\alpha \int_0^s \left[ \norm{\cL_v^2}{v(t)}^2 + \norm{\cL_i^2}{\rd_t i(t)}^2 + \norm{\cL_i^2}{i(t)}^2 + \norm{\cL_w^2}{ w(t)}^2 + \norm{\cH_w^1}{p(t)}^2\right] \rd t.
	\end{multline*}
	Now, setting $T_1 = \frac{3}{4}$, it follows from the integral form of Gr\"{o}nwall's inequality \cite[Appx. B.2]{Evans:PDE:2010} that $(v(s),i(s),w(s)) = 0$ for all $s\in[0,T_1]$. 
	Repeating the same arguments for intervals $[T_1,2T_1]$, $[2T_1,3T_1]$, $\dots$, we deduce $(v(t),i(t),w(t)) = 0$ for all $t\in[0,T]$, and hence, $(\tilde{v},\tilde{i},\tilde{w}) = (\hat{v},\hat{i},\hat{w})$ for all $t\in[0,T]$, which is a contradiction and completes the proof of uniqueness.  
	\qquad
\end{proof}	

\begin{proposition}[Regularity of weak solutions] \label{prp:RegularityWeakSolution}
	Suppose that the assumptions of Theorem \ref{th:WeakExistenceUniqueness} hold, namely, $g \in L^2(0,T;\cL_i^2)$, $v_0 \in \cL_v^2$, $i_0 \in \cL_i^2$, $i'_0 \in \cL_i^2$, $w_0 \in \cH_w^1$, and $w'_0 \in \cL_w^2$. 
	Then the $\Omega$-periodic weak solution $(v,i,w)$ of the initial value problem \eqref{eq:Voltage}--\eqref{eq:InitialValues} satisfies 
	\begin{align}  \label{eq:WeakProperties1}
		\esssup_{t\in [0,T]}\left( \norm{\cL_v^2}{v(t)}^2 \right) + \norm{L^2(0,T;\cL_v^2)}{\rd_t v}^2 &\leq \kappa_v, \\
		\esssup_{t\in [0,T]}\left(\norm{\cL_i^2}{\rd_t i(t)}^2 + \norm{\cL_i^2}{i(t)}^2 \right) 
		+\norm{L^2(0,T; \cL_i^2)}{\rd_t^2 i}^2 &\leq \kappa_i, \nonumber \\
		\esssup_{t\in[0,T]} \left(\norm{\cL_w^2}{\rd_t w(t)}^2  + \norm{\cH_w^1}{w(t)}^2 \right) 
		+ \norm{L^2(0,T;{\cH^1_w}^{\!\!\ast})}{\rd_t^2 w}^2 &\leq \kappa_w,\nonumber  
	\end{align}
	\begin{alignat}{3} \label{eq:WeakProperties2}
		v &\in H^1(0,T;\cL_v^2) \cap C^2([0,T];\cL_v^2), &&\\
		i &\in H^2(0,T;\cL_i^2)  \cap C^{1,\frac{1}{2}}([0,T];\cL_i^2), &\quad
		\rd_t i &\in H^1(0,T;\cL_i^2)  \cap C^{0,\frac{1}{2}}([0,T];\cL_i^2), \nonumber\\
		w &\in H^1(0,T;\cL_w^2)   \cap C^0([0,T];\cH_w^1), &\quad
		\rd_t w &\in C^0([0,T];\cL_w^2), \nonumber
	\end{alignat}
	where $\kappa_v$, $\kappa_i$, and $\kappa_w$ are given by \eqref{eq:KappaV}--\eqref{eq:KappaW}.
	Moreover, if $g \in C^0([0,T];\cL_i^2)$, then
	\begin{equation} \label{eq:WeakProperties3}
		v \in  C^3([0,T];\cL_v^2), \quad i \in C^2([0,T];\cL_i^2), \quad \rd_t i \in C^1([0,T];\cL_i^2),
	\end{equation}
	and if $g \in C^1([0,T];\cL_i^2)$, then
	\begin{equation} \label{eq:WeakProperties4}
		v \in  C^4([0,T];\cL_v^2), \quad i \in C^3([0,T];\cL_i^2), \quad \rd_t i \in C^2([0,T];\cL_i^2).
	\end{equation}
\end{proposition}
\begin{proof}
	First, recall that $\cL_v^2={\cL^2_v}^{\ast}$ and $\cL_i^2={\cL^2_i}^{\ast}$. 
	Assertion \eqref{eq:WeakProperties1} follows immediately from \eqref{eq:BoundVoltage}--\eqref{eq:BoundWave} by setting $m=m_k$ and passing to the limits through \eqref{eq:StrongConvergence} and \eqref{eq:WeakConvergence}.
	The inclusions in $H^1$ and $H^2$ 
	in assertion \eqref{eq:WeakProperties2} are immediate from \eqref{eq:WeakProperties1}. 
	The Sobolev embedding theorems \cite[Th. 6.6-1]{Ciarlet:FunctionalAnalysis:2013} applied to Banach space-valued functions on $[0,T] \subset \bbR$ imply that	
	$v \in C^{0,\frac{1}{2}}([0,T];\cL_v^2)$, $i \in C^{1,\frac{1}{2}}([0,T];\cL_i^2)$, and $\rd_t i \in C^{0,\frac{1}{2}}([0,T];\cL_i^2)$, which further imply by \eqref{eq:Voltage} that $v \in C^2([0,T];\cL_v^2)$.
	
	Consider the time-independent self-adjoint linear operator 
	$A:=(-\frac{3}{2} \nu^2 \Delta + I):\cH_w^1 {\rightarrow \cH_w^1}^{\!\!\ast}$. 
	Note that $f(v) \in C^2([0,T];\cL_v^{\infty})$, since $f$ is a bounded smooth function and $v \in C^2([0,T];\cL_v^2)$. 
	Then, it follows from \eqref{eq:Wave} and \eqref{eq:WeakProperties1} that 
	$\rd_t^2 w + A w \in L^2(0,T;\cL_w^2)$.
	Therefore, by \cite[Lemma II.4.1]{Temam:InfiniteDimensional:1997} we have 
	$w \in C^0([0,T];\cH_w^1)$ and $\rd_t w \in C^0([0,T];\cL_w^2)$, which completes the proof of  \eqref{eq:WeakProperties2}. 
	Assertions \eqref{eq:WeakProperties3} and \eqref{eq:WeakProperties4} are now immediate from \eqref{eq:Voltage}, \eqref{eq:Current}, and \eqref{eq:WeakProperties2}.
	\qquad
\end{proof}

\begin{theorem}[Existence and uniqueness of strong solutions] \label{th:StrongExistenceUniqueness}
	Suppose that 
	$g \in L^2(0,T;\cL_i^2)$, $v_0 \in \cL_v^2$, $i_0 \in \cL_i^2$, $i'_0 \in \cL_i^2$, $w_0 \in \cH_w^2$, and $w'_0 \in \cH_w^1$. 
	Then there exists a unique $\Omega$-periodic strong solution $(v,i,w)$ of the initial value problem \eqref{eq:Voltage}--\eqref{eq:InitialValues}.  
\end{theorem}
\begin{proof}
	Uniqueness follows immediately from Theorem \ref{th:WeakExistenceUniqueness} since every strong solution of 
	\eqref{eq:Voltage}--\eqref{eq:InitialValues}
	is also a weak solution. 
	Moreover, Proposition \ref{prp:RegularityWeakSolution} implies that the weak solutions $v \in H^1(0,T;\cL_v^2)$ and $i \in H^2(0,T;\cL_i^2)$ are indeed strong solutions as given in Definition \ref{def:StrongSolution}.
	It remains to prove the regularity required for $w$ by Definition \ref{def:StrongSolution}. 
	
	Consider \eqref{eq:WaveApprx} with the approximation \eqref{eq:WaveExpansion}, 
	let $\cB_w = \seq{h_w^{(k)}}_{k=1}^{\infty}$ be the orthogonal basis of $\cH_w^1$ consisting of the eigenfunctions of $A:=-\Delta +I$ 
	as given by Lemma \ref{lem:DualOrthogonalBasis}, and let $\lambda_k$ denote the eigenvalue corresponding to the eigenfunction $h_w^{(k)}$. Multiplying \eqref{eq:WaveApprx} by $\lambda_k c_{w_k}^{(m)}$ and summing over $k=1, \dots, m$ yields
	\begin{multline*}
		\inner{\cL_w^2}{\rd_t^2 w^{(m)}}{A w^{(m)}} + 2 \nu \inner{\cL_w^2}{\Lambda \rd_t w^{(m)}}{A w^{(m)}} + \tfrac{3}{2} \nu^2  \inner{\cL_{\partial w}^2}{\dx w^{(m)}}{A \dx w^{(m)}}\\
		+ \nu^2 \inner{\cL_w^2}{\Lambda^2 w^{(m)}}{A w^{(m)}} 	
		- \nu^2 \inner{\cL_w^2}{\Lambda^2 \rM J_8 f(v^{(m)} )}{A w^{(m)}} =0. 
	\end{multline*}
	Now, Young's inequality implies that, for every $\varepsilon_1, \dots, \varepsilon_4>0$,
	\begin{align*}
		-\inner{\cL_w^2}{\rd_t^2 w^{(m)}}{A w^{(m)}} &\leq \varepsilon_1 \norm{\cL_w^2}{A w^{(m)}}^2 + \frac{1}{4\varepsilon_1} \norm{\cL_w^2}{\rd_t^2 w^{(m)}}^2,\\
		-\inner{\cL_w^2}{\Lambda \rd_t w^{(m)}}{A w^{(m)}} &\leq \varepsilon_2 \norm{\cL_w^2}{A w^{(m)}}^2 + \frac{1}{4\varepsilon_2} \norm{\cL_w^2}{\Lambda \rd_t w^{(m)}}^2, \\
		-\inner{\cL_w^2}{\Lambda^2 w^{(m)}}{A w^{(m)}} &\leq \varepsilon_3 \norm{\cL_w^2}{A w^{(m)}}^2 + \frac{1}{4\varepsilon_3} \norm{\cL_w^2}{\Lambda^2 w^{(m)}}^2, \\
		\inner{\cL_w^2}{\Lambda^2 \rM J_8 f(v^{(m)} )}{A w^{(m)}} &\leq 
		\varepsilon_4 \norm{\cL_w^2}{A w^{(m)}}^2 + \frac{1}{4\varepsilon_4}\norm{\cL_w^2}{\Lambda^2 \rM J_8 f(v^{(m)} )}^2  \\
		&\leq  \varepsilon_4 \norm{\cL_w^2}{A w^{(m)}}^2 +\frac{1}{4\varepsilon_4} |\Omega| \rF_{\rE}^2 \trace (\Lambda^4 \rM^2).
	\end{align*}	
	Therefore, using \eqref{eq:DeltaFormulas},
	\begin{align*}
		\tfrac{3}{2} \nu^2 \norm{\cH_w^2}{w^{(m)}}^2 & \leq (\varepsilon_1 + 2\nu \varepsilon_2+ \nu^2 \varepsilon_3+ \nu^2 \varepsilon_4) \left( \norm{\cH_w^2}{w^{(m)}}^2 + \norm{\cL_{\partial w}^2}{\dx w^{(m)}}^2 \right)\\
		&\quad + \tfrac{3}{2} \nu^2 \norm{\cL_w^2}{w^{(m)}}^2 +  \frac{1}{4\varepsilon_1} \norm{\cL_w^2}{\rd_t^2 w^{(m)}}^2 +\frac{\nu}{2\varepsilon_2} \norm{\cL_w^2}{\Lambda \rd_t w^{(m)}}^2 \\
		&\quad + \frac{\nu^2}{4\varepsilon_3} \norm{\cL_w^2}{\Lambda^2 w^{(m)}}^2+ \frac{\nu^2}{4\varepsilon_4} |\Omega| \rF_{\rE}^2 \trace (\Lambda^4 \rM^2). 
	\end{align*}
	
	Next, set $\varepsilon_1=\frac{\nu^2}{8}$, $\varepsilon_2=\frac{\nu}{16}$, $\varepsilon_3 = \frac{1}{8}$, and $\varepsilon_4=\frac{1}{8}$, and note that, for some constant $\beta>0$,
	\begin{equation} \label{eq:BoundH2}
		\norm{\cH_w^2}{w^{(m)}}^2 \leq \beta \left( \norm{\cL_w^2}{\rd_t^2 w^{(m)}}^2 + \norm{\cL_w^2}{\rd_t w^{(m)}}^2 + \norm{\cH_w^1}{w^{(m)}}^2  + |\Omega| \rF_{\rE}^2 \trace (\Lambda^4 \rM^2) \right).
	\end{equation}
	Bounds on $\norm{\cL_w^2}{\rd_t w^{(m)}}$ and $\norm{\cH_w^1}{w^{(m)}}$ are given by the energy estimate \eqref{eq:BoundWave}. 
	To establish bounds on $\norm{\cL_w^2}{\rd_t^2 w^{(m)}}$ and $\norm{\cH_w^1}{\rd_t w^{(m)}}$, consider \eqref{eq:WaveApprx} with the initial values given in \eqref{eq:ApprxInitial}. Differentiating \eqref{eq:WaveApprx} with respect to $t$, multiplying the result by $\rd_t^2 c_{w_k}^{(m)}$, and summing over $k=1,\dots m$, yields 
	\begin{multline*}
		\inner{\cL_w^2}{\rd_t^2 \dot{w}^{(m)}}{\rd_t \dot{w}^{(m)}} + 2 \nu \inner{\cL_w^2}{\Lambda \rd_t \dot{w}^{(m)}}{\rd_t \dot{w}^{(m)}} + \tfrac{3}{2} \nu^2  \inner{\cL_{\partial w}^2}{\dx \dot{w}^{(m)}}{\rd_t \dx \dot{w}^{(m)}}  \\
		+ \nu^2 \inner{\cL_w^2}{\Lambda^2 \dot{w}^{(m)}}{\rd_t \dot{w}^{(m)}} 	- \nu^2 \inner{\cL_w^2}{\Lambda^2 \rM J_8 \,\rd_t f(v^{(m)} )}{\rd_t \dot{w}^{(m)}} =0, 
	\end{multline*}	
	where $\dot{w}:= \rd_t w$ and
	$\rd_t f_{\rE}( v_{\rE}^{(m)} ) = \partial_{v_{\rE}} f_{\rE}(v_{\rE}^{(m)}) \, \rd_t v_{\rE}^{(m)}$.  
	Now, \eqref{eq:FireingFuncDerivative} with $\rX = \rE$ gives
	\begin{align} \label{eq:BoundOnDiffFiringFunc}
		\norm{\cL_w^2}{\Lambda^2 \rM J_8 \, \rd_t f(v^{(m)} )}^2  
		&=\trace (\Lambda^4 \rM^2)  \int_{\Omega} \big |\rd_t f_{\rE}(v_{\rE}^{(m)}) \big |^2 \rd x \\
		&\leq \trace (\Lambda^4 \rM^2) \frac{\rF_{\rE}^2}{8 \sigma_{\rE}^2} \int_{\Omega} \big |\rd_t v_{\rE}^{(m)} \big |^2 \rd x
		\leq \trace (\Lambda^4 \rM^2) \frac{\rF_{\rE}^2}{8 \sigma_{\rE}^2} \norm{\cL_v^2}{\rd_t v^{(m)}}^2. \nonumber
	\end{align}
	Using similar arguments as in the proof of Proposition \ref{prp:EnergyEstimates}, it follows from the above inequality and Young's inequality that, for every $\varepsilon >0$,
	\begin{multline*}
		\rd_t \left[ \norm{\cL_w^2}{\rd_t \dot{w}^{(m)}}^2  + \tfrac{3}{2} \nu^2  \norm{\cL_{\partial w}^2}{\dx \dot{w}^{(m)}}^2 + \nu^2 \norm{\cL_w^2}{\Lambda \dot{w}^{(m)}}^2 \right]	
		+ 2\nu(2 \Lambda_{\min} - \varepsilon \nu) \norm{\cL_w^2}{\rd_t \dot{w}^{(m)}}^2   \\ 
		\leq \frac{\nu^2}{2\varepsilon} \frac{\rF_{\rE}^2}{8 \sigma_{\rE}^2}\trace (\Lambda^4 \rM^2) \norm{\cL_v^2}{\rd_t v^{(m)}}^2, 
	\end{multline*}	
	where $\Lambda_{\min}:= \min \{\Lambda_{\rE \rE}, \Lambda_{\rE \rI}\}$ is the smallest eigenvalue of $\Lambda$.
	Next, setting $\varepsilon = \frac{2}{\nu}\Lambda_{\min}$, replacing $\dot{w}= \rd_t w$, and using Gr\"{o}nwall's inequality yields
	\begin{align} \label{eq:1214151119}
		\norm{\cL_w^2}{\rd_t^2 w^{(m)}(t)}^2  &+ \tfrac{3}{2} \nu^2  \norm{\cL_{\partial w}^2}{\rd_t \dx w^{(m)}(t)}^2 + \nu^2 \norm{\cL_w^2}{\Lambda \rd_t w^{(m)}(t)}^2 \\
		& \leq \left( \norm{\cL_w^2}{\rd_t^2 w^{(m)}}^2  + \tfrac{3}{2} \nu^2  \norm{\cL_{\partial w}^2}{\rd_t \dx w^{(m)}}^2 + \nu^2 \norm{\cL_w^2}{\Lambda \rd_t w^{(m)}}^2 \right)  \! \Big|_{t=0} \nonumber\\
		&\quad  +  \tfrac{1}{32}\frac{\nu^3}{\Lambda_{\min} \sigma^2} \rF_{\rE}^2 \trace (\Lambda^4 \rM^2) \norm{L^2(0,T;\cL_v^2)}{\rd_t v^{(m)}}^2. \nonumber
	\end{align}
	
	Finally, it follows from \eqref{eq:WaveApprx} and \eqref{eq:ApprxInitial} that, for some $\alpha_1>0$,  
	\begin{equation*}
		\norm{\cL_w^2}{\rd_t^2 w^{(m)}}^2\Big|_{t=0} \leq \alpha_1 \left( \norm{\cH_w^1}{w'_0}^2 + \norm{\cH_w^2}{w_0}^2 + 
		\nu^2 |\Omega| \rF_{\rE}^2 \trace (\Lambda^4 \rM^2) \right).
	\end{equation*}
	Now, using the energy estimate \eqref{eq:BoundVoltage} and the above inequality in \eqref{eq:1214151119} it follows that 
	\begin{equation*}
		\norm{\cL_w^2}{\rd_t^2 w^{(m)}(t)}^2  + \norm{\cH_w^1}{\rd_t w^{(m)}(t)}^2 
		\leq \alpha_2 \left( \norm{\cH_w^1}{w_0'}^2  + \norm{\cH_w^2}{w_0}^2 + 
		(|\Omega| + \kappa_v) \rF_{\rE}^2  \right)
	\end{equation*}
	for some $\alpha_2 > 0$ and all $t\in[0,T]$. Since this inequality and \eqref{eq:BoundH2} hold for all $t\in[0,T]$, it follows that
	\begin{align} \label{eq:BoundWaveH2Hat}
		\sup_{t\in[0,T]} \left(\norm{\cL_w^2}{\rd_t^2 w^{(m)}(t)}^2 + \norm{\cH_w^1}{\rd_t w^{(m)}(t)}^2  + \norm{\cH_w^2}{w^{(m)}(t)}^2 \right) \leq \hat{\beta}_w,
	\end{align}
	where
	\begin{align*}
		\hat{\beta}_w := \alpha \left(
		\norm{\cH_w^1}{w_0'}^2  + \norm{\cH_w^2}{w_0}^2 + (|\Omega| + \kappa_v) \rF_{\rE}^2 \right)
	\end{align*}
	for some $\alpha >0$. Now, using the above estimate and passing to the limits, the result follows by similar arguments as in the proof of Theorem \ref{th:WeakExistenceUniqueness}.
	\qquad
\end{proof}

\begin{proposition} [Regularity of strong solutions] \label{prp:RegularityStrongSolution}
	Suppose that the assumptions of Theorem \ref{th:StrongExistenceUniqueness} hold, namely, $g \in L^2(0,T;\cL_i^2)$, $v_0 \in \cL_v^2$, $i_0 \in \cL_i^2$, $i'_0 \in \cL_i^2$, $w_0 \in \cH_w^2$, and $w'_0 \in \cH_w^1$. Then, in addition to the properties of the weak solution given in Proposition \ref{prp:RegularityWeakSolution}, the $\Omega$-periodic strong solution $(v,i,w)$ of the initial value problem \eqref{eq:Voltage}--\eqref{eq:InitialValues} satisfies 
	\begin{align}  \label{eq:StrongProperties1}
		\esssup_{t\in[0,T]} \left(\norm{\cL_w^2}{\rd_t^2 w(t)}^2 + \norm{\cH_w^1}{\rd_t w(t)}^2  + \norm{\cH_w^2}{w(t)}^2 \right) + \norm{L^2(0,T;{\cH^1_w}^{\!\!\ast})}{\rd_t^3 w}^2 \leq \beta_w,  
	\end{align}
	\begin{align} \label{eq:StrongProperties2}
		w &\in H^2(0,T;\cL_w^2) \cap H^1(0,T;\cH_w^1)  \cap C^{1,\frac{1}{2}}([0,T];\cL_w^2) \cap C^{0,\frac{1}{2}}([0,T];\cH_w^1)\\
		&\hspace{5.23cm }  \cap C^{0}([0,T];\cH_w^2) \cap C^0([0,T];C_{\rm per}^{0,\lambda}(\overline{\Omega}, \bbR^2)),   \nonumber\\
		\rd_t w &\in H^1(0,T;\cL_w^2) \cap C^{0,\frac{1}{2}}([0,T];\cL_w^2) \cap C^{0}([0,T];\cH_w^1),   \nonumber\\
		\rd_t^2 w &\in C^{0}([0,T];\cL_w^2), \nonumber
	\end{align}
	for all $\lambda \in (0,1)$ and some $\beta_w >0$.
\end{proposition}
\begin{proof}
	Differentiate \eqref{eq:WaveApprx} with respect to $t$ and denote $\dot{w}:= \rd_t w$. Use \eqref{eq:BoundOnDiffFiringFunc} and follow the same steps used to prove \eqref{eq:BoundWave} in Proposition \ref{prp:EnergyEstimates} to show $\norm{L^2(0,T;{\cH^1_w}^{\!\!\ast})}{\rd_t^2 \dot{w}^{(m)}}^2 \leq \tilde{\beta}_w$ for every positive integer $m$, all $t\in[0,T]$, and some $\tilde{\beta}_w>0$ proportional to $\hat{\beta}_w$ in \eqref{eq:BoundWaveH2Hat}. 
	Replacing $\dot{w}= \rd_t w$, adding the result to \eqref{eq:BoundWaveH2Hat}, and passing to the limits establishes \eqref{eq:StrongProperties1} for some $\beta_w>0$  proportional to $\hat{\beta}_w$.
	
	The inclusions in $H^1$ and $H^2$ in assertion \eqref{eq:StrongProperties2} follow immediately from \eqref{eq:StrongProperties1}. 
	The inclusions in the H\"{o}lder spaces $C^{0, \frac{1}{2}}$ and $C^{1, \frac{1}{2}}$ are implied by the Sobolev embedding theorems \cite[Th. 6.6-1]{Ciarlet:FunctionalAnalysis:2013} applied to Banach space-valued functions on $[0,T] \subset \bbR$. 
	
	To show $\rd_t w \in C^{0}([0,T];\cH_w^1)$ and $\rd_t^2 w \in C^{0}([0,T];\cL_w^2)$, consider the time-independent self-adjoint linear operator $A:=(-\frac{3}{2} \nu^2 \Delta + I):\cH_w^1 {\rightarrow \cH_w^1}^{\!\!\ast}$. 
	Differentiate \eqref{eq:Wave} with respect to $t$ and denote $\dot{w}:= \rd_t w$.
	Note that $\rd_t f(v) \in C^1([0,T];\cL_v^{\infty})$ since $\partial_v f$ is a bounded smooth function and $\rd_t v \in C^1([0,T];\cL_v^2)$, given by Proposition \ref{prp:RegularityWeakSolution}. 
	Then, it follows from \eqref{eq:Wave} and \eqref{eq:StrongProperties1} that 
	$\rd_t^2 \dot{w} + A \dot{w} \in L^2(0,T;\cL_w^2)$.
	Therefore, by \cite[Lemma II.4.1]{Temam:InfiniteDimensional:1997} we have 
	$\dot{w} \in C^0([0,T];\cH_w^1)$ and $\rd_t \dot{w} \in C^0([0,T];\cL_w^2)$. 
	
	Next, noting that $f(v) \in C^2([0,T];\cL_v^{\infty})$, $ w\in C^{1,\frac{1}{2}}([0,T];\cL_w^2)$, 
	$\rd_t w \in C^{0,\frac{1}{2}}([0,T];\cL_w^2)$, and $\rd_t^2 w \in C^{0}([0,T];\cL_w^2)$, it follows from
	\eqref{eq:Wave} that $(-\Delta + I) w \in C^{0}([0,T];\cL_w^2)$, and hence, $w \in C^{0}([0,T];\cH_w^2)$. 
	Moreover, using the Sobolev embedding theorems applied to $\Omega$-periodic functions in $\bbR^2$, this further implies that $w \in C^0([0,T];C_{\rm per}^{0,\lambda}(\overline{\Omega}, \bbR^2))$.
	\qquad
\end{proof}

Other than the regularity properties given in Propositions \ref{prp:RegularityWeakSolution} and \ref{prp:RegularityStrongSolution}, boundedness of weak and strong solutions associated with bounded input functions $g$ can also be established. 
We defer this result to Section \ref{sec:BiophysicalPhaseSpaces}, as a corollary of Proposition \ref{prp:WaveNonnegativity}.

In the remainder of the paper, as suggested in \cite[Sec. 11.1.2]{Robinson:InfiniteDimensional:2001}, we give formal arguments for some of the proofs, in the sense that we take the inner product of \eqref{eq:Wave} with functions that belong to $\cL_w^2$, instead of functions belonging to $\cH_w^1$ as required for the test functions $h_w$ in \eqref{eq:WeakWave}. 
However, the proofs can be made rigorous using the Galerkin approximation technique based on the dual orthogonal basis of $\cH_w^1 \Subset \cL_w^2$ and then passing to the limits, as in the proofs of Theorems \ref{th:WeakExistenceUniqueness} and \ref{th:StrongExistenceUniqueness}. 
See the discussion and results in \cite[Sec. 11.1.2]{Robinson:InfiniteDimensional:2001} for further details.

\section{Semidynamical Systems and Biophysical Plausibility of the Evolution} \label{sec:SemidynamicalSystems}
In this section, we establish a semidynamical system framework for the initial-value problem presented in Section \ref{sec:ExistenceUniqueness}. 
Assume $g \in L^2(0,\infty;\cL_i^2)$ and let $u(t):=(v(t),i(t),\rd_t i(t), w(t),\rd_t w(t))$ denote a solution of \eqref{eq:Voltage}--\eqref{eq:Wave} with the initial value $u_0:=u(0)=(v_0, i_0, i'_0, w_0, w'_0)$. 
Recall the  Definitions \ref{def:WeakSolution} and \ref{def:StrongSolution} and the results of Theorems \ref{th:WeakExistenceUniqueness} and \ref{th:StrongExistenceUniqueness} to note that the Hilbert spaces 
\begin{align} \label{eq:PhaseSpaceLarge}
	\cU_{\rw} &:= \cL_v^2 \times \cL_i^2 \times \cL_{i} \times \cH_w^1 \times \cL_w^2, \\
	\cU_{\rs} &:= \cL_v^2 \times \cL_i^2 \times \cL_{i} \times \cH_w^2 \times \cH_w^1, \nonumber
\end{align} 
construct, respectively, the phase spaces associated with the weak and strong solutions.
Now, for every $t\in[0,\infty)$, define the mappings 
\begin{alignat*}{3}
	S_{\rw}(t)&:\cU_{\rw} \rightarrow \cU_{\rw}, &\quad  S_{\rw}(t) u_0 &:= u(t), \\
	S_{\rs}(t)&:\cU_{\rs} \rightarrow \cU_{\rs}, &\quad  S_{\rs}(t) u_0 &:= u(t).
\end{alignat*}
The existence and uniqueness of solutions given by Theorems \ref{th:WeakExistenceUniqueness} and
\ref{th:StrongExistenceUniqueness} along with the time-continuity of solutions given by Propositions
\ref{prp:RegularityWeakSolution} and \ref{prp:RegularityStrongSolution} imply that the above mappings are well-defined for all $t\in[0,\infty)$. 
Then, $\seq{S_{\rw}(t)}_{t \in [0,\infty)}$ and $\seq{S_{\rs}(t)}_{t \in [0,\infty)}$ form semigroups of operators which give the weak and strong solutions of \eqref{eq:Model}, respectively. 
The following propositions show that these semigroups are continuous, which also ensures that the initial-value problems of finding weak and strong solutions for \eqref{eq:Model} are well-posed.
\begin{proposition}[Continuity of the semigroup $\{S_{\rw}\}$] \label{prp:SemigroupC0Weak}
	The semigroup $\seq{S_{\rw}(t)}_{t \in [0,\infty)}$ of weak solution operators is continuous for all $g \in L^2(0,\infty;\cL_i^2)$.
\end{proposition}
\begin{proof}
	Continuity of the semigroup with respect to $t$ follows immediately from the continuity of the weak solutions given in Proposition \ref{prp:RegularityWeakSolution}. 
	It remains to prove continuous dependence of the solution on the initial values. 
	Let $\tilde{u}_0$ and $\hat{u}_0$ be any two initial values in $\cU_{\rw}$ that give the solutions $\tilde{u}(t) = S_{\rw}(t)\tilde{u}_0$ and $\hat{u}(t)=S_{\rw}(t)\hat{u}_0$ for all $t\in [0,T]$, $T>0$. 
	Let $u(t) := \tilde{u}(t) -\hat{u}(t)$ be the weak solution with the initial value $u_0:=\tilde{u}_0 - \hat{u}_0$.
	Now, consider \eqref{eq:Voltage}--\eqref{eq:Wave} satisfied by $\tilde{u}$ and $\hat{u}$, and take the inner product of \eqref{eq:Voltage}--\eqref{eq:Wave} in each set with $v$, $\rd_t i$, and $\rd_t w$, respectively. Subtracting the resulting two sets of equations yields
	\begin{align}
		\inner{\cL_v^2}{\Phi \rd_t v}{v} + \inner{\cL_v^2}{v}{v} - \inner{\cL_v^2}{J_1 i}{v} \label{eq:ContinuityVoltageS1} \hspace{6.85cm}\\
		+ \inner{\cL_v^2}{ J_2 (\tilde{v} \tilde{i}^{\rT} - \hat{v} \hat{i}^{\rT}) \Psi J_4 + J_3 (\tilde{v} \tilde{i}^{\rT} -\hat{v} \hat{i}^{\rT}) \Psi J_5}{v} &= 0,\nonumber  			\\
		\inner{\cL_i^2}{\rd_t^2i}{\rd_t i} + 2 \inner{\cL_i^2}{\Gamma \rd_t i}{\rd_t i} + \inner{\cL_i^2}{\Gamma^2 i}{\rd_t i} - e \inner{\cL_i^2}{\Upsilon \Gamma J_6 w}{\rd_t i} \label{eq:ContinuityCurrentS1} \hspace{2.45 cm} \\	 
		-	e \inner{\cL_i^2}{\Upsilon \Gamma \rN J_7 (f( \tilde{v}) - f( \hat{v}) )}{\rd_t i} &= 0,  \nonumber \\
		\inner{\cL_w^2}{\rd_t^2 w}{\rd_t w} + 2 \nu \inner{\cL_w^2}{\Lambda \rd_t w}{\rd_t w} + \tfrac{3}{2} \nu^2  \inner{\cL_{\partial w}^2}{\dx w}{\rd_t \dx w} + \nu^2 \inner{\cL_w^2}{\Lambda^2 w}{\rd_t w} \label{eq:ContinuityWaveS1} \hspace{0cm}\\ 
		- \nu^2 \inner{\cL_w^2}{\Lambda^2 \rM J_8 ( f(\tilde{v}) - f(\hat{v}) )}{\rd_t w} &=0. \nonumber 
	\end{align}
	
	As in the proof of uniqueness given in Theorem \ref{th:WeakExistenceUniqueness},
	\begin{align} \label{eq:ContinuityMajorization}
		-\inner{\cL_v^2}{ J_2 (\tilde{v} \tilde{i}^{\rT} - \hat{v} \hat{i}^{\rT}) \Psi J_4}{v} 
		& \leq \sqrt{2 \kappa_{\tilde{i}}} \; \norm{2}{\Psi} \norm{\cL_v^2}{v}^2 + \norm{\cL_v^2}{v}^2 + \tfrac{1}{2}\kappa_{\hat{v}} \norm{2}{\Psi}^2 \norm{\cL_i^2}{i}^2,\\
		-\inner{\cL_v^2}{ J_3 (\tilde{v} \tilde{i}^{\rT} - \hat{v} \hat{i}^{\rT}) \Psi J_5}{v} 
		& \leq \sqrt{2 \kappa_{\tilde{i}}} \; \norm{2}{\Psi} \norm{\cL_v^2}{v}^2 + \norm{\cL_v^2}{v}^2 + \tfrac{1}{2}\kappa_{\hat{v}} \norm{2}{\Psi}^2 \norm{\cL_i^2}{i}^2, \nonumber\\
		e \inner{\cL_i^2}{\Upsilon \Gamma \rN J_7 (f( \tilde{v}) - f( \hat{v}) )}{\rd_t i}
		&\leq \norm{\cL_i^2}{\rd_t i}^2 + \tfrac{1}{32}e^2 \norm{2}{\Upsilon \Gamma \rN J_7}^2  \max \left \{
		\frac{\rF_{\rE}^2}{\sigma_{\rE}^2} , \frac{\rF_{\rI}^2}{\sigma_{\rI}^2} \right \} \norm{\cL_v^2}{v}^2, \nonumber\\
		\nu^2 \inner{\cL_w^2}{\Lambda^2 \rM J_8 ( f(\tilde{v}) - f(\hat{v}) )}{\rd_t w} 
		&\leq \nu^2 \norm{\cL_w^2}{\rd_t w}^2 + \tfrac{1}{32}\nu^2 \frac{\rF_{\rE}^2}{\sigma_{\rE}^2} \,  \trace(\Lambda^4 \rM^2 ) \norm{\cL_v^2}{v}^2, \nonumber\\
		\inner{\cL_v^2}{J_1 i}{v} 
		&\leq \norm{\cL_v^2}{v}^2 + \tfrac{1}{2} \norm{\cL_i^2}{i}^2, \nonumber\\
		e \inner{\cL_i^2}{\Upsilon \Gamma J_6 w}{\rd_t i} 
		&\leq  \norm{\cL_i^2}{\rd_t i}^2 +  \tfrac{1}{4} e^2 \norm{2}{\Upsilon \Gamma J_6}^2 \norm{\cL_w^2}{w}^2, \nonumber
	\end{align} 
	where $\kappa_{\hat{v}}$ and $\kappa_{\tilde{i}}$ are in the form of \eqref{eq:KappaV} and \eqref{eq:KappaI}. Now, substituting the above inequalities into \eqref{eq:ContinuityVoltageS1}--\eqref{eq:ContinuityWaveS1}, adding the resulting inequalities together, and using Gr\"{o}nwall's inequality yield, for some $\alpha, \beta >0$,
	\begin{equation} \label{eq:ContinuityS1}
		\norm{\cU_{\rw}}{u(t)}^2 \leq \beta e^{\alpha T} \norm{\cU_{\rw}}{u_0}^2 \quad \text{for all } t\in[0,T],
	\end{equation}
	which completes the proof. 
	\qquad
\end{proof}

\begin{proposition} [Continuity of the semigroup $\{S_{\rs}\}$]\label{prp:SemigroupC0Strong}
	The semigroup $\seq{S_{\rs}(t)}_{t \in [0,\infty)}$ of strong solution operators is continuous for all $g \in L^2(0,\infty;\cL_i^2)$.
\end{proposition}
\begin{proof}
	Continiuity of the semigroup with respect to $t$ follows immediately from the time continuity of the strong solutions given by Proposition \ref{prp:RegularityStrongSolution}. 
	To prove continuous dependence on the initial values, 
	consider any two initial values $\tilde{u}_0$ and $\hat{u}_0$ in $\cU_{\rs}$ and construct the solutions $\tilde{u}(t) = S_{\rs}(t)\tilde{u}_0$ and $\hat{u}(t)=S_{\rs}(t)\hat{u}_0$, $t\in [0,T]$, $T>0$, for \eqref{eq:Voltage}--\eqref{eq:Wave}.
	Let $u := \tilde{u}-\hat{u}$ and $A:= -\Delta + I$, and take the inner product of \eqref{eq:Voltage}--\eqref{eq:Wave} for each solutions with $v$, $\rd_t i$, and $A\rd_t w$, respectively.
	Subtracting the resulting two sets of equations gives \eqref{eq:ContinuityVoltageS1}, \eqref{eq:ContinuityCurrentS1}, and 
	\begin{multline}\label{eq:ContinuityWaveS2}
		\tfrac{1}{2}\rd_t \norm{\cH_w^1}{\rd_t w}^2 + 2\nu \norm{\cH_w^1}{\Lambda^{\frac{1}{2}} \rd_t w}^2 + \tfrac{3}{4} \nu^2 \rd_t \norm{\cH_{\partial w}^1}{\dx w}^2 
		+\tfrac{1}{2} \nu^2 \rd_t \norm{\cH_w^1}{\Lambda w}^2 \\
		= \nu^2 \inner{\cL_w^2}{ \Lambda^2 \rM J_8 ( f(\tilde{v}) - f(\hat{v}) )}{A \rd_t  w}.
	\end{multline}
	Note that \eqref{eq:ContinuityS1} also holds since $\cU_{\rs} \subset \cU_{\rw}$, and since \eqref{eq:ContinuityVoltageS1} and \eqref{eq:ContinuityCurrentS1} remain unchanged, the continuity of $v$ and $i$ holds. 
	
	Now, it follows from \eqref{eq:ContinuityWaveS2} by integrating over $[0,t]$ that
	\begin{align*}
		\norm{\cH_w^1}{\rd_t w}^2 + \nu^2 \left[ \tfrac{3}{2} \norm{\cH_{\partial w}^1}{\dx w}^2 
		+\norm{\cH_w^1}{\Lambda w}^2 \right]
		&\leq \left( 	\norm{\cH_w^1}{\rd_t w}^2 + \nu^2 \left[ \tfrac{3}{2} \norm{\cH_{\partial w}^1}{\dx w}^2 
		+\norm{\cH_w^1}{\Lambda w}^2 \right] \right) \hspace{-1mm} \Big | _{t=0} \\
		& \quad + 2\nu^2 \int_0^t \inner{\cL_w^2}{ \Lambda^2 \rM J_8 ( f(\tilde{v}) - f(\hat{v}) )}{A \rd_s  w} \rd s,
	\end{align*}
	which, using \eqref{eq:DeltaFormulas}, can be written equivalently for some $\alpha_1, \beta_1>0$ as
	\begin{equation} \label{eq:1211151207}
		Q(w(t),\rd_t w(t))
		\leq \alpha_1 Q(w(0), \rd_t w(0))
		+ \beta_1 \int_0^t \inner{\cL_w^2}{ \Lambda^2 \rM J_8 ( f(\tilde{v}) - f(\hat{v}) )}{A \rd_s  w} \rd s,
	\end{equation}
	where
	\begin{equation} \label{eq:121115115}
		Q(w(t), \rd_t w(t)):= \norm{\cH_w^1}{\rd_t w(t)}^2 
		+\norm{\cL_w^2}{Aw(t)}^2.
	\end{equation}
	Integrating by parts in the second term of the right-hand side of \eqref{eq:1211151207} yields
	\begin{align} \label{eq:121015945}
		\beta_1 \int_0^t &\inner{\cL_w^2}{\Lambda^2 \rM J_8 ( f(\tilde{v}) - f(\hat{v}) )}{A \rd_s  w} \rd s \\
		&=\beta_1\inner{\cL_w^2}{\Lambda^2 \rM J_8 ( f(\tilde{v}) - f(\hat{v}) )}{A w}
		- \beta_1 \inner{\cL_w^2}{\Lambda^2 \rM J_8 ( f(\tilde{v}_0) - f(\hat{v}_0) )}{A w_0} \nonumber\\
		&\hspace{5.8 cm}  - \beta_1 \int_0^t \inner{\cL_w^2}{\Lambda^2 \rM J_8 \rd_s ( f(\tilde{v}) - f(\hat{v}) )}{A w} \rd s. \nonumber 
	\end{align}	
	
	Next, recalling that $\sup_{v_{\rm X}(x,t)\in \bbR}|\partial_{v_{\rX}} f_{\rX}(v_{\rX})| \leq \frac{\rF_{\rX}}{2\sqrt{2} \sigma_{\rX}}$ by \eqref{eq:FireingFuncDerivative} and using Young's inequality we obtain 
	\begin{align} \label{eq:61216858}
		\beta_1 \inner{\cL_w^2}{\Lambda^2 \rM J_8 ( f(\tilde{v}) - f(\hat{v}) )}{A w}
		& \leq \tfrac{1}{2 } \norm{\cL_w^2}{A w}^2 + \frac{\beta_1^2}{16} \frac{\rF_{\rE}^2}{\sigma_{\rE}^2} \,  \trace(\Lambda^4 \rM^2 )  \norm{\cL_v^2}{v}^2, \\
		-\beta_1 \inner{\cL_w^2}{\Lambda^2 \rM J_8 ( f(\tilde{v}_0) - f(\hat{v}_0) )}{A w_0}
		& \leq \tfrac{1}{2} \norm{\cL_w^2}{A w_0}^2 + \frac{\beta_1^2}{16} \frac{\rF_{\rE}^2}{\sigma_{\rE}^2} \,  \trace(\Lambda^4 \rM^2 )  \norm{\cL_v^2}{v_0}^2. \nonumber
	\end{align}
	Moreover, 
	\begin{align*}
		-\beta_1 \inner{\cL_w^2}{\Lambda^2 \rM J_8 \rd_s ( f(\tilde{v}) &- f(\hat{v}) )}{A w} \\
		& = -\beta_1 \inner{\cL_w^2}{\Lambda^2 \rM J_8 (\partial_{\tilde{v}} f(\tilde{v}) \rd_s \tilde{v} - \partial_{\hat{v}} f(\hat{v}) \rd_s \hat{v})}{A w}\\
		& \leq \tfrac{1}{2} \norm{\cL_w^2}{A w}^2 + \tfrac{1}{2} \beta_1^2 \norm{\cL_w^2}{\Lambda^2 \rM J_8 (\partial_{\tilde{v}} f(\tilde{v}) \rd_s \tilde{v} - \partial_{\hat{v}} f(\hat{v}) \rd_s \hat{v})}^2\\
		& = \tfrac{1}{2} \norm{\cL_w^2}{A w}^2 + \tfrac{1}{2} \beta_1^2 \trace(\Lambda^4 \rM^2 ) \int_{\Omega} 
		|\partial_{\tilde{v}_{\rE}} f(\tilde{v}_{\rE}) \rd_s \tilde{v}_{\rE} - \partial_{\hat{v}_{\rE}} f(\hat{v}_{\rE}) \rd_s \hat{v}_{\rE}|^2 \rd x,
	\end{align*} 
	where, noting that 
	$\sup_{v_{\rE}(x,t) \in \bbR} |\partial_{v_{\rE}}^2 f_{\rE}(v_{\rE})| < \frac{1}{5} \frac{\rF_{\rE}}{\sigma_{\rE}^2}$ by direct computation of the derivative of \eqref{eq:FireingFuncDerivative}, we can write
	\begin{align*}
		|\partial_{\tilde{v}_{\rE}} f(\tilde{v}_{\rE}) \rd_s \tilde{v}_{\rE} - \partial_{\hat{v}_{\rE}} f(\hat{v}_{\rE}) \rd_s \hat{v}_{\rE}|^2 \rd x
		&= |\partial_{\tilde{v}_{\rE}} f(\tilde{v}_{\rE}) \rd_s v_{\rE} + (\partial_{\tilde{v}_{\rE}} f(\tilde{v}_{\rE})-\partial_{\hat{v}_{\rE}} f(\hat{v}_{\rE}) ) \rd_s \hat{v}_{\rE}|^2\\
		& \leq 2 |\partial_{\tilde{v}_{\rE}} f(\tilde{v}_{\rE})|^2 |\rd_s v_{\rE}|^2 
		+ 2 |\partial_{\tilde{v}_{\rE}} f(\tilde{v}_{\rE})-\partial_{\hat{v}_{\rE}} f(\hat{v}_{\rE})|^2 |\rd_s \hat{v}_{\rE}|^2\\
		& \leq \tfrac{1}{4} \frac{\rF_{\rE}^2}{\sigma_{\rE}^2}|\rd_s v_{\rE}|^2 
		+2 \left[\sup_{v_{\rE}(x,t) \in \bbR} |\partial_{v_{\rE}}^2 f_{\rE}(v_{\rE})| \right]^2 |v_{\rE}|^2 |\rd_s \hat{v}_{\rE}|^2 \\
		& \leq \tfrac{1}{4} \frac{\rF_{\rE}^2}{\sigma_{\rE}^2}|\rd_s v_{\rE}|^2 
		+ \tfrac{2}{25} \frac{\rF_{\rE}^2}{\sigma_{\rE}^4} |v_{\rE}|^2 |\rd_s \hat{v}_{\rE}|^2.
	\end{align*} 
	Therefore, it follows that
	\begin{align} \label{eq:61216900}
		-\beta_1 \inner{\cL_w^2}{\Lambda^2 \rM J_8 \rd_s ( f(\tilde{v}) - f(\hat{v}) )}{A w} 
		& \leq \tfrac{1}{2} \norm{\cL_w^2}{A w}^2 
		+ \frac{\beta_1^2}{8} \frac{\rF_{\rE}^2}{\sigma_{\rE}^2} \trace(\Lambda^4 \rM^2 ) \norm{\cL_v^2}{\rd_s v}^2 \\
		& \quad +  \frac{\beta_1^2}{25} \frac{\rF_{\rE}^2}{\sigma_{\rE}^4} \trace(\Lambda^4 \rM^2 ) \norm{C^1([0,T];\cL_v^2)}{\rd_s \hat{v}}^2 \norm{\cL_v^2}{v}^2. \nonumber
	\end{align}	
	Moreover, \eqref{eq:Voltage} implies that for some $\alpha_2>0$,
	\begin{equation} \label{eq:VolatgeDerivativeBound}
		\norm{\cL_v^2}{\rd_s v(s)}^2 \leq 
		\alpha_2 \left( \norm{\cL_v^2}{v(s)}^2 + \norm{\cL_i^2}{i(s)}^2 + \norm{\cL_v^2}{v(s)}^2 \norm{\cL_i^2}{i(s)}^2\right) \quad  \text{for all } s \in [0,T].
	\end{equation}
	
	Now, substituting \eqref{eq:61216858}, \eqref{eq:61216900} and \eqref{eq:VolatgeDerivativeBound} into \eqref{eq:121015945} and using \eqref{eq:ContinuityS1}, it follows that there exist some $\beta_2,\dots, \beta_6>0$ such that
	\begin{align*}
		\beta_1 \int_0^t &\inner{\cL_w^2}{\Lambda^2 \rM J_8 ( f(\tilde{v}) - f(\hat{v}) )}{\rd_s A w} \rd s\\
		&\leq \tfrac{1}{2} \int_0^t \norm{\cL_w^2}{A w}^2 \rd s 
		+\beta_2 \int_0^t  \left( \norm{\cL_v^2}{v}^2 + \norm{\cL_i^2}{i}^2 +  \norm{\cL_v^2}{v}^2  \norm{\cL_i^2}{i}^2 \right)  \rd s	\\ 
		& \hspace{3.2cm} + \tfrac{1}{2} \norm{\cL_w^2}{A w}^2 
		+\beta_3 \norm{\cL_v^2}{v}^2  + \tfrac{1}{2} \norm{\cL_w^2}{A w_0}^2 + \beta_4  \norm{\cL_v^2}{v_0}^2,\\
		&\leq \tfrac{1}{2} \int_0^t \norm{\cL_w^2}{A w}^2  \rd s 
		+\beta_5  \norm{\cU_{\rw}}{u_0}^2 \left(1+ \norm{\cU_{\rw}}{u_0}^2 \right) t 
		+ \tfrac{1}{2} \norm{\cL_w^2}{A w}^2 + \tfrac{1}{2} \norm{\cL_w^2}{A w_0}^2
		+\beta_6 \norm{\cU_{\rw}}{u_0}^2 .	
	\end{align*} 
	Substituting this inequality into \eqref{eq:1211151207} yields
	\begin{align} \label{eq:531161204}
		\tfrac{1}{2}Q(w(t),\rd_t w(t)) 
		&\leq \tfrac{1}{2} \int_0^t Q(w(s),\rd_s w(s))  \rd s
		+ \beta_5 \norm{\cU_{\rw}}{u_0}^2 \left(1+ \norm{\cU_{\rw}}{u_0}^2 \right) t	  \\
		&\quad + \alpha_1 Q(w(0), \rd_t w(0))  
		+ \tfrac{1}{2} \norm{\cL_w^2}{A w_0}^2 + \beta_6 \norm{\cU_{\rw}}{u_0}^2, \nonumber
	\end{align}
	where, using Gr\"{o}nwall's inequality for the function $\tfrac{1}{2} \int_0^t Q(w(s),\rd_s w(s))\rd s$, we can write
	\begin{align*}
		\tfrac{1}{2}\int_0^t Q(w(s),\rd_s w(s)) \rd s
		&\leq \beta_5  \norm{\cU_{\rw}}{u_0}^2 \left(1+ \norm{\cU_{\rw}}{u_0}^2 \right) \big(e^t - (t+1) \big) \\
		&\quad + 	\left[\alpha_1 Q(w(0), \rd_t w(0)) + \tfrac{1}{2} \norm{\cL_w^2}{A w_0}^2 + \beta_6 \norm{\cU_{\rw}}{u_0}^2\right]\big( e^t -1 \big).
	\end{align*}
	This inequality along with \eqref{eq:531161204} and the definition of $Q$, given by \eqref{eq:121115115}, implies that for some $\beta_7 > 0$,
	\begin{align*}
		Q(w(t),\rd_t w(t)) \leq \beta_7 e^T \left[ Q(w(0), \rd_t w(0)) + \norm{\cU_{\rw}}{u_0}^2 \left(1+ \norm{\cU_{\rw}}{u_0}^2 \right) \right] \quad \text{for all } t\in [0,T].
	\end{align*}
	Now, noting that 
	$Q(w(0), \rd_t w(0)) = \norm{\cH_w^1}{w_0'}^2+\norm{\cL_w^2}{A w_0}^2$, 
	it follows from the above inequality and \eqref{eq:ContinuityS1} that, for some $\hat{\alpha}, \hat{\beta}>0$, 
	\begin{equation*} 
		\norm{\cU_{\rs}}{u(t)}^2 \leq \hat{\beta} e^{\hat{\alpha} T}  \norm{\cU_{\rs}}{u_0}^2 \left(1+ \norm{\cU_{\rw}}{u_0}^2 \right) \quad \text{for all } t\in[0,T],
	\end{equation*}
	which completes the proof.	
	\qquad
\end{proof}

\subsection{Biophysically Plausible Phase Spaces} \label{sec:BiophysicalPhaseSpaces}
Although the spaces $\cU_{\rw}$ and $\cU_{\rs}$ constructed in \eqref{eq:PhaseSpaceLarge} provide the theoretical phase spaces of the problem for the solutions constructed in Section \ref{sec:ExistenceUniqueness}, 
the evolution of the dynamics of the model is not biophysically plausible on the entire spaces $\cU_{\rw}$ and $\cU_{\rs}$. 
As described in Section \ref{sec:ModelDiscription}, $i(x,t)$, $w(x,t)$, and $g(x,t)$ represent nonnegative biophysical quantities. 
In fact, initial functions $i'_0 \in \cL_i^2$ and $w'_0 \in \cL_w^2$ can be constructed such that the solutions $i(x,t)$ and $w(x,t)$, despite starting from nonnegative initial values $i_0 \in \cL_i^2$ and $w_0 \in \cH_w^1$, take negative values over a set $\sX \subset \Omega$ of positive measure for a time interval of positive length. 
In the following propositions, we establish conditions under which the dynamics of the model is guaranteed to evolve in biophysically plausible subsets of $\cU_{\rw}$ and $\cU_{\rs}$.

\begin{proposition}[Nonnegativity of the solution $w(x,t)$] \label{prp:WaveNonnegativity}
	Suppose that $w \in L^2(0,T;\cH_w^1)$ is the $w$-component of an $\Omega$-periodic weak solution $u(t) = S_{\rw}(t) u_0$ of \eqref{eq:Voltage}--\eqref{eq:InitialValues} and define the set 
	$\cD_w \subset \cH_w^1 \times \cL_w^2$ as
	\begin{multline} \label{eq:WavePositiveRegion}
		\cD_w := \set{(w_0,w_0') \in \cW_w^{1,\infty} \times \cL_w^{\infty}}
		{w_0'+\nu \Lambda w_0 \geq 0 \text{ a.e. in } \Omega,  \right. \\ \left. \text{ and }
			w_0(y) + \dy w_0(y)(y-x) \geq 0 \text{ for almost every } x \in \Omega, y \in B(x,t), t\in (0,T]}.
	\end{multline}
	Then, for every initial values $(w_0,w_0') \in \cD_w$, the solution $w(x,t)$ remains nonnegative almost everywhere in $\Omega$ for all $t\in (0,T]$. 
\end{proposition}
\begin{proof} 
	First, note that the weak and strong solutions coincide for $v(t)$ and they satisfy \eqref{eq:Voltage} and \eqref{eq:Current} almost everywhere in $\Omega$ for all $t\in [0,T]$, $T>0$; see the proof of Theorem \ref{th:StrongExistenceUniqueness}.
	Substituting $v(t)$ into $f$, we can interpret $f(v)$ in \eqref{eq:Wave} as a function $\hat{f}(x,t) := f(v(x,t))$ for almost every $x\in \Omega$ and all  $t\in [0,T]$. 
	Next, using \eqref{eq:FiringRateFunction}, \eqref{eq:Parameters}, and Proposition \ref{prp:RegularityWeakSolution}, it is implied that $\hat{f} \in L^{\infty}(0,T;\cL_v^{\infty})$, and $\hat{f} >0$ in $\Omega \times [0,T]$. 
	Now, replace $f(v)$ in \eqref{eq:Wave} by $\hat{f}$ and scale $x$ by the factor $\sqrt{\tfrac{3}{2}}\nu$ to obtain 
	\begin{alignat*}{3}
		&\dt^2 \tilde{w} + 2 \nu \Lambda \dt \tilde{w} - \Delta \tilde{w} + \nu^2 \Lambda^2 \tilde{w} - \tilde{f} = 0, & &\quad \text{in } \tilde{\Omega} \times (0,T],\\
		&\tilde{w} = \tilde{w}_0, \quad \dt \tilde{w} = \tilde{w}'_0,& &\quad \text{on } \tilde{\Omega} \times \{0\},
	\end{alignat*}
	where $\tilde{\Omega}:=\sqrt{\tfrac{3}{2}}\nu\Omega$, and $\tilde{w}$, $\tilde{w}_0$, $\tilde{w}_0'$, and $\tilde{f}$ denote $w$, $w_0$, $w_0'$, and $\nu^2 \Lambda^2 \rM J_8 \hat{f}$ in the scaled domain $\tilde{\Omega}$, respectively.
	Note that with the new interpretation of $f$, the above equation  is a system of two decoupled telegraph equations. 
	Therefore, applying the same arguments to each of the two equations independently, in what follows we assume without loss of generality that the above equation is a scalar equation. 
	
	Using the change of variable $q:= e^{\nu \Lambda t} \tilde{w}$ the problem can be transformed to the initial-value problem of the standard wave equation given by 
	\begin{alignat}{3} \label{eq:TelegraphWaveForm}
		&\dt^2 q - \Delta q = e^{\nu \Lambda t} \tilde{f}, & &\quad \text{in } \bbR^2 \times (0,T],\\
		& q =  \tilde{w}_0, \quad \dt q= \tilde{w}'_0 + \nu \Lambda \tilde{w}_0, & &\quad \text{on } \bbR^2 \times \{0\}. \nonumber
	\end{alignat}
	Here, the extension from $\tilde{\Omega}$ to $\bbR^2$ is done periodically due to the $\tilde{\Omega}$-periodicity of the functions. 
	Let $\tilde{w}_{0 \varepsilon}$,  ${\tilde{w}'}_{0 \varepsilon}$, and  $\tilde{f}_{\varepsilon}$ denote,
	respectively,   $\tilde{w}_0$,  $\tilde{w}_0'$, and  $\tilde{f}$ after mollification by the standard positive mollifier $\phi_{\varepsilon} \in C_{\rc}^{\infty}(\bbR^2)$; see \cite[Sec. 2.6]{Ciarlet:FunctionalAnalysis:2013}. 
	Using Poisson's formula for the homogeneous wave equation in $\bbR^2$, along with Duhamel's principle for the nonhomogeneous problem \cite[Sec. 2.4]{Evans:PDE:2010}, it follows that the function
	\begin{align} \label{eq:TelegraphWaveSolution}
		q_{\varepsilon}(x,t) &:= \tfrac{1}{2} \dashint_{B(x,t)} 
		\frac{t \big[ \tilde{w}_{0 \varepsilon}(y) + \inner{\bbR^2}{\dy \tilde{w}_{0 \varepsilon}(y)}{y-x} \big] + t^2 \big[ {\tilde{w}'}_{0 \varepsilon}(y) +\nu \Lambda \tilde{w}_{0 \varepsilon}(y) \big]}
		{\big[ t^2 - \norm{\bbR^2}{y-x}^2\big]^{\frac{1}{2}}} \,\rd y \\
		&\hspace{2.05cm} + \tfrac{1}{2} \int_0^t (t-s)^2 e^{\nu \Lambda s} \dashint_{B(x,t-s)}
		\frac{\tilde{f}_{\varepsilon}(y,s)}
		{\big[(t-s)^2 - \norm{\bbR^2}{y-x}^2\big]^{\frac{1}{2}}} \,\rd y \, \rd s \nonumber
	\end{align}  
	solves \eqref{eq:TelegraphWaveForm} classically for the forcing term $e^{\nu \Lambda t}\tilde{f}_{\varepsilon}$ and initial values $\tilde{w}_{0 \varepsilon}$ and  ${\tilde{w}'}_{0 \varepsilon}$. 
	
	The second term in this solution is nonnegative for all $t\in[0,T]$ since $\tilde{f}$, and consequently, $\tilde{f}_{\varepsilon}$ are nonnegative on $B(x,t)$ for all $x \in \Omega$ and all $t \in [0,T]$. 
	Moreover, by \cite[Theorem 2.6-1]{Ciarlet:FunctionalAnalysis:2013} and the definition of weak derivative we can write
	\begin{align*}
		\inner{\bbR^2}{\dy \tilde{w}_{0 \varepsilon}(y)}{y-x}
		&= \left( \int_{B(y,\varepsilon)} \dy \phi_{\varepsilon}(y-z) \tilde{w}_0(z) \rd z\, , \, y-x \right)_{\hspace{-0.9mm} \bbR^2}\\
		&=\left(-\int_{B(y,\varepsilon)} \dz \phi_{\varepsilon}(y-z) \tilde{w}_0(z) \rd z\, , \,y-x \right)_{\hspace{-0.9mm} \bbR^2}\\
		&= \left(\int_{B(y,\varepsilon)} \phi_{\varepsilon}(y-z) \dz \tilde{w}_0(z) \rd z\, , \, y-x 
		\right)_{\hspace{-0.9mm} \bbR^2}\\
		&= \int_{B(y,\varepsilon)} \phi_{\varepsilon}(y-z) \inner{\bbR^2}{\dz \tilde{w}_0(z)}{z-x}\rd z\\
		& \hspace{1.55cm} + \int_{B(y,\varepsilon)} \phi_{\varepsilon}(y-z) \inner{\bbR^2}{\dz \tilde{w}_0(z)}{y-z}\rd z,
	\end{align*}
	where, using H\"{o}lder's inequality and the property $\int_{B(0,\varepsilon)} \phi_{\varepsilon}(x)\rd x =1$, we have
	\begin{align*}
		\left|  \int_{B(y,\varepsilon)} \phi_{\varepsilon}(y-z) \inner{\bbR^2}{\dz \tilde{w}_0(z)}{y-z}\rd z \right|
		&\leq \norm{\cL_{\partial w}^{\infty}}{\dx \tilde{w}_0} \int_{B(y,\varepsilon)} \phi_{\varepsilon}(y-z) 
		\norm{1}{y-z} \rd z\\
		& \leq \sqrt{2}  \norm{\cL_{\partial w}^{\infty}}{\dx \tilde{w}_0} \varepsilon. 
	\end{align*} 
	Therefore, it follows that
	\begin{align*}
		\dashint_{B(x,t)} &
		\frac{t \big[ \tilde{w}_{0 \varepsilon}(y) + \inner{\bbR^2}{\dy \tilde{w}_{0 \varepsilon}(y)}{y-x} \big] }
		{\big[ t^2 - \norm{\bbR^2}{y-x}^2\big]^{\frac{1}{2}}} \,\rd y\\ 
		&\geq \dashint_{B(x,t)} t \left[\frac{ \int_{B(y,\varepsilon)} \phi_{\varepsilon}(y-z) \big[\tilde{w}_0(z) + \inner{\bbR^2}{\dz \tilde{w}_0(z)}{z-x}\big] \rd z }
		{\big[ t^2 - \norm{\bbR^2}{y-x}^2\big]^{\frac{1}{2}}} - \frac{\sqrt{2} \norm{\cL_{\partial w}^{\infty}}{\dx w_0}  \varepsilon}
		{\big[ t^2 - \norm{\bbR^2}{y-x}^2\big]^{\frac{1}{2}}}\right] \rd y\\
		& \geq -\sqrt{2} \norm{\cL_{\partial w}^{\infty}}{\dx \tilde{w}_0}  \varepsilon 
		\quad \text{for all }(\tilde{w}_0,\tilde{w}_0')\in \tilde{\cD}_w,	
	\end{align*}
	where $\tilde{\cD}_w$ denotes $\cD_w$ in the scaled domain $\tilde{\Omega}$.
	Note that the last inequality holds since the first term in the integration on the right-hand side is nonnegative by \eqref{eq:WavePositiveRegion}, and 
	$t \big[ t^2 - \norm{\bbR^2}{y-x}^2\big]^{-\frac{1}{2}}$ 
	takes the average value $1$ over the ball $B(x,t)$. 
	Finally, note that ${\tilde{w}'}_{0 \varepsilon}(y) +\nu \Lambda \tilde{w}_{0 \varepsilon}(y)$ in \eqref{eq:TelegraphWaveSolution} is nonnegative on $B(x,t)$ when $(\tilde{w}_0,\tilde{w}_0')\in \tilde{\cD}_w$. Therefore, it follows that
	\begin{equation} \label{eq:MollifiedTelegraphBound}
		q_{\varepsilon}(x,t) \geq -\sqrt{2} \norm{\cL_{\partial w}^{\infty}}{\dx \tilde{w}_0}  \varepsilon 
		\quad \text{for all }(\tilde{w}_0,\tilde{w}_0')\in \tilde{\cD}_w.
	\end{equation}
	
	Now, taking the limits as $\varepsilon \rightarrow 0$, 
	it follows from \cite[Theorem 2.6-3]{Ciarlet:FunctionalAnalysis:2013} that $\tilde{w}_{0 \varepsilon} \rightarrow \tilde{w}_0$, ${\tilde{w}'}_{0 \varepsilon} \rightarrow \tilde{w}_0'$, and ${\tilde{f}}_{\varepsilon} \rightarrow \tilde{f}$ in $L^2(\tilde{\Omega}_t)$, where 
	$\tilde{\Omega}_t := \set{y \in \bbR^2}{y \in B(x,t), x\in \Omega}$. 
	Therefore, there exists a subsequence $\seq{\varepsilon_n}_{n=1}^{\infty}$, convergent to $0$, such that
	$\tilde{w}_{0 \varepsilon_n} \rightarrow \tilde{w}_0$, $\tilde{w}'_{0 \varepsilon_n} \rightarrow \tilde{w}_0'$, and ${\tilde{f}}_{\varepsilon_n} \rightarrow \tilde{f}$ 
	almost everywhere on $\Omega_t$ as $n \rightarrow \infty$ \cite[Th. 2.30]{Folland:RealAnalysis:1999}. 
	Moreover, since $(\tilde{w}_0,\tilde{w}_0') \in \cW_w^{1,\infty} \times \cL_w^{\infty}$ in $\tilde{\cD}_w$, $\tilde{f} \in L^{\infty}(0,T;\cL^{\infty}_v)$, and	the function $\big[ t^2 - \norm{\bbR^2}{y-x}^2\big]^{-\frac{1}{2}}$ is integrable over $B(x,t)$,  
	it follows that the integrands in \eqref{eq:TelegraphWaveSolution} are uniformly bounded with respect to $\varepsilon$ by integrable functions over $B(x,t)$. 
	The Lebesgue dominated convergence theorem then implies that  
	$q(x,t):=\lim_{n\rightarrow \infty} q_{\varepsilon_n}(x,t)$ exists on $\tilde{\Omega}_t$ and, by uniqueness of the weak solution, is a weak solution of the wave equation \eqref{eq:TelegraphWaveForm}. 
	Now, letting $\varepsilon=\varepsilon_n\rightarrow 0$ in \eqref{eq:MollifiedTelegraphBound}, it follows that if  $(\tilde{w}_0, \tilde{w}_0')\in \tilde{\cD}_w$, then $q(x,t) \geq 0$ for almost every $x\in \tilde{\Omega}$ and all $t\in(0,T]$. 
	This completes the proof since the change of variable $\tilde{w}=e^{-\nu\Lambda t}q$ and space rescaling $\Omega = \sqrt{\frac{2}{3}}\nu^{-1}\tilde{\Omega}$ do not change the sign of solutions.   
	\qquad
\end{proof}

\begin{corollary}[Boundedness of the weak solutions] \label{cor:Boundedness}
	Suppose $g \in L^{\infty}(0,T;\cL_i^{\infty})$, $v_0 \in \cL_v^{\infty}$, $i_0 \in \cL_i^{\infty}$, $i_0' \in \cL_i^{\infty}$,  $w_0 \in \cW_w^{1,\infty} $, and $w_0' \in \cL_w^{\infty}$.
	Then, in addition to the regularities given by Proposition \ref{prp:RegularityWeakSolution}, the weak solution $(v(t),i(t),w(t))$ of \eqref{eq:Voltage}--\eqref{eq:InitialValues} satisfies
	\begin{align*}
		v &\in W^{2, \infty}(0,T;\cL_v^{\infty}) \cap C^{1,1}([0,T];\cL_v^{\infty}), \\
		i &\in W^{1, \infty}(0,T;\cL_i^{\infty}) \cap C^{0,1}([0,T];\cL_i^{\infty}), \\
		w &\in L^{\infty}(0,T;\cL_w^{\infty}). 
	\end{align*}
\end{corollary}
\begin{proof}
	The boundedness of $w$ follows immediately from the proof of Proposition \ref{prp:WaveNonnegativity}, since under the assumption $w_0 \in \cW_w^{1,\infty} $
	and $w_0' \in \cL_w^{\infty}$ the integrands in \eqref{eq:TelegraphWaveSolution} are integrable and each component of the weak solution $w(t)$ is achieved almost everywhere in $\Omega$ as the limit of \eqref{eq:TelegraphWaveSolution} when $\varepsilon \rightarrow 0$, 
	followed by the space rescaling from $\tilde{\Omega}$ to $\Omega$.
	
	Now, to prove boundedness of $v$, $i$, and $\rd_t i$, let $x_0 \in \Omega$ be any Lebesgue point\footnote{
		The choice of a Lebesgue point is for the sake of definiteness. Almost every point in $\Omega$ 
		can be used as $x_0$.  
	} of the initial functions $v_0$, $i_0$, $i_0'$, $w_0$, and $g(0)$. 
	Take the $\bbR^4$-inner product of \eqref{eq:Current} at $x_0$ with $\rd_t i(x_0,t)$ for every $t \in (0,T]$ to obtain
	\begin{multline*} 
		\inner{\bbR^4}{\rd_t^2 i_{x_0}}{\rd_t i_{x_0}} + 2 \inner{\bbR^4}{\Gamma \rd_t i_{x_0}}{\rd_t i_{x_0}} + \inner{\bbR^4}{\Gamma^2 i_{x_0}}{\rd_t i_{x_0}} \\
		\hspace{1.1cm} - e \inner{\bbR^4}{\Upsilon \Gamma J_6 w_{x_0}}{\rd_t i_{x_0}} - e \inner{\bbR^4}{\Upsilon \Gamma \rN J_7 f( v_{x_0})}{\rd_t i_{x_0}} 
		= e \inner{\bbR^4}{\Upsilon\Gamma g_{x_0}}{\rd_t i_{x_0}},
	\end{multline*}  
	where $v_{x_0}(t) := v(x_0,t)$, $i_{x_0}(t) := i(x_0,t)$, $w_{x_0}(t) := w(x_0,t)$, and $g_{x_0}(t) := g(x_0,t)$.
	This equality is similar to \eqref{eq:CurrentApprx2} in the proof of Proposition \ref{prp:EnergyEstimates}, with
	$\cL_i^2$-inner products being replaced by $\bbR^4$-inner product, and $v^{(m)}$, $i^{(m)}$, and $w^{(m)}$ being replaced by $v_{x_0}$, $i_{x_0}$, and $w_{x_0}$, respectively. 
	Therefore, similar arguments as in the proof of Proposition \ref{prp:EnergyEstimates} imply that 
	\begin{align} \label{eq:517161116} 
		\sup_{t\in [0,T]}\left(\norm{\bbR^4}{\rd_t i_{x_0}(t)}^2 + \norm{\bbR^4}{i_{x_0}(t)}^2 \right) \leq \kappa_i ,
	\end{align}
	where, with $\kappa_w:= \norm{L^{\infty}(0,T;\cL_w^{\infty})}{w}^2$ and for some $\alpha_1 >0$ independent of $x_0$,
	\begin{multline*}
		\kappa_i = \alpha_1 \left( \norm{\cL_i^{\infty}}{i'_0}^2 + \norm{\cL_i^{\infty}}{i_0}^2
		+\left[ \frac{e^2 \kappa_w}{\gamma_{\min}} \norm{2}{\Upsilon\Gamma J_6}^2 
		+ \frac{e^2 |\Omega|}{\gamma_{\min}}  (\rF_{\rE}^2 + \rF_{\rI}^2 ) \norm{2}{\Upsilon \Gamma \rN J_7}^2 \right]T \right. \\ \left.
		+ \frac{e^2}{2 \gamma_{\min}} \norm{2}{\Upsilon\Gamma}^2 \norm{L^{\infty}(0,T;\cL_i^{\infty} )}{g}^2 \right),
	\end{multline*} 
	and $\gamma_{\min}$ is the smallest eigenvalue of $\Gamma$.
	
	Similarly, taking the $\bbR^2$-inner product of \eqref{eq:Voltage} at $x_0$ with $v_{x_0}(t)$ and using the arguments following \eqref{eq:VoltageApprx2} in the proof of Proposition \ref{prp:EnergyEstimates} yields
	\begin{align} \label{eq:517161117}
		\sup_{t\in [0,T]}\left( \norm{\cL_v^2}{v_{x_0}(t)}^2 \right) \leq \kappa_v,
	\end{align}
	where, for some $\alpha_2, \beta >0$ independent of $x_0$,
	\begin{align*}
		\kappa_v = \alpha_2 \exp \left( \beta \sqrt{2\kappa_i} \norm{2}{\Psi} T \right) \left( \norm{\cL_v^{\infty}}{v_0}^2 + \frac{\kappa_i}{\sqrt{2 \kappa_i} \norm{2}{\Psi}} \right).
	\end{align*}  
	Now, note that almost every point $x_0 \in \Omega$ is a Lebesgue point for the locally integrable initial functions, and the estimates $\kappa_v$ and $\kappa_i$ are independent of $x_0$. 
	Therefore, taking the supremum over all Lebesgue points $x_0 \in \Omega$ in \eqref{eq:517161116} and \eqref{eq:517161117} implies $v \in L^{\infty}(0,T;\cL_v^{\infty})$ and $i \in W^{1,\infty}(0,T;\cL_i^{\infty})$
	which, recalling \eqref{eq:Voltage}, further imply $v \in W^{2,\infty}(0,T;\cL_v^{\infty})$.
	Finally, it follows by using Morrey's inequality \cite[Th. 5.6-4 and Th. 5.6-5]{Evans:PDE:2010} that $v \in C^{1,1}([0,T];\cL_v^{\infty})$ and $i \in C^{0,1}([0,T];\cL_i^{\infty})$, which completes the proof.
	\qquad
\end{proof}

Next, we recall and use the following standard result in the theory of ordinary differntial equations to establish conditions that guarantee nonnegativity of $i(x,t)$ for all biophysically plausible values of the input $g$, that is, for all $g\in L^2(0,T ; \cD_g)$, where
\begin{equation}\label{eq:InputPositiveRegion} 
	\cD_g := \set{\ell \in \cL_i^2}{\ell \geq 0 \text{ a.e. in } \Omega}.
\end{equation}
\begin{proposition}[Invariance of the nonnegative cone {\cite[Prop. I.1.1]{Chepyzhov:Attractors:2002}}] \label{prop:ConeInvariance}
	$ $Let $\seq{S(t)}_{t\in [0,\infty)}$ be the semigroup of solution operators associated with the ordinary differential equation 
	\begin{equation*}
		\rd_t q(t) = P(q(t)), \quad q(t) \in \bbR^n, \quad t\in [0,\infty), 
	\end{equation*}
	where $P:\bbR^n \rightarrow \bbR^n$ is a continuous locally Lipschitz mapping. 
	Then the nonnegative cone $\bbR_{+}^n$ is invariant for $\seq{S(t)}_{t\in [0,\infty)}$
	if and only if $P(q)$ is quasipositive, that is, for every $j \in \{1,\dots,n\}$,
	\begin{equation*}
		P_j(q_1,\dots, q_n) \geq 0 \text{ whenevr } q_j =0 \text{ and } q_k \geq 0 \text{ for all } k\neq j.
	\end{equation*}
\end{proposition}

\begin{proposition}[Positively invariant region for the solution $i(x,t)$] \label{prp:CurrentNonnegativity}
	Suppose $g\in L^2(0,T ; \cD_g)$ and let $u(t) = S_{\rw}(t) u_0$ be an $\Omega$-periodic weak solution  of \eqref{eq:Voltage}--\eqref{eq:InitialValues}. Suppose the $w$-component of the weak solution, $w(x,t)$, is nonnegative for almost every $x \in \Omega$ and all $t\in[0,T]$, $T>0$, 
	and define the set 
	\begin{equation}\label{eq:CurrentPositiveRegion} 
		\cD_i := \set{(\ell, \ell') \in \cL_i^2 \times \cL_i^2}{\ell \geq 0 \text{ and }\ell' + \Gamma \ell \geq 0 \text{ a.e. in } \Omega}.
	\end{equation}
	Then, for every $(i_0, i'_0) \in \cD_i$, we have $(i(t), \rd_t i(t)) \in \cD_i$ almost everywhere in $\Omega$ for all $t \in [0,T]$.
	An identical result holds for strong solutions $u(t) = S_{\rs}(t) u_0$ of \eqref{eq:Voltage}--\eqref{eq:InitialValues} with nonnegative $w$-component.
\end{proposition}
\begin{proof}
	Let $b := \rd_t i + \Gamma i$ and rewrite \eqref{eq:Current} as the first-order system of equations
	\begin{align} \label{eq:721161241}
		\rd_t i &=  - \Gamma i + b,\\
		\rd_t b &=  - \Gamma b + e \Upsilon \Gamma J_6 w + e \Upsilon \Gamma \rN J_7 f(v) + e \Upsilon\Gamma g. \nonumber
	\end{align}	 
	Let $x_0 \in \Omega$ be a Lebesgue point of the initial functions $v_0$, $i_0$, $i_0'$, $w_0$, and $g(0)$, and define $v_{x_0}(t)$, $i_{x_0}(t)$, $w_{x_0}(t)$, and $g_{x_0}(t)$ as given in the proof of Corollary \ref{cor:Boundedness}. 
	Accordingly, let $b_{x_0}(t) := b(x_0,t) = \rd_t i_{x_0}(t) + \Gamma i_{x_0}(t)$. 
	
	Now, \eqref{eq:721161241} implies that the function $q_{x_0}:=(i_{x_0},b_{x_0})$ satisfies the ordinary differential equation $\rd_t q_{x_0}(t) = P(q_{x_0}(t))$, $t\in[0,T]$, where the mapping $P : \bbR^8 \rightarrow \bbR^8$ given by
	\begin{equation*}
		P(q_{x_0}) = P(i_{x_0},b_{x_0}):= (- \Gamma i_{x_0} + b_{x_0}, - \Gamma b_{x_0} + e \Upsilon \Gamma J_6 w_{x_0} + e \Upsilon \Gamma \rN J_7 f(v_{x_0}) + e \Upsilon\Gamma g_{x_0} )
	\end{equation*}
	is Lipschitz continuous.
	Moreover, note that by assumption we have $w_{x_0} \geq 0$ and $g_{x_0} \geq 0$ which, along with the definitions of 
	$f$, $\Upsilon$, $\Gamma$, $\rN$, $J_6$, and $J_7$ given by \eqref{eq:FiringRateFunction} and \eqref{eq:Parameters}, implies 
	$e \Upsilon \Gamma J_6 w_{x_0}(t) \geq 0$, $e \Upsilon \Gamma \rN J_7 f(v_{x_0}(t)) \geq 0$,  and $e \Upsilon\Gamma g_{x_0}(t) \geq 0$ for all $t \in [0,T]$. 
	Therefore, it follows that $P$ is quasipositive, and hence, by Proposition \ref{prop:ConeInvariance} we have $q_{x_0}(t) \geq 0$ for all $t \in [0,T]$.
	This completes the proof since $x_0$ is an arbitrary Lebesgue point of the initial functions and almost every points in $\Omega$ is a Lebesgue point for these functions.\footnote{
		Note that there are fairly standard results in the literature that ensure the positivity of a
		$C^1(\overline{\Omega} \times [0,T]; \bbR^m)$ function as it evolves in time, provided its time-derivative satisfies certain conditions on the boundary of the positive cone;
		see for example \cite[Lemma 6]{Kufner:NoDEA:1996} and \cite{Chueh:IndianaMathJ:1977}.
		The proofs of these results are relatively geometrical and usually use continuity of the functions and the compactness of $\overline{\Omega}$. 
		However, these proofs are by no means applicable to functions in $C^1([0,T]; \cL^2({\Omega};\bbR^m))$.
		In fact, functions in $C^1([0,T]; \cL^2({\Omega};\bbR^m))$ are allowed to \emph{leak} through the boundary of the positive cone on sets of measure zero at every $t \in [0,T]$.
		Since any subinterval of $[0,T]$ is uncountable, it is not guaranteed that the uncountable union of such leakage sets remains having measure zero over a subinterval.
		In the proof of Proposition \ref{prp:CurrentNonnegativity}, we use the additional property that the functions are governed by a system of ODE's. 
		Therefore, for all $t \in (0,T]$, the Banach-space valued function $i(t)$ is defined at the same almost every 
		points in $\Omega$ as they are defined initially at $t=0$.
		In other words, the leakage set remains unchanged for all $t \in (0,T]$.
	}  
	\qquad
\end{proof}

\begin{remark}[Biophysically plausible set of initial values]
	Propositions \ref{prp:WaveNonnegativity} and \ref{prp:CurrentNonnegativity} ensure that if $g\in L^2(0,\infty ; \cD_g)$, where $\cD_g$ is given by \eqref{eq:InputPositiveRegion}, and the initial values lie in the set
	\begin{equation} \label{eq:PhysiologicalSet}
		\cD_{\rm Bio} := \cL_v^2 \times \cD_i \times \cD_w,
	\end{equation}
	where $\cD_w$ and $\cD_i$ are given by \eqref{eq:WavePositiveRegion} and \eqref{eq:CurrentPositiveRegion},
	respectively, then $i(x,t)$ and $w(x,t)$ always remain nonnegative at almost every point in $\Omega$ as they
	evolve in time. 
	However, it should be noted that this does not imply that the set $\cD_{\rm Bio} \subset \cU_{\rw}$ is positively invariant, since Proposition \ref{prp:WaveNonnegativity} does not imply positive invariance of the set $\cD_w$. 
	Therefore, $\cD_{\rm Bio}$ cannot serve as a phase space for the semidynamical system framework of the problem.
\end{remark}

In the analysis of next sections, nonnegativity of the solution $i(x,t)$ is essential. 
Moreover, it would be of no practical value if we analyze the dynamics of the model out of the biophysical regions of the phase space.  
Therefore, we define 
\begin{align} \label{eq:PlausibleInitialValues}
	\cD_{\rw} &:= \set{u_0 \in \cU_{\rw}}{i(t) \geq 0, w(t) \geq 0 \text{ a.e. in } \Omega \text{ for all } t\in [0,\infty), u(t)=S_{\rw}(t)u_0},\\
	\cD_{\rs} &:= \set{u_0 \in \cU_{\rs}}{i(t) \geq 0, w(t) \geq 0 \text{ a.e. in } \Omega \text{ for all } t\in [0,\infty), u(t)=S_{\rs}(t)u_0}, \nonumber
\end{align}
as the maximal closed subsets of $\cU_{\rw}$ and $\cU_{\rs}$ for the initial values of the weak and strong solutions, respectively, such that $i$ and $w$ initiated from the points in these sets evolve nonnegatively over time. 
Note that $\cD_{\rw}$ and $\cD_{\rs}$ are nonempty since
$\cD_{\rm Bio} \subset \cD_{\rw}$ and  
$\cD_{\rm Bio} \cap \cU_{\rs} \subset \cD_{\rs}$ when $g \in L^2(0,\infty, \cD_g)$.
Moreover, 
$\cD_{\rw}$ and $\cD_{\rs}$ are closed sets since $\seq{S_{\rw}(t)}_{t \in [0,\infty)}$ and $ \seq{S_{\rs}(t)}_{t \in [0,\infty)}$ are continuous semigroups, as given by Propositions \ref{prp:SemigroupC0Weak} and \ref{prp:SemigroupC0Strong}. 
Moreover, it follows immediately from the definitions given by \eqref{eq:PlausibleInitialValues} that $\cD_{\rw}$ and $\cD_{\rs}$ are positively invariant sets. 
Therefore, endowed with the metric induced by the norm in $\cU_{\rw}$ and $\cU_{\rs}$, the sets $\cD_{\rw}$ and $\cD_{\rs}$ form positively invariant complete metric spaces and can be considered as biophysically plausible phase spaces of the model, based on which, we construct the semidynamical systems 
\begin{equation*}
	\Big(\cD_{\rw}, \seq{S_{\rw}(t)}_{t \in [0,\infty)} \Big), \quad 
	\Big(\cD_{\rs}, \seq{S_{\rs}(t)}_{t \in [0,\infty)} \Big),
\end{equation*}
associated with the weak and strong solutions of \eqref{eq:Voltage}--\eqref{eq:InitialValues}, respectively, and investigate their global dynamics in the remainder of the paper.


\section{Existence of Absorbing Sets} \label{sec:AbsorbingSets}
In this section, we prove the existence of bounded absorbing sets for the semigroups 
$\seq{S_{\rw}(t)}_{t \in [0,\infty)}$ and $\seq{S_{\rs}(t)}_{t \in [0,\infty)}$ acting on $\cD_{\rw}$ and $\cD_{\rs}$, respectively. 
First, we recall the following definition of an absorbing set for an operator semigroup.
\begin{definition}[Absorbing set {\cite[Def. II.2.3]{Chepyzhov:Attractors:2002}}]
	$ $ A set $\sB_0$ in a complete metric space $\cD$ is called an \emph{absorbing set} for the semigroup $\seq{S(t):\cD \rightarrow \cD}_{t \in [0,\infty)}$ if for every bounded set $\sB \in \cD$ there exists $t_0(\sB) \in (0,\infty)$ such that $S(t) \sB \subset \sB_0$ for all $t \geq t_0(\sB)$.
\end{definition}

\begin{theorem}[Existence of absorbing sets in $\cD_{\rw}$] \label{th:AbsorbingSetD1}
	Assume that $g \in L^{\infty}(0,\infty; \cD_g)$ and there exists $\theta > 2\gamma_{\min}^{-3}$ such that 
	\begin{enumerate}
		\item $\tfrac{4}{3} \theta e^2 \Upsilon_{\rE \rE}^2 \gamma_{\max} (\nu \Lambda_{\rE \rE})^{-3} < 1$, \vspace{0.8mm}
		\item $\tfrac{4}{3} \theta e^2 \Upsilon_{\rE \rI}^2 \gamma_{\max} (\nu \Lambda_{\rE \rI})^{-3} < 1$, \vspace{0.8mm}	
	\end{enumerate}
	where $\gamma_{\min}:=\min\{\gamma_{\rE\rE}, \gamma_{\rE\rI}, \gamma_{\rI\rE}, \gamma_{\rI\rI} \}$ and $\gamma_{\max}:= \max\{\gamma_{\rE\rE}, \gamma_{\rE\rI}, \gamma_{\rI\rE}, \gamma_{\rI\rI} \}$ are the smallest and largest eigenvalues of $\Gamma$, respectively.
	Then 
	the semigroup $\seq{S_{\rw}(t):\cD_{\rw} \rightarrow \cD_{\rw}}_{t \in [0,\infty)}$ associated with the weak solutions of  \eqref{eq:Voltage}--\eqref{eq:InitialValues} has a bounded absorbing set $\sB_{\rw}$. 
	Specifically, consider the functions $Q_{\rw}^{-}:\cD_{\rw} \rightarrow [0, \infty)$ and $Q_{\rw}^{+}:\cD_{\rw} \rightarrow [0, \infty)$ defined by
	\begin{align} \label{eq:QLimitsD1}
		Q_{\rw}^{-}(u) &:=\norm{\cL_v^2}{\Phi^{\frac{1}{2}} v}^2 +\theta \norm{\cL_i^2}{\rd_t i + \tfrac{3}{2} \Gamma i}^2 
		+ \tfrac{1}{4}\theta  \norm{\cL_i^2}{\Gamma i}^2  + \norm{\cL_w^2}{\rd_t w + \tfrac{3}{2}\nu \Lambda w}^2\\
		&\hspace{7.7 cm} + \tfrac{1}{4} \nu^2  \min \{6 , \Lambda_{\min}^{2}\}\norm{\cH_w^1}{w}^2, \nonumber\\
		Q_{\rw}^{+}(u) &:=\norm{\cL_v^2}{\Phi^{\frac{1}{2}} v}^2 +\theta \norm{\cL_i^2}{\rd_t i + \tfrac{3}{2} \Gamma i}^2 
		+ \tfrac{1}{4} \theta \norm{\cL_i^2}{\Gamma i}^2  + \norm{\cL_w^2}{\rd_t w + \tfrac{3}{2}\nu \Lambda w}^2 \nonumber\\
		&\hspace{7.6 cm} + \tfrac{1}{4} \nu^2  \max \{6 ,  \Lambda_{\max}^{2}\} \norm{\cH_w^1}{ w}^2, \nonumber
	\end{align}	
	and a scalar $\varepsilon$ such that
	\begin{equation} \label{eq:EpsilonRange}
		\max \left\{\tfrac{4}{3} \theta e^2 \Upsilon_{\rE \rE}^2  \gamma_{\max} (\nu \Lambda_{\rE \rE})^{-3},  \tfrac{4}{3} \theta e^2 \Upsilon_{\rE \rI}^2  \gamma_{\max} (\nu \Lambda_{\rE \rI})^{-3} \right\} < 2 \gamma_{\max} \varepsilon < 1.
	\end{equation}
	Let $\tau_{\max} :=\max\{\tau_{\rE}, \tau_{\rI} \}$ denote the largest eigenvalue of $\Phi$, and 
	$\Lambda_{\min} :=\min\{\Lambda_{\rE\rE}, \Lambda_{\rE\rI} \}$ and $\Lambda_{\max}:=\max\{\Lambda_{\rE\rE}, \Lambda_{\rE\rI} \}$ denote the smallest and largest eigenvalues of $\Lambda$, respectively.
	Let $\rho_{\rw}^2 :=\frac{\beta_{\rw}}{\alpha_{\rw}}$, where
	\begin{align} 
		\alpha_{\rw} &:= \min \Big\{ 
		\tfrac{2}{3} \tau_{\max}^{-1},
		\left(\tfrac{1}{2}\gamma_{\max}^{-1} - \varepsilon \right) \gamma_{\min}^2,
		3 \theta^{-1} \left(\theta \gamma_{\min} - 2\gamma_{\min}^{-2} \right),
		\tfrac{1}{2} \nu \Lambda_{\min}, \label{eq:DecayRateD1}\\
		&\hspace{5.82cm} 3 \nu \Lambda_{\max}^{-2} \min \{ \Lambda_{\rE\rE}^3 - \tfrac{2}{3} \tfrac{\theta e^2}{\nu^3 \varepsilon} \Upsilon_{\rE\rE}^2 , \Lambda_{\rE\rI}^3 - \tfrac{2}{3} \tfrac{\theta e^2}{\nu^3 \varepsilon} \Upsilon_{\rE\rI}^2 \}
		\Big\},  \nonumber\\
		\beta_{\rw} &:= \frac{4\theta e^2}{\gamma_{\max}^{-1} - 2\varepsilon} \left[ |\Omega| (\rF_{\rE}^2 
		+ \rF_{\rI}^2 ) \norm{2}{\Upsilon \rN J_7}^2 + \norm{2}{\Upsilon}^2 \norm{L^{\infty}(0,\infty; \cL_i^2)}{g}^2 \right]
		+2 \nu^3  |\Omega| \rF_{\rE}^2 \trace (\Lambda^3 \rM^2). \label{eq:UltimateBoundBetaD1}
	\end{align} 
	Then, for all $\rho > \rho_{\rw}$, the bounded sets $\sB_{\rw}:=\set{u \in \cD_{\rw}}{Q_{\rw}^{-}(u) \leq \rho^2}$ are absorbing in $\cU_{\rw}$. 
	Moreover, for every bounded set $\sB \subset \cD_{\rw}$ there exists $R > 0$ such that $Q_{\rw}^{+}(u_0) \leq R^2$ for all $u_0 \in \sB$, and $S(t) \sB \subset \sB_{\rw}$ for all $t \geq t_{\rw}(\sB)$, where
	\begin{equation} \label{eq:UltimateTimeD1}
		t_{\rw}(\sB) = t_{\rw}(R) := \max \left\{ 0, \frac{1}{\alpha_{\rw}} \log \frac{R^2}{\rho^2 - \rho_{\rw}^2} \right\}.
	\end{equation}
\end{theorem}
\begin{proof}
	First, taking the inner product of ($\ref{eq:Voltage}$) with $v$ yields
	\begin{equation*} 
		\tfrac{1}{2} \rd_t \norm{\cL_v^2}{\Phi^{\frac{1}{2}} v}^2 + \norm{\cL_v^2}{v}^2 - \inner{\cL_v^2}{J_1 i}{v}
		+ \int_{\Omega} \left( v_1^2 i^{\rT} \Psi J_4 + v_2^2 i^{\rT} \Psi J_5 \right) \rd x = 0. 
	\end{equation*}	 
	The integral term in this equation is nonnegative in $\cD_{\rw}$ for all $t \in [0,\infty)$; see \eqref{eq:Parameters} and \eqref{eq:PlausibleInitialValues}. 
	Therefore, dropping the integral term and using Young's inequality yields, for every $\varepsilon_1 >0$,
	\begin{align} \label{eq:VoltageLyapunov} 
		\rd_t \norm{\cL_v^2}{\Phi^{\frac{1}{2}} v}^2 &\leq -2(1-\varepsilon_1) \norm{\cL_v^2}{v}^2 + \frac{1}{\varepsilon_1} \norm{\cL_i^2}{i}^2 \\
		&\leq -2(1-\varepsilon_1) \tau_{\max}^{-1} \norm{\cL_v^2}{\Phi^{\frac{1}{2}} v}^2 + \frac{1}{\varepsilon_1 \gamma_{\min}^2} \norm{\cL_i^2}{\Gamma i}^2. \nonumber
	\end{align}
	
	Next, let $b := \rd_t i + \frac{3}{2} \Gamma i$ and rewrite \eqref{eq:Current} as
	\begin{equation*}
		\rd_t b + \tfrac{1}{2} \Gamma b + \tfrac{1}{4} \Gamma^2 i 
		- e \Upsilon \Gamma J_6 w - e \Upsilon \Gamma \rN J_7 f(v) = e \Upsilon\Gamma g.
	\end{equation*}
	Taking the inner product of the above equality with $b$ yields
	\begin{multline*}
		\tfrac{1}{2}\rd_t \norm{\cL_i^2}{b}^2 + \tfrac{1}{2} \inner{\cL_i^2}{\Gamma b}{b} + \tfrac{1}{8}\rd_t \norm{\cL_i^2}{\Gamma i}^2 + \tfrac{3}{8} \norm{\cL_i^2}{\Gamma^{\frac{3}{2}} i}^2 \\- e \inner{\cL_i^2}{\Upsilon \Gamma J_6 w}{b} - e \inner{\cL_i^2}{\Upsilon \Gamma \rN J_7 f(v)}{b}
		= e \inner{\cL_i^2}{\Upsilon\Gamma g}{b}.
	\end{multline*}	
	Note that	
	\begin{align*}
		\inner{\cL_i^2}{\Gamma b}{b} &\geq \gamma_{\max}^{-1} \norm{\cL_i^2}{\Gamma b}^2, \\
		\norm{\cL_i^2}{\Gamma^{\frac{3}{2}} i}^2 &\geq \gamma_{\min} \norm{\cL_i^2}{\Gamma i}^2,
	\end{align*}	
	and, using similar arguments as in the proof of Proposition \ref{prp:EnergyEstimates}, it follows that for every $\varepsilon_2, \varepsilon_3, \varepsilon_4>0$,
	\begin{align*}
		e \inner{\cL_i^2}{\Upsilon \Gamma J_6 w}{b} & \leq \varepsilon_2 \norm{\cL_i^2}{\Gamma b}^2 
		+ \frac{e^2}{4 \varepsilon_2} \norm{\cL_i^2}{\Upsilon J_6 w}^2 \\
		e \inner{\cL_i^2}{\Upsilon \Gamma \rN J_7 f(v)}{b} & \leq  \varepsilon_3 \norm{\cL_i^2}{\Gamma b}^2
		+ \frac{e^2 |\Omega|}{4 \varepsilon_3}  (\rF_{\rE}^2 + \rF_{\rI}^2 ) \norm{2}{\Upsilon \rN J_7}^2,\\
		e \inner{\cL_i^2}{\Upsilon\Gamma g}{b} 
		& \leq \varepsilon_4 \norm{\cL_i^2}{\Gamma b}^2 + \frac{e^2}{4 \varepsilon_4} \norm{2}{\Upsilon}^2 \norm{\cL_i^2}{g}^2.
	\end{align*}
	Therefore,
	\begin{multline} \label{eq:CurrentLyapunov}
		\rd_t \left[ \norm{\cL_i^2}{b}^2 + \tfrac{1}{4} \norm{\cL_i^2}{\Gamma i}^2 \right] 
		\leq - \left(\gamma_{\max}^{-1} - 2(\varepsilon_2 + \varepsilon_3 + \varepsilon_4) \right) \norm{\cL_i^2}{\Gamma b}^2 - \tfrac{3}{4}\gamma_{\min} \norm{\cL_i^2}{\Gamma i}^2 \\
		+ \frac{e^2}{2 \varepsilon_2} \norm{\cL_i^2}{\Upsilon J_6 w}^2 + \frac{e^2}{2 \varepsilon_3} |\Omega| (\rF_{\rE}^2 + \rF_{\rI}^2 ) \norm{2}{\Upsilon \rN J_7}^2 + \frac{e^2}{2 \varepsilon_4} \norm{2}{\Upsilon}^2 \norm{\cL_i^2}{g}^2.
	\end{multline} 
	
	Next, let $q := \rd_t w + \frac{3}{2}\nu \Lambda w$ and rewrite \eqref{eq:Wave} as
	\begin{equation} \label{eq:WaveChangeOfVariable}
		\rd_t q + \tfrac{1}{2}\nu \Lambda q - \tfrac{3}{2} \nu^2 \Delta w + \tfrac{1}{4} \nu^2 \Lambda^2 w 
		- \nu^2 \Lambda^2 \rM J_8 f(v) = 0.
	\end{equation}
	Taking the inner product of this equality with $q$ yields
	\begin{multline*}
		\tfrac{1}{2}\rd_t \norm{\cL_w^2}{q}^2 + \tfrac{1}{2}\nu \norm{\cL_w^2}{\Lambda^{\frac{1}{2}} q}^2 + \tfrac{3}{4} \nu^2 \rd_t \norm{\cL_{\partial w}^2}{\dx w}^2 + \tfrac{9}{4} \nu^3 \norm{\cL_{\partial w}^2}{\Lambda^{\frac{1}{2}} \dx w}^2 +\tfrac{1}{8} \nu^2 \rd_t \norm{\cL_w^2}{\Lambda w}^2 \\
		+\tfrac{3}{8} \nu^3 \norm{\cL_w^2}{\Lambda^{\frac{3}{2}}w}^2 - \nu^2 \inner{\cL_w^2}{ \Lambda^2 \rM J_8 f(v)}{q} = 0.
	\end{multline*}	
	Using similar arguments as in the proof of Proposition \ref{prp:EnergyEstimates} we can write, for every $\varepsilon_5>0$,
	\begin{equation*}
		\inner{\cL_w^2}{\Lambda^2 \rM J_8 f(v^{(m)} )}{q} \leq
		\varepsilon_5 \norm{\cL_w^2}{\Lambda^{\frac{1}{2}} q}^2 +\frac{1}{4\varepsilon_5} |\Omega| \rF_{\rE}^2 \trace (\Lambda^3 \rM^2),
	\end{equation*}
	and hence, it follows that
	\begin{align} \label{eq:WaveLyapunovD1}
		\rd_t \Big[\norm{\cL_w^2}{q}^2 &+ \tfrac{3}{2} \nu^2 \norm{\cL_{\partial w}^2}{\dx w}^2 +\tfrac{1}{4} \nu^2 \norm{\cL_w^2}{\Lambda w}^2 \Big] \\
		&\leq  - \nu( 1 -2\nu \varepsilon_5) \norm{\cL_w^2}{\Lambda^{\frac{1}{2}}q}^2
		-3\nu \Big(\tfrac{3}{2} \nu^2 \norm{\cL_{\partial w}^2}{\Lambda^{\frac{1}{2}}\dx w}^2 +\tfrac{1}{4} \nu^2 \norm{\cL_w^2}{\Lambda^{\frac{3}{2}} w}^2  \Big) \nonumber\\
		&\quad +\frac{\nu^2}{2\varepsilon_5} |\Omega| \rF_{\rE}^2 \trace (\Lambda^3 \rM^2). \nonumber
	\end{align}
	
	Now, set $\varepsilon_1 = \tfrac{2}{3}$ in \eqref{eq:VoltageLyapunov}, set $\varepsilon_3 = \varepsilon_4 = \tfrac{1}{8}(\gamma_{\max}^{-1} - 2\varepsilon)$ in \eqref{eq:CurrentLyapunov} with $\varepsilon := \varepsilon_2$, and set $\varepsilon_5 = \tfrac{1}{4\nu}$ in \eqref{eq:WaveLyapunovD1}. 
	Then, multiplying \eqref{eq:CurrentLyapunov} by $\theta>0$ and adding the result to \eqref{eq:VoltageLyapunov} and \eqref{eq:WaveLyapunovD1} yields
	\begin{align*}
		\rd_t Q_{\rw}
		&\leq -\tfrac{2}{3}  \tau_{\max}^{-1} \norm{\cL_v^2}{\Phi^{\frac{1}{2}} v}^2 
		- \theta \left(\tfrac{1}{2}\gamma_{\max}^{-1} - \varepsilon \right) \norm{\cL_i^2}{\Gamma b}^2 
		- \tfrac{3}{4} \left(\theta \gamma_{\min} - 2 \gamma_{\min}^{-2} \right) \norm{\cL_i^2}{\Gamma i}^2
		- \tfrac{1}{2} \nu \norm{\cL_w^2}{\Lambda^{\frac{1}{2}}q}^2 \\
		& \quad -3\nu \left( \tfrac{3}{2} \nu^2 \norm{\cL_{\partial w}^2}{\Lambda^{\frac{1}{2}}\dx w}^2 
		+\tfrac{1}{4} \nu^2 \inner{\cL_w^2}{\left[ \Lambda^3 - \tfrac{2}{3} \frac{\theta e^2}{\nu^3 \varepsilon} J_6^{\rT}\Upsilon^2 J_6 \right] w} { w} \right) + \beta_{\rw},  
	\end{align*} 
	where $\beta_{\rw}$ is given by \eqref{eq:UltimateBoundBetaD1} and
	\begin{equation} \label{eq:LyapunovD1}
		Q_{\rw}(u) = \norm{\cL_v^2}{\Phi^{\frac{1}{2}} v}^2 +\theta  \norm{\cL_i^2}{b}^2 
		+ \tfrac{1}{4} \theta \norm{\cL_i^2}{\Gamma i}^2
		+ \norm{\cL_w^2}{q}^2 + \tfrac{3}{2} \nu^2 \norm{\cL_{\partial w}^2}{\dx w}^2 +\tfrac{1}{4} \nu^2 \norm{\cL_w^2}{\Lambda w}^2. 
	\end{equation}
	Note that for $\theta > 2\gamma_{\min}^{-3}$ we have $\theta \gamma_{\min} - 2\gamma_{\min}^{-2} >0$ and for the range of values of $\varepsilon$ given by \eqref{eq:EpsilonRange} we have $\tfrac{1}{2}\gamma_{\max}^{-1} - \varepsilon >0$.
	Moreover, Assumptions (i) and (ii) along with \eqref{eq:EpsilonRange} ensure that $\Lambda^3 - \tfrac{2}{3} \frac{\theta e^2}{\nu^3 \varepsilon} J_6^{\rT}\Upsilon^2 J_6 >0$.
	Therefore, with the decay rate $\alpha_{\rw}$ given by \eqref{eq:DecayRateD1}, 
	\begin{equation} \label{eq:LyapunovInequalityD1}
		\rd_t Q_{\rw}(u) \leq -\alpha_{\rw} Q_{\rw}(u) + \beta_{\rw},
	\end{equation}
	and hence, using Gr\"{o}nwall's inequality \cite[Sec. III.1.1.3.]{Temam:InfiniteDimensional:1997},
	\begin{equation} \label{eq:LyapunovBoundD1}
		Q_{\rw}^{-}(u(t)) \leq Q_{\rw}^{+}(u(0)) e^{-\alpha_{\rw} t} + \rho_0^2 \left( 1 - e^{-\alpha_{\rw} t} \right),
	\end{equation}
	where $Q_{\rw}^{-}$ and $Q_{\rw}^{+}$ are given in \eqref{eq:QLimitsD1} and 
	$\limsup_{t \rightarrow \infty} Q_{\rw}^{-}(u(t)) \leq \rho_0^2 := \frac{\beta_{\rw}}{\alpha_{\rw}}$.
	Now, since the mapping 
	\begin{equation} \label{eq:IsomorphismD1}
		(v,i,i',w,w') \mapsto ( \Phi^{\frac{1}{2}}v, \tfrac{1}{2} \theta^{\frac{1}{2}} \Gamma i, \theta^{\frac{1}{2}}[i'+\tfrac{3}{2}\Gamma i], \tfrac{1}{2} \nu [ \max \{6 , \Lambda_{\max}^{2}\} ]^{\frac{1}{2}} w, w'+\tfrac{3}{2}\nu \Lambda w )
	\end{equation}
	is a linear isomorphism over $\cU_{\rw}$, for every bounded set $\sB \subset \cD_{\rw}$ there exists $R > 0$ such that $Q_{\rw}^{+}(u_0) \leq R^2$ for all $u_0 \in \sB$. 
	Hence, it is immediate from \eqref{eq:LyapunovBoundD1} that $S_{\rw}(t) \sB \subset \sB_{\rw}$ for all $t \geq t_{\rw}(\sB)$, where $t_{\rw}(\sB)$ is given by \eqref{eq:UltimateTimeD1}. 
	\qquad
\end{proof}

\begin{theorem}[Existence of absorbing sets in $\cD_{\rs}$] \label{th:AbsorbingSetD2}
	Suppose the assumptions of Theorem \ref{th:AbsorbingSetD1} hold, namely, assume $g \in L^{\infty}(0,\infty; \cD_g)$ and there exists $\theta > 2\gamma_{\min}^{-3}$ such that the biophysical parameters of the model satisfy
	\begin{enumerate}
		\item $\tfrac{4}{3} \theta e^2 \Upsilon_{\rE \rE}^2 \gamma_{\max} (\nu \Lambda_{\rE \rE})^{-3} < 1$, \vspace{0.8mm}
		\item $\tfrac{4}{3} \theta e^2 \Upsilon_{\rE \rI}^2 \gamma_{\max} (\nu \Lambda_{\rE \rI})^{-3} < 1$, \vspace{0.8mm}	
	\end{enumerate}
	where $\gamma_{\min}$ and $\gamma_{\max}$ are the smallest and largest eigenvalues of $\Gamma$, respectively.
	Then 
	the semigroup $\seq{S_{\rs}(t):\cD_{\rs} \rightarrow \cD_{\rs}}_{t \in [0,\infty)}$ associated with the strong solutions of  \eqref{eq:Voltage}--\eqref{eq:InitialValues} has a bounded absorbing set $\sB_{\rs}$. 
	Specifically, consider the function $Q_{\rs}^{-}:\cD_{\rs} \rightarrow [0, \infty)$ defined by
	\begin{multline} \label{eq:QMinusLimitsD2}
		Q_{\rs}^{-}(u) :=\norm{\cL_v^2}{\Phi^{\frac{1}{2}} v}^2 +\theta  \norm{\cL_i^2}{\rd_t i + \tfrac{3}{2} \Gamma i}^2 
		+ \tfrac{1}{4} \theta \norm{\cL_i^2}{\Gamma i}^2 + \norm{\cH_w^1}{\rd_t w + \tfrac{3}{2}\nu \Lambda w}^2  \\
		+ \tfrac{1}{8} \nu^2  \min \{6, \Lambda_{\min}^{2} \}\norm{\cL_w^2}{(-\Delta + I) w}^2, 
	\end{multline}	
	and denote by $\Lambda_{\min}$ and $\Lambda_{\max}$ the smallest and largest eigenvalues of $\Lambda$, respectively, and by $\tau_{\max}$ the largest eigenvalue of $\Phi$.
	Let $\rho_{\rs}^2 :=\frac{2\beta_{\rs}}{\alpha_{\rs}}$ with
	\begin{align} 
		\alpha_{\rs} &:= \min \Big\{ 
		\tfrac{2}{3} \tau_{\max}^{-1},
		\left(\tfrac{1}{2}\gamma_{\max}^{-1} - \varepsilon \right) \gamma_{\min}^2,
		3 \theta^{-1} \left(\theta \gamma_{\min} - 2\gamma_{\min}^{-2} \right),
		\nu \Lambda_{\min}, \label{eq:DecayRateD2}\\
		&\hspace{4.2cm} 3 \nu \Lambda_{\max}^{-2} \min \{ \Lambda_{\rE\rE}^3 - \tfrac{2}{3} \tfrac{\theta e^2}{\nu^3 \varepsilon} \Upsilon_{\rE\rE}^2 , \Lambda_{\rE\rI}^3 - \tfrac{2}{3} \tfrac{\theta e^2}{\nu^3 \varepsilon} \Upsilon_{\rE\rI}^2 \}
		\Big\},  \nonumber\\
		\beta_{\rs} &:= \frac{4 \theta e^2 }{\gamma_{\max}^{-1} - 2\varepsilon} \left[ |\Omega| (\rF_{\rE}^2 
		+ \rF_{\rI}^2 ) \norm{2}{\Upsilon \rN J_7}^2 	+ \norm{2}{\Upsilon}^2 \norm{L^{\infty}(0,\infty;\cL_i^2)}{g}^2 \right] \label{eq:UltimateBoundBetaD2}\\	
		&\hspace{2.02cm} +2\nu^2 \left[ \frac{1}{32\varepsilon_1}\frac{ \rF_{\rE}^2}{ \sigma_{\rE}^2} \trace (\Lambda^4 \rM^2)
		\eta \rho_{\rw}^2 (1+\rho_{\rw}^2)
		+\tfrac{1}{4} |\Omega| \rF_{\rE}^2 \trace (\Lambda^4 \rM^2) \left(\frac{1}{\varepsilon_1} + \frac{\alpha_{\rs}}{\varepsilon_2} \right)  \right], \nonumber
	\end{align}
	where $\eta$ is a positive constant, $\rho_{\rw}^2 := \frac{\beta_{\rw}}{\alpha_{\rw}}$ is the same constant given in Theorem \ref{th:AbsorbingSetD1},  the scalar $\varepsilon$ takes values within the same range given by \eqref{eq:EpsilonRange}, and 
	\begin{equation} \label{eq:62516903}
		\varepsilon_1 := \tfrac{1}{32} \alpha_{\rs} \min \{6, \Lambda_{\min}^{2} \} \big(1 + \norm{2}{\tfrac{3}{2}\nu\Lambda - \alpha I}^2 \big)^{-1}, \quad 
		\varepsilon_2 := \tfrac{1}{16} \min \{6, \Lambda_{\min}^{2}\}.
	\end{equation}
	Then, for all $\rho > \rho_{\rs}$, the bounded sets $\sB_{\rs}:=\set{u \in \cD_{\rs}}{Q_{\rs}^{-}(u) \leq \rho^2}$ are absorbing in $\cD_{\rs}$. 
\end{theorem}
\begin{proof}
	Let $A:= -\Delta +I$ and take the inner product of \eqref{eq:WaveChangeOfVariable} with $A q$ to obtain
	\begin{multline*}
		\tfrac{1}{2}\rd_t \norm{\cH_w^1}{q}^2 + \tfrac{1}{2}\nu \norm{\cH_w^1}{\Lambda^{\frac{1}{2}} q}^2 + \tfrac{3}{4} \nu^2 \rd_t \norm{\cH_{\partial w}^1}{\dx w}^2 + \tfrac{9}{4} \nu^3 \norm{\cH_{\partial w}^1}{\Lambda^{\frac{1}{2}} \dx w}^2 +\tfrac{1}{8} \nu^2 \rd_t \norm{\cH_w^1}{\Lambda w}^2\\
		+\tfrac{3}{8} \nu^3 \norm{\cH_w^1}{\Lambda^{\frac{3}{2}} w}^2 - \nu^2 \inner{\cL_w^2}{ \Lambda^2 \rM J_8 f(v)}{A q} = 0.
	\end{multline*}	
	This equality, along with the inequalities \eqref{eq:VoltageLyapunov} and \eqref{eq:CurrentLyapunov} derived in the proof of Theorem \ref{th:AbsorbingSetD1} and the same values of $\varepsilon_1, \dots, \varepsilon_4$ therein, implies that 
	\begin{align*}
		\rd_t Q_{\rs}
		&\leq -\tfrac{2}{3} \tau_{\max} \norm{\cL_v^2}{\Phi^{\frac{1}{2}}v}^2 
		- \theta \left(\tfrac{1}{2}\gamma_{\max}^{-1} - \varepsilon \right) \norm{\cL_i^2}{\Gamma b}^2 
		- \tfrac{3}{4} \left(\theta \gamma_{\min} - 2\gamma_{\min}^{-2} \right) \norm{\cL_i^2}{\Gamma i}^2
		- \nu \norm{\cH_w^1}{\Lambda^{\frac{1}{2}}q}^2 \\
		& \quad -3\nu \left( \tfrac{3}{2} \nu^2 \norm{\cH_{\partial w}^1}{\Lambda^{\frac{1}{2}}\dx w}^2 
		+\tfrac{1}{4} \nu^2 \inner{\cH_w^1}{\left[ \Lambda^3 - \tfrac{2}{3} \frac{\theta e^2}{\nu^3 \varepsilon} J_6^{\rT}\Upsilon^2 J_6 \right] w} { w} \right) \\
		& \quad + 2 \nu^2 \inner{\cL_w^2}{ \Lambda^2 \rM J_8 f(v)}{A q}
		+ \beta,  
	\end{align*} 
	where
	\begin{align*}
		Q_{\rs}(u) &:= \norm{\cL_v^2}{\Phi^{\frac{1}{2}} v}^2 +\theta  \norm{\cL_i^2}{b}^2 
		+ \tfrac{1}{4} \theta \norm{\cL_i^2}{\Gamma i}^2
		+ \norm{\cH_w^1}{q}^2 + \tfrac{3}{2} \nu^2 \norm{\cH_{\partial w}^1}{\dx w}^2 +\tfrac{1}{4} \nu^2 \norm{\cH_w^1}{\Lambda w}^2, \label{eq:LyapunovD2}\\
		\beta &:= \frac{4\theta e^2 }{\gamma_{\max}^{-1} - 2\varepsilon} \left[ |\Omega| (\rF_{\rE}^2 
		+ \rF_{\rI}^2 ) \norm{2}{\Upsilon \rN J_7}^2 	+ \norm{2}{\Upsilon}^2 \norm{L^{\infty}(0,\infty;\cL_i^2)}{g}^2 \right], \nonumber
	\end{align*}
	and $\varepsilon$ takes values within the range given by \eqref{eq:EpsilonRange}.
	Now, using similar arguments as in the proof of Theorem \ref{th:AbsorbingSetD1}, 
	it follows from Assumptions (i) and (ii) with $\theta > 2\gamma_{\min}^{-3}$ that
	\begin{equation} 
		\rd_t Q_{\rs}(u) \leq -\alpha_{\rs} Q_{\rs}(u) + 2 \nu^2 \inner{\cL_w^2}{ \Lambda^2 \rM J_8 f(v)}{A q} + \beta,
	\end{equation}
	where the decay rate $\alpha_{\rs}$ is given by \eqref{eq:DecayRateD2}.
	Then, Gr\"{o}nwall's inequality \cite[Sec. III.1.1.3.]{Temam:InfiniteDimensional:1997} implies
	\begin{equation} \label{eq:LyapunovBoundD2}
		Q_{\rs}(u(t)) \leq Q_{\rs}(u(0)) e^{-\alpha_{\rs} t} 
		+ 2 \nu^2 \int_0^t \inner{\cL_w^2}{ \Lambda^2 \rM J_8 f(v)}{A q} e^{\alpha_{\rs}(s-t)} \rd s
		+ \frac{\beta}{\alpha_{\rs}} \left( 1 - e^{-\alpha_{\rs} t} \right).
	\end{equation}
	
	Replacing $q := \rd_t w + \frac{3}{2}\nu \Lambda w$ in the integral term in the above inequality and integrating by parts yields
	\begin{align*}
		\int_0^t & \inner{\cL_w^2}{\Lambda^2 \rM J_8 f(v)}{A q} e^{\alpha_{\rs}(s-t)}\rd s \\
		&=  - \int_0^t \inner{\cL_w^2}{\Lambda^2 \rM J_8 \rd_s f(v)}{A w} e^{\alpha_{\rs}(s-t)}\rd s 
		+ \int_0^t \inner{\cL_w^2}{\Lambda^2 \rM J_8 f(v)}{(\tfrac{3}{2}\nu\Lambda - \alpha_{\rs} I) A w} e^{\alpha_{\rs}(s-t)}\rd s\\
		& \hspace{0.42cm} +\inner{\cL_w^2}{\Lambda^2 \rM J_8 f(v)}{A w} - \inner{\cL_w^2}{\Lambda^2 \rM J_8 f(v_0)}{A w_0} e^{-\alpha_{\rs} t}.
	\end{align*}
	Next, noting that $\rd_s f(v) = \dv f(v) \rd_s v$ and $\sup_{v_{\rE}(x,t)\in \bbR}|\partial_{v_{\rE}} f_{\rE}(v_{\rE})| \leq \frac{\rF_{\rE}}{2\sqrt{2} \sigma_{\rE}}$ by \eqref{eq:FireingFuncDerivative}, it follows that for every $\varepsilon_1, \varepsilon_2 >0$,
	\begin{align*}
		\int_0^t & \inner{\cL_w^2}{\Lambda^2 \rM J_8 f(v)}{A q} e^{\alpha_{\rs}(s-t)}\rd s \\
		&\leq 
		\varepsilon_1 \big(1 + \norm{2}{\tfrac{3}{2}\nu\Lambda - \alpha_{\rs} I}^2 \big) \int_0^t \norm{\cL_w^2}{A w}^2 e^{\alpha_{\rs}(s-t)} \rd s 
		+\frac{1}{32\varepsilon_1}\frac{\rF_{\rE}^2}{\sigma_{\rE}^2 } \trace (\Lambda^4 \rM^2) \int_0^t  \norm{\cL_v^2}{\rd_s v}^2 e^{\alpha_{\rs}(s-t)} \rd s	\\ 
		&\hspace{0.42cm} 
		+ \varepsilon_2 \norm{\cL_w^2}{A w}^2 
		+\tfrac{1}{4} |\Omega| \rF_{\rE}^2 \trace (\Lambda^4 \rM^2) \left(\frac{1}{\alpha_{\rs} \varepsilon_1 }+ \frac{1}{\varepsilon_2} \right) - \inner{\cL_w^2}{\Lambda^2 \rM J_8 f(v_0)}{A w_0} e^{-\alpha_{\rs} t}.
	\end{align*}   
	Moreover, it follows from Theorem \ref{th:AbsorbingSetD1} that for every bounded set $\sB \subset \cD_{\rs}$ there exists a time $t_{\rw}(\sB)$, given by \eqref{eq:UltimateTimeD1}, and positive constant $\eta_1$ and $\eta_2$ such that $\norm{\cL_v^2}{v(t)}^2 \leq \eta_1 \rho_{\rw}^2$ and $\norm{\cL_i^2}{i(t)}^2 \leq \eta_1 \rho_{\rw}^2$ for all $t \geq t_{\rw}(\sB)$. 
	Therefore, using the estimate \eqref{eq:VolatgeDerivativeBound} we can write
	\begin{align} \label{eq:IntDVBound}
		\int_0^t  \norm{\cL_v^2}{\rd_s v}^2 e^{\alpha_{\rs}(s-t)} \rd s
		&\leq \int_0^{t_{\rw}(\sB)}  \norm{\cL_v^2}{\rd_s v}^2 e^{\alpha_{\rs}(s-t)} \rd s + \frac{1 }{\alpha_{\rs}}\eta \rho_{\rw}^2(1+ \rho_{\rw}^2) \\
		&\leq \kappa_0(\sB) e^{-\alpha_{\rs} t} + \frac{1}{\alpha_{\rs}}\eta \rho_{\rw}^2(1+ \rho_{\rw}^2), \nonumber
	\end{align}
	where $\eta$ is a positive constant and, for some $\alpha >0$, 
	\begin{equation*}
		\kappa_0(\sB) :=\alpha \int_0^{t_{\rw}(\sB)}  \left( \norm{\cL_v^2}{v(s)}^2 + \norm{\cL_i^2}{i(s)}^2 + \norm{\cL_v^2}{v(s)}^2 \norm{\cL_i^2}{i(s)}^2\right)  e^{\alpha_{\rs} s}  \rd s < \infty.
	\end{equation*}
	
	Now, using the above estimate for the integral term in \eqref{eq:LyapunovBoundD2}, with $\varepsilon_1$ and $\varepsilon_2$ given by \eqref{eq:62516903}, yields
	\begin{equation} \label{eq:LyapunovInequalityD2}
		Q_{\rs}^{-}(u) e^{\alpha_{\rs} t} \leq \tfrac{1}{2} \alpha_{\rs} \int_0^t Q_{\rs}^{-}(u) e^{\alpha_{\rs} s} \rd s + \kappa (\sB) + \frac{\beta_{\rs}}{\alpha_{\rs}} e^{\alpha_{\rs} t} ,
	\end{equation} 
	where 
	$\beta_{\rs} :=\beta + 2\nu^2 \left[ \frac{1}{32\varepsilon_1}\frac{ \rF_{\rE}^2}{ \sigma_{\rE}^2} \trace (\Lambda^4 \rM^2)	\eta \rho_{\rw}^2(1+ \rho_{\rw}^2)	+\tfrac{1}{4} |\Omega| \rF_{\rE}^2 \trace (\Lambda^4 \rM^2) \big(\tfrac{1}{\varepsilon_1} + \tfrac{\alpha_{\rs}}{\varepsilon_2} \big)  \right]$ as given in \eqref{eq:UltimateBoundBetaD2},
	$Q_{\rs}^{-}(u)$ is given in \eqref{eq:QMinusLimitsD2}, and
	\begin{align*}
		\kappa (\sB) &:= Q_{\rs}^{+}(u(0)) + 2 \nu^2 \left[\frac{1}{32\varepsilon_1}\frac{ \rF_{\rE}^2}{ \sigma_{\rE}^2} \trace (\Lambda^4 \rM^2)
		\kappa_0(\sB)  - \inner{\cL_w^2}{\Lambda^2 \rM J_8 f(v_0)}{A w_0} \right] - \frac{\beta}{\alpha_{\rs}}, \\
		Q_{\rs}^{+}(u) &:= \norm{\cL_v^2}{\Phi^{\frac{1}{2}} v}^2 +\theta  \norm{\cL_i^2}{b}^2 
		+ \tfrac{1}{4} \theta \norm{\cL_i^2}{\Gamma i}^2 + \norm{\cH_w^1}{q}^2 + \tfrac{1}{4} \nu^2  \max \{6, \Lambda_{\max}^{2} \}\norm{\cL_w^2}{A w}^2.
	\end{align*}	
	Next, using Gr\"{o}nwall's inequality for the function $\int_0^t Q_{\rs}^{-}(u) e^{\alpha_{\rs} s} \rd s$ in \eqref{eq:LyapunovInequalityD2} gives  
	\begin{equation*} 
		\int_0^t Q_{\rs}^{-}(u) e^{\alpha_{\rs} s} \rd s \leq \frac{1}{\tfrac{1}{2}\alpha_{\rs}} \left[ \kappa(\sB) \left( e^{\frac{1}{2}\alpha_{\rs} t} -1 \right) 
		+ \frac{\beta_{\rs}}{\alpha_{\rs}} \left( e^{\alpha_{\rs} t} - e^{\frac{1}{2}\alpha_{\rs} t} \right) \right],
	\end{equation*}
	which, along with \eqref{eq:LyapunovInequalityD2} implies
	\begin{equation} \label{eq:625161102}
		Q_{\rs}^{-}(u) \leq \kappa(\sB) e^{-\frac{1}{2}\alpha_{\rs} t} + \rho_{\rs}^2 \left( 1 - \tfrac{1}{2} e^{-\frac{1}{2}\alpha_{\rs} t} \right),
	\end{equation}
	where $\limsup_{t \rightarrow \infty} Q_{\rs}^{-}(u(t)) \leq \rho_{\rs}^2 :=  \frac{2\beta_{\rs}}{\alpha_{\rs}}$. 
	
	Finally, considering the linear isomorphism \eqref{eq:IsomorphismD1} over $\cU_{\rs}$, it follows that for every bounded set $\sB \subset \cD_{\rs}$ there exists $R > 0$ such that $\kappa(\sB) \leq R^2$ for all $u_0 \in \sB$. 
	Therefore, \eqref{eq:625161102} implies that $S_{\rs}(t) \sB \subset \sB_{\rs}$ for all $t \geq t_{\rs}(\sB)$ and some $t_{\rs}(\sB)>0$, which completes the proof. 
	\qquad
\end{proof}

Note that an estimate similar to \eqref{eq:UltimateTimeD1} given in Theorem \ref{th:AbsorbingSetD1} can be also obtained for $t_{\rs}(\sB)$ in the proof of Theorem \ref{th:AbsorbingSetD2}. 
However, this would be of limited practical value since the bound \eqref{eq:IntDVBound} is very conservative for times $t \ll t_{\rw}(\sB)$.  

\begin{remark}[Conditions on parameter sets] \label{rem:ConditionsParameterSpace}
	For the range of values given in Table \ref{tb:Parameters}, the maximum value that the left-hand side of the inequalities in Assumptions (i) and (ii) of Theorems \ref{th:AbsorbingSetD1} and \ref{th:AbsorbingSetD2} may take is $39.4083\, \theta$, 
	which is achieved when 
	$\Upsilon_{\rE \rE} = 2$, $\Upsilon_{\rE \rI} = 2$, $\Lambda_{\rE \rE} = 0.1$, $\Lambda_{\rE \rI} = 0.1$, 
	$\nu = 100$, and $\gamma_{\max} = 1000$. Assumptions (i) and (ii) then require that $\theta < \frac{1}{39.4083} = 0.0254$. 
	Moreover, Theorems \ref{th:AbsorbingSetD1} and \ref{th:AbsorbingSetD2} allow for $\theta > 2\gamma_{\min}^{-3} \geq 0.002$, in accordance with Table \ref{tb:Parameters}. 
	This implies that---for the entire range of values that the biophysical parameters of the model may take---the conditions imposed by Theorems \ref{th:AbsorbingSetD1} and \ref{th:AbsorbingSetD2} are satisfied at least for any $0.002 < \theta < 0.0254$,
	and the model \eqref{eq:Model} possesses bounded absorbing sets as given by these theorems.  
\end{remark}

\section{Existence and Nonexistence of a Global Attractor} \label{sec:Attractor}
In this section, we investigate the problem of existence of a global attractor for the semigroups $\seq{S_{\rw}(t):\cD_{\rw} \rightarrow \cD_{\rw}}_{t \in [0,\infty)}$ and $\seq{S_{\rs}(t):\cD_{\rs} \rightarrow \cD_{\rs}}_{t \in [0,\infty)}$ of solution operators of \eqref{eq:Voltage}--\eqref{eq:InitialValues}. 
First, we recall the definition of a global attractor, and a widely used theorem for establishing the existence of a global attractor.  
See \cite[Ch.1]{Hale:AsymptoticBehavior:1988} for the motivation behind this definition, and \cite[Ch.3]{Hale:AsymptoticBehavior:1988} for further results.

\begin{definition}[Attracting set {\cite[Def. II.2.4]{Chepyzhov:Attractors:2002}}]
	$ $ A set $\sP$ in a complete metric space $\cD$ is called an \emph{attracting set} for a semigroup $\seq{S(t)}_{t \in [0,\infty)}$ acting in $\cD$ if for every bounded set $\sB \in \cD$, $\dist_{\cD}(S(t) \sB, \sP) \rightarrow 0$ as $t \rightarrow \infty$. Here,  
	$\dist_{\cD}(\sG, \sH) := \sup_{g \in \sG} \inf_{h \in \sH} \norm{\cD}{g-h}$ is the Hausdorff distance between the two sets $\sG, \sH \subset \cD$.
\end{definition}

\begin{definition}[Global attractor {\cite[Def. II.3.1]{Chepyzhov:Attractors:2002}}]
	$ $ A bounded set $\sA$ in a complete metric space $\cD$ is called a \emph{global attractor} for a semigroup $\seq{S(t)}_{t \in [0,\infty)}$ acting in $\cD$ if it satisfies the following conditions:
	\begin{enumerate}
		\item $\sA$ is compact in $\cD$.
		\item $\sA$ is an attracting set for $\seq{S(t)}_{t \in [0,\infty)}$.
		\item $\sA$ is strictly invariant with respect to $\seq{S(t)}_{t \in [0,\infty)}$, that is,
		$S(t)\sA = \sA$ for all $t\in[0,\infty)$.
	\end{enumerate}
\end{definition}	

\begin{definition}[Asymptotic compactness {\cite[Def. II.2.5]{Chepyzhov:Attractors:2002}}]
	$ $ The semigroup $\seq{S(t)}_{t \in [0,\infty)}$ acting in a complete metric space $\cD$ is called \emph{asymptotically compact} if it possesses a compact attracting set $\sK \Subset \cD$.
\end{definition}

\begin{theorem} [Global Attractor {\cite[Th. II.3.1]{Chepyzhov:Attractors:2002}}] \label{th:AttractorGeneral}
	Let $\seq{S(t)}_{t \in [0,\infty)}$ be an asymptotically compact continuous semigroup in a complete metric space $\cD$, possessing a compact attracting set $\sK \Subset \cD$. Then $\seq{S(t)}_{t \in [0,\infty)}$ has a global attractor $\sA \subset \sK$ given by $\sA = \omega(\sK)$, where $\omega (\sK)$ is the $\omega$-limit set of $\sK$. 
\end{theorem} 

\subsection{Challenges in Establishing a Global Attractor} \label{sec:ChallengesAttractor}
In this section, we discuss some of the standard approaches available in the literature for establishing a global attractor based on Theorem \ref{th:AttractorGeneral}, and identify reasons that make these approaches rather unpromising for the model \eqref{eq:Voltage}--\eqref{eq:Wave}.

Continuity of $\seq{S_{\rw}(t)}_{t \in [0,\infty)}$ and $\seq{S_{\rs}(t)}_{t \in [0,\infty)}$, as required by Theorem \ref{th:AttractorGeneral}, is established in Propositions \ref{prp:SemigroupC0Weak} and \ref{prp:SemigroupC0Strong}, respectively.
To prove asymptotic compactness of a semigroup $\seq{S(t)}_{t \in [0,\infty)}$ acting in $\cD$,
a general approach is to first show that the semigroup possesses a bounded absorbing set and then show that the semigroup is $\kappa$-contracting, meaning that 
$\lim_{t \rightarrow \infty} \kappa (S(t) \sB) =0$ for any bounded set $\sB \in \cD$, 
where $\kappa$ denotes the Kuratowski measure of noncompactness; see \cite{Ma:IndianaMathJ:2002, You:DynamicPDE:2007} and \cite[Ch.~3]{Hale:AsymptoticBehavior:1988}. 
An effective way to establish the later property is through a decomposition $S(t) = S_1(t) + S_2(t)$ such that for every bounded set $\sB \in \cD$ the component $S_1(t) \sB$ converges uniformly to $0$ as $t \rightarrow 0$, 
and the component $S_2(t) \sB$ is $\kappa$-contractive or is precompact in $\cD$ for large $t$
\cite{Sell:EvolutionaryEquations:2002, Temam:InfiniteDimensional:1997}.

As the first step towards proving the asymptotic compactness property stated above,  
existence of bounded absorbing sets for  $\seq{S_{\rw}(t)}_{t \in [0,\infty)}$ and $\seq{S_{\rs}(t)}_{t \in [0,\infty)}$ is established in Theorems \ref{th:AbsorbingSetD1} and \ref{th:AbsorbingSetD2}, respectively. 
However, it turns out that the $\kappa$-contracting property is hard to achieve for the model \eqref{eq:Voltage}--\eqref{eq:Wave} with parameter values in the range given in Table \ref{tb:Parameters}, due to 
the lack of space-dissipative terms in the ordinary differential equations \eqref{eq:Voltage} and \eqref{eq:Current},
the nature of nonlinear couplings in \eqref{eq:Voltage} and \eqref{eq:Current}, 
and the range of values of the biophysical parameters of the model.

The uniform compactness of the component $S_2(t)$ in the decomposition approach stated above is usually verified by establishing energy estimates in more regular function spaces and then deducing compactness from compact embedding theorems. 
This approach, although successfully used in \cite{Marion:SIMA:1989} to prove existence of a global attractor for a coupled ODE-PDE reaction-diffusion system, is not very promising here. 
In \cite{Marion:SIMA:1989}, the ODE subsystem is linear and the energy estimates in a higher regular space are achieved by taking space-derivatives of the ODE's and constructing energy functionals for the resulting equations.
As seen in the proof of Theorem \ref{th:AbsorbingSetD1}, the nonnegativity of $i(x,t)$
is a key property that permits elimination of the sign-indefinite quadratic term in the energy equation of \eqref{eq:Voltage}, which results in the energy variation inequality \eqref{eq:VoltageLyapunov}.
This nonnegativity property, however, is not preserved in the derivative or any other variations of $i(x,t)$, leaving some sign-indefinite quadratic terms in the analysis. 
Moreover, it can be observed from the range of parameter values given in Table \ref{tb:Parameters} that the sign-indefinite nonlinear terms that would appear in the energy equations of any variations of \eqref{eq:Voltage} and \eqref{eq:Current} have significantly larger coefficients than the sign-definite dissipative terms. 
This makes it challenging to balance the terms in the energy functional to absorb the nondissipative terms into dissipative ones. 
Finally, the nonlinear terms appearing in \eqref{eq:Voltage} and \eqref{eq:Current} do not satisfy the usual assumptions, e.g., as in \cite{Efendiev:arXiv:2011}, that enable shaping the energy functional to eliminate the nondissipative terms that would otherwise appear in the equations.

Some other techniques are available in the literature to avoid energy estimations in higher regular spaces. 
In \cite{Ma:IndianaMathJ:2002}, for instance, the notion of $\omega$-limit compactness is used to develop necessary and sufficient conditions for existence of a global attractor. 
This is accomplished by decomposing the phase space into two spaces, one of which being finite-dimensional, and then showing that for every bounded set $\sB \subset \cD$ the canonical projection of $S(t) \sB$ onto the finite-dimensional space is bounded,
and the canonical projection on the complement space remains arbitrarily small for sufficiently large $t \geq t_0$, for some $t_0 = t_0(\sB)>0$. 
These decomposition techniques, however, rely on the spectral decomposition of the space-acting operators to construct the desired phase space decomposition. Such operators do not exist in the ODE subsystems \eqref{eq:Voltage} and \eqref{eq:Current} in our problem.

\subsection{Nonexistence of a Global Attractor} \label{sec:NonexistenceOfAttractor}
As discussed in Section \ref{sec:ChallengesAttractor}, establishing a global attractor for \eqref{eq:Voltage}--\eqref{eq:Wave} is a challenging problem. 
In fact, in this section we show that there exit sets of parameter values, leading to physiologically reasonable behavior in the model, for which the semigroups $\seq{S_{\rw}(t)}_{t \in [0,\infty)}$ and $\seq{S_{\rs}(t)}_{t \in [0,\infty)}$ do not possess a global attractor.

We first use \cite[Prop. 4.7]{Efendiev:arXiv:2011} to prove Theorem \ref{th:Noncompactness} below, which gives sufficient conditions for noncompactness of the equilibrium sets of \eqref{eq:Voltage}--\eqref{eq:Wave} in $\cU_{\rw}$ and $\cU_{\rs}$.
However, before embarking on the technical details of this theorem,
we motivate the main idea using the following intuitive discussion.

Assume that the ODE components \eqref{eq:Voltage} and \eqref{eq:Current} are decoupled from the PDE 
component \eqref{eq:Wave} by freezing $w(x,t)$ in space and time in \eqref{eq:Current}.
In this case, \eqref{eq:Voltage} and \eqref{eq:Current} can be viewed \emph{pointwise} as an uncountable set of dynamical systems governed by ODE's that are enumerated by points $x \in \Omega$.
To distinguish this pointwise view, 
let $(v_x(t), i_x(t))$ denote the solution of the dynamical system located at $x \in \Omega$,
in contrast with $(v(x,t), i(x,t))$ that denotes the solution of the decoupled ODE's \eqref{eq:Voltage} and \eqref{eq:Current} defined over $\Omega$.   
Note that the pointwise-defined dynamical systems are fully decoupled from each other, that means the solutions $(v_x(t), i_x(t))$ 
and $(v_y(t), i_y(t))$ evolve totally independently in time for every $x \neq y \in \Omega$. 

Now, assume further that the decoupled ODE system \eqref{eq:Voltage} and \eqref{eq:Current} possesses more than one equilibrium, two of which denoted by $(v_{\re}, i_{\re})$ and $(v_0, i_0)$.
Then, all pointwise defined dynamical systems correspondingly possess more than one equilibrium, in particular,  $({v_{x}}_{\re}, {i_x}_{\re}) = (v_{\re}(x), i_{\re}(x))$ and $({v_x}_0, {i_x}_0) = (v_0(x), i_0(x))$ 
for the system located at $x$.
This implies that the solutions $(v_x(t), i_x(t))$ can converge independently to different values at different points
$x \in \Omega$.
Therefore, when composed together, they form a solution $(v(x,t), i(x,t))$ for the decoupled ODE system \eqref{eq:Voltage} and \eqref{eq:Current}, which can possibly develop drastic discontinuities over $\Omega$ as it evolves in time.
Note that such discontinuities in the solutions can occur even though the initial values are smooth.
Moreover, it follows in particular that the ODE system \eqref{eq:Voltage} and \eqref{eq:Current} possesses
an uncountable discrete equilibrium set. 
In fact, any function composed arbitrarily of either values $({v_{x}}_{\re}, {i_x}_{\re})$ and $({v_x}_0, {i_x}_0)$ at each point $x \in \Omega$ would be an equilibrium.

The idea of Theorem \ref{th:Noncompactness} is to prove that the space-smoothing effect of the coupling with the PDE component \eqref{eq:Wave} is not sufficiently strong to rule out the discontinuities of the above nature in 
$(v, i)$ and, in particular, having a noncompact equilibrium set.
Define the mappings   
\begin{align} \label{eq:PvPiDefinition}
	P_v(v,i) &:= v - J_1 i + J_2 v i^{\rT} \Psi J_4 + J_3 v i^{\rT} \Psi J_5, \\
	P_i(v,i) &:= (e \Upsilon )^{-1} \Gamma i - \rN J_7 f(v) - g, \nonumber
\end{align}
and let 
$(v_{\re}, i_{\re}, w_{\re})$ be an equilibrium of \eqref{eq:Voltage}--\eqref{eq:Wave}, that is 
$P_v(v_{\re},i_{\re}) = 0$ and $P_i(v_{\re}, i_{\re}) = J_6 w_{\re}$.
Assume there exists $(v_0, i_0) \neq (v_{\re}, i_{\re})$ such that $P_v(v_0,i_0) = 0$ and $P_i(v_0, i_0) = P_i(v_{\re}, i_{\re})$.
In this case,  $(v_{\re}, i_{\re})$ and $(v_0, i_0)$ are both equilibrium of the system \eqref{eq:Voltage} and \eqref{eq:Current} if we assume it is decoupled from \eqref{eq:Wave} by freezing $w$ at $w = w_{\re}$. 
Therefore, motivated by the discussion above, we can construct a new equilibrium $(\bar{v}, \bar{i})$ 
for this decoupled system by letting $(\bar{v}, \bar{i}) = (v_0, i_0)$ over an arbitrary set $\Omega_0$,
and $(\bar{v}, \bar{i}) = (v_{\re}, i_{\re})$ over the complement set $\Omega_{\re}$.
This construction is illustrated in Figure \ref{fig:Noncompactness}.

Since $w$ is not actually frozen at $w = w_{\re}$, the function $(\bar{v}, \bar{i})$ is not necessarily a component of a new equilibrium of the coupled system
\eqref{eq:Voltage}--\eqref{eq:Wave}.
However, if it is ensured that $w$ remains close to $w_{\re}$, then we can expect that there
exists a new equilibrium $(v^{\ast}, i^{\ast}, w^{\ast})$ of \eqref{eq:Voltage}--\eqref{eq:Wave} whose
component $(v^{\ast}, i^{\ast})$ is close to $(\bar{v}, \bar{i})$.
Since the $w$ component of an equilibrium of \eqref{eq:Voltage}--\eqref{eq:Wave} 
is continuous over $\overline{\Omega}$, we may postulate that, 
provided the sets $\Omega_0$ are sufficiently small, updating $(v_{\re}, i_{\re})$ by $(\bar{v}, \bar{i})$
in the equilibrium equations would not largely deviate the $w$ component from $w_{\re}$ and the above expectation is satisfied.   
This postulation is indeed true and it is proved in Theorem \ref{th:Noncompactness} that under certain conditions 
a new equilibrium $(v^{\ast}, i^{\ast}, w^{\ast})$ exists such that $(v^{\ast}, i^{\ast})$ are arbitrarily close to
$(\bar{v}, \bar{i})$ provided $\Omega_0$ is sufficiently small. 
The proof is relatively involved and constitutes the core part of the proof of Theorem \ref{th:Noncompactness}.
It highly relies on the $\cL_w^{\infty}$-boundedness of the space-acting operator $A^{-1}$ that appears in the equilibrium equations, and on Assumption (iv) of Theorem \ref{th:Noncompactness}. 
Figure \ref{fig:Noncompactness} gives and illustration of the component $(v^{\ast}, i^{\ast})$ lying uniformly closer than $\varepsilon$ to $(\bar{v}, \bar{i})$.   

Finally, the noncompactness of the equilibrium set of \eqref{eq:Voltage}--\eqref{eq:Wave} 
follows if we show that the existence of equilibria $(v^{\ast}, i^{\ast}, w^{\ast})$ is uniform with respect to the
shape of the sets $\Omega_0$, that is, as long as only the size of $\Omega_0$ is smaller than a uniform bound. 
In this case, we take $\varepsilon$ small enough such that the distance between $(v_{\re}(x), i_{\re}(x))$ and 
$(v_0(x), i_0(x))$ is larger than $3 \varepsilon$. 
Then, for any two sufficiently small sets $\tilde{\Omega}_0$ and $\hat{\Omega}_0$ we can construct new 
equilibria as discussed above, having components closer than $\varepsilon$ to their associated estimates $(\bar{v}, \bar{i})$. 
It can be observed from Figure \ref{fig:Noncompactness} that the associated components $(v^{\ast}, i^{\ast})$
of these two equilibria would certainly be at a distance larger than $\varepsilon$ from each other at least on the difference of the two sets $\tilde{\Omega}_0$ and $\hat{\Omega}_0$.
Therefore, since this construction is independent of the shape of the sets $\tilde{\Omega}_0$ and $\hat{\Omega}_0$
and we have uncountably different choices for these sets, it follows that we can construct an uncountable 
set of disjoint equilibria. 
This implies the noncompactness of the equilibrium set of \eqref{eq:Voltage}--\eqref{eq:Wave} .
Theorem \ref{th:Noncompactness} below gives rigorous arguments for the above discussion.   

\begin{figure}
	\centering
	\includegraphics[width=0.85\linewidth]{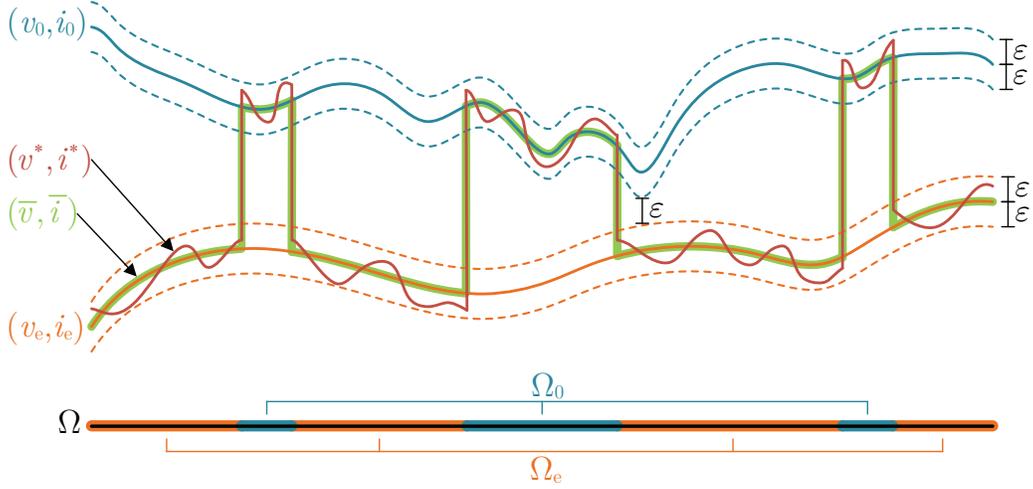}
	\caption{Illustrative construction of new equilibria as given by Theorem \ref{th:Noncompactness}.
		To avoid unnecessary complexities in the graphs, only one representative curve out of the six curves in the
		$(v,i)$ components of the solutions is shown.   }
	\label{fig:Noncompactness}
\end{figure}

\begin{theorem}[Noncompactness of equilibrium sets] \label{th:Noncompactness}
	Suppose $g$ is bounded and constant in time, that is, $g(x,t) = g(x)$ for all $(x,t) \in \Omega \times [0,\infty)$ and $g \in \cL_i^{\infty}$.
	Let $u_{\re} :=(v_{\re}, i_{\re}, 0, w_{\re}, 0)$ 
	be an equilibrium of \eqref{eq:Voltage}--\eqref{eq:Wave}  such that
	$v_{\re} \in \cL_v^{\infty}$, $i_{\re} \in \cL_i^{\infty}$, and $w_{\re} \in \cH_w^2$. 
	Define the mapping $P = (P_v, P_i) : \cL_v^{\infty} \times \cL_i^{\infty} \rightarrow \cL_v^{\infty} \times \cL_i^{\infty}$ as in \eqref{eq:PvPiDefinition}		
	and let $A:= -\frac{3}{2} \Delta + \Lambda^2 I$. 
	Assume that the following conditions hold: 
	\begin{enumerate}
		\item $\Lambda_{\rE \rE}$ and $\Lambda_{\rE \rI}$ take the same values, that is, 
		$\Lambda = \Lambda_{\rE \rE} I_{2 \times 2} = \Lambda_{\rE \rI} I_{2 \times 2}$.
		\item There exists $(v_0,i_0) \in \cL_v^{\infty} \times \cL_i^{\infty}$ such that 
		\begin{equation*}
			\essinf_{x \in \Omega} \norm{\infty}{(v_{\re}(x), i_{\re}(x)) - (v_0(x), i_0(x))} >0
		\end{equation*}
		and
		\begin{equation} \label{eq:v0i0}
			P_v(v_0,i_0) = 0, \quad
			P_i(v_0, i_0) = P_i(v_{\re}, i_{\re}).
		\end{equation}
		\item  $\partial_{(v,i)} P(v_{\re},i_{\re})$ and $\partial_{(v,i)} P(v_0,i_0)$ are nonsingular almost everywhere in $\Omega$.
		\item There exists $\alpha > 0$ such that, for every $b = (b_v,b_i) \in \cL_v^{\infty} \times \cL_i^{\infty}$, the system of equations
		\begin{align} \label{eq:HyperbolicityEquilibrium}
			\partial_{(v,i)} P_v(v_{\re}, i_{\re}) \phi &= b_v,\\
			\partial_{(v,i)} P_i(v_{\re}, i_{\re}) \phi - J_6 A^{-1} \Lambda^2 \rM J_8 \dv f(v_{\re}) \phi_v &= b_i, \nonumber
		\end{align}
		has a unique solution $\phi = (\phi_v, \phi_i) \in \cL_v^{\infty} \times \cL_i^{\infty}$ that satisfies
		\begin{equation} \label{eq:BoundednessOfInverse}
			\norm{\cL_v^{\infty} \times \cL_i^{\infty}}{\phi} \leq \alpha \norm{\cL_v^{\infty} \times \cL_i^{\infty}}{b}.
		\end{equation}
	\end{enumerate}
	Then, for a measurable partition $\Omega = \Omega_{\re} \cup \Omega_0$ and 
	\begin{equation} \label{eq:vBariBar}
		\bar{v} := v_{\re} \chi_{\Omega_{\re}}  + v_0 \chi_{\Omega_0}, \quad \quad
		\bar{i} := i_{\re} \chi_{\Omega_{\re}}  + i_0 \chi_{\Omega_0},
	\end{equation}
	the following assertions hold: 
	\begin{enumerate}
		\renewcommand\theenumi{\Roman{enumi}}
		\item For every $\varepsilon>0$ there exists $\delta > 0$ and an equilibrium 
		$u^{\ast}:=(v^{\ast}, i^{\ast}, 0, w^{\ast}, 0)$ 
		of \eqref{eq:Voltage}--\eqref{eq:Wave} such that
		\begin{equation*}
			\norm{\cL_v^{\infty} \times \cL_i^{\infty}}{(v^{\ast}, i^{\ast}) - (\bar{v}, \bar{i})} \leq \varepsilon, \quad
			\text{whenever } |\Omega_0| \leq \delta.
		\end{equation*}  
		\item The equilibrium sets of \eqref{eq:Voltage}--\eqref{eq:Wave} are noncompact in $\cU_{\rs}$ and $\cU_{\rw}$.
	\end{enumerate}
\end{theorem}
\begin{proof}
	The proof is organized in three steps. 
	
	\emph{Step 1.}
	We show that there exists $\bar{\alpha}>0$ such that, for every $b = (b_v,b_i) \in \cL_v^{\infty} \times \cL_i^{\infty}$, the system of equations 
	\begin{align} \label{eq:HyperbolicityBar}
		\partial_{(v,i)} P_v(\bar{v}, \bar{i}) \phi &= b_v,\\
		\partial_{(v,i)} P_i(\bar{v}, \bar{i}) \phi - J_6 A^{-1} \Lambda^2 \rM J_8 \dv f(\bar{v}) \phi_v &= b_i,
		\nonumber
	\end{align}
	has a unique solution $\phi \in \cL_v^{\infty} \times \cL_i^{\infty}$ 
	that satisfies 
	$\norm{\cL_v^{\infty} \times \cL_i^{\infty}}{\phi} \leq \bar{\alpha} \norm{\cL_v^{\infty} \times \cL_i^{\infty}}{b}$.
	This provides the required conditions of the implicit function theorem that is used in Step 2 to prove the existence of the equilibrium $u^{\ast}$.
	The proof proceeds by iteratively constructing a solution by starting from the solution of \eqref{eq:HyperbolicityEquilibrium} and applying certain corrections at each iteration. 
	
	Let $\phi^{(0)} = (\phi_v^{(0)}, \phi_i^{(0)})$ be the solution of \eqref{eq:HyperbolicityEquilibrium} for a given $b \in \cL_v^{\infty} \times \cL_i^{\infty}$
	and construct an approximate solution for \eqref{eq:HyperbolicityBar} of the form
	$\phi^{(1)} := \phi^{(0)} + \phi_{\rr}^{(1)}$, where $\phi_{\rr}^{(1)} = (\phi^{(1)}_{\rr_v}, \phi^{(1)}_{\rr_i} )$ is the unique solution of
	\begin{align} \label{eq:Phir}
		\partial_{(v,i)} P(v_0, i_0) \phi_{\rr}^{(1)} = \left( \partial_{(v,i)} P(v_{\re}, i_{\re}) - \partial_{(v,i)} P(v_0, i_0) \right) \phi^{(0)} \chi_{\Omega_0}. 
	\end{align}   
	Note that by Assumption (iii) the unique solution $\phi_{\rr}^{(1)}$ exists and belongs to $\cL_v^{\infty} \times \cL_i^{\infty}$.
	The approximate solution $\phi^{(1)}$ solves
	\begin{align*} 
		\partial_{(v,i)} P_v(\bar{v}, \bar{i}) \phi^{(1)} &= b_v,\\
		\partial_{(v,i)} P_i(\bar{v}, \bar{i}) \phi^{(1)} - J_6 A^{-1} \Lambda^2 \rM J_8 \dv f(\bar{v}) \phi_v^{(1)} &=
		b_i + b_{\rr_i}^{(1)}, \nonumber
	\end{align*}
	where $	b_{\rr}^{(1)} = (0, b_{\rr_i}^{(1)})$, with
	\begin{equation} \label{eq:br}
		b_{\rr_i}^{(1)} := J_6 A^{-1} \Lambda^2 \rM J_8 \left[ \left( \dv f(v_{\re}) - \dv f(v_0) \right) \phi_v^{(0)} - \dv f(v_0)
		\phi^{(1)}_{\rr_v} \right] \chi_{\Omega_0}, 
	\end{equation}
	is the remainder resulting from the approximation error in $\phi^{(1)}$.
	
	Now, note that by Assumption (iv) there exist $\alpha_0 := \alpha>0$ such that
	\begin{equation} \label{eq:32716951}
		\norm{\cL_v^{\infty} \times \cL_i^{\infty}}{\phi^{(0)}} \leq \alpha_0 \norm{\cL_v^{\infty} \times \cL_i^{\infty}}{b}.
	\end{equation}
	Moreover, since by Assumption (ii) we have $(v_0,i_0) \in \cL_v^{\infty} \times \cL_i^{\infty}$,  	
	it is immediate from the definition of $P_v$ and $P_i$, given by \eqref{eq:PvPiDefinition}, that 
	$\partial_{(v,i)} P(v_0, i_0)$ is bounded. 
	This, along with Assumption (iii) and \eqref{eq:32716951}, implies that the solution $\phi_{\rr}^{(1)}$ of \eqref{eq:Phir} satisfies
	\begin{equation} \label{eq:32716950}
		\norm{\cL_v^{\infty} \times \cL_i^{\infty}}{\phi_{\rr}^{(1)}} 
		\leq \zeta_1 \norm{\cL_v^{\infty} \times \cL_i^{\infty}}{\phi^{(0)}}
		\leq \alpha_1 \norm{\cL_v^{\infty} \times \cL_i^{\infty}}{b}
	\end{equation}
	for some $\zeta_1, \alpha_1>0$.
	
	Next, note that since $A^{-1}: \cL_w^2 \rightarrow \cH_w^2$ is a bounded operator and $f$ is smooth,
	the definition of $b_{\rr_i}^{(1)}$, given by \eqref{eq:br}, implies that $b_{\rr_i}^{(1)} \in H_{\rm per}^2(\Omega;\bbR^4)$.  
	Moreover, it further implies by the Sobolev embedding theorems \cite[Th. 6.6-1]{Ciarlet:FunctionalAnalysis:2013} that $b_{\rr_i}^{(1)} \in C_{\rm per}^{0,\lambda}(\overline{\Omega}, \bbR^4)$ for all $\lambda \in (0,1)$ and, in particular, 
	$\norm{\cL_i^{\infty}}{b_{\rr_i}^{(1)}} \leq \zeta_2 \norm{H_{\rm per}^2(\Omega;\bbR^4)}{b_{\rr_i}^{(1)}}$
	for some $\zeta_2>0$.
	Therefore, using \eqref{eq:32716951} and \eqref{eq:32716950},
	there exist $\zeta_3, \zeta_4, \zeta_5, \beta_1 >0$ such that
	\begin{align} \label{eq:32416624}
		\norm{\cL_v^{\infty} \times \cL_i^{\infty}}{b_{\rr}^{(1)}} 
		&\leq \zeta_2 \norm{H_{\rm per}^2(\Omega;\bbR^4)}{b_{\rr_i}^{(1)}}   
		\leq \zeta_3 \left( \norm{\cL_v^2}{\phi_v^{(0)}} + \norm{\cL_v^2}{\phi^{(1)}_{\rr_v}} \right) \\
		&\leq \zeta_4 \norm{\cL_v^2 \times \cL_i^2}{\phi^{(0)}}    
		\leq \zeta_5 |\Omega_0|^{\frac{1}{2}} \norm{\cL_v^{\infty} \times \cL_i^{\infty}}{\phi^{(0)}} \nonumber\\
		&\leq \beta_1 |\Omega_0|^{\frac{1}{2}} \norm{\cL_v^{\infty} \times \cL_i^{\infty}}{b}.\nonumber
	\end{align}
	
	Now, for $m=2,3,\dots$, let $\phi^{(m)} := \phi^{(m-1)} + \phi_{\rr}^{(m)}$, where $\phi_{\rr}^{(m)}$ is the unique solution of
	\begin{equation*}
		\partial_{(v,i)} P(v_0, i_0) \phi_{\rr}^{(m)} = - b_{\rr}^{(m-1)} \chi_{\Omega_0}. 
	\end{equation*}
	It follows immediately that, for some $\eta>0$,
	\begin{equation} \label{eq:32716948}
		\norm{\cL_v^{\infty} \times \cL_i^{\infty}}{\phi_{\rr}^{(m)}} 
		\leq \eta \norm{\cL_v^{\infty} \times \cL_i^{\infty}}{b_{\rr}^{(m-1)}}, \quad m=2,3,\dots.
	\end{equation}
	Moreover, $\phi_{\rr}^{(m)}$ solves the system of equations 
	\begin{align*} 
		\partial_{(v,i)} P_v(\bar{v}, \bar{i}) \phi^{(m)} &= b_v,\\
		\partial_{(v,i)} P_i(\bar{v}, \bar{i}) \phi^{(m)} - J_6 A^{-1} \Lambda^2 \rM J_8 \dv f(\bar{v}) \phi_v^{(m)} &=
		b_i + b_{\rr_i}^{(m)}, \nonumber
	\end{align*}
	where
	\begin{equation*} 
		b_{\rr_i}^{(m)} := - J_6 A^{-1} \Lambda^2 \rM J_8 \dv f(v_0) \phi^{(m)}_{\rr_v} \chi_{\Omega_0}, \quad m=2,3,\dots.
	\end{equation*}
	Using the Sobolev embedding theorems and \eqref{eq:32716948}, 
	the remainder $b_{\rr}^{(m)} = (0, b_{\rr_i}^{(m)})$ 
	satisfies, for some $\zeta_6, \zeta_7, \zeta_8, \beta>0$,
	\begin{align*} 
		\norm{\cL_v^{\infty} \times \cL_i^{\infty}}{b_{\rr}^{(m)}} 
		&\leq \zeta_6 \norm{H_{\rm per}^2(\Omega;\bbR^4)}{b_{\rr_i}^{(m)}}   
		\leq \zeta_7 \norm{\cL_v^2 \times \cL_i^2}{\phi^{(m)}_{\rr}} 
		\leq \zeta_8  |\Omega_0|^{\frac{1}{2}} \norm{\cL_v^{\infty} \times \cL_i^{\infty}}{\phi^{(m)}_{\rr}} \\
		&\leq \beta |\Omega_0|^{\frac{1}{2}} \norm{\cL_v^{\infty} \times \cL_i^{\infty}}{b_{\rr}^{(m-1)}}, \quad m=2,3,\dots, \nonumber
	\end{align*}
	which, letting $\kappa:= \beta |\Omega_0|^{\frac{1}{2}}$ and recalling \eqref{eq:32416624}, implies
	\begin{equation} \label{eq:427171054}
		\norm{\cL_v^{\infty} \times \cL_i^{\infty}}{b_{\rr}^{(m)}}  \leq \beta_1 |\Omega_0|^{\frac{1}{2}} \kappa^{(m-1)} \norm{\cL_v^{\infty} \times \cL_i^{\infty}}{b}  \quad m=2,3,\dots. 
	\end{equation}
	
	Now, let $|\Omega_0|< \bar{\delta}$, $\bar{\delta}>0$, and choose $\bar{\delta}$ such that $\kappa <1$. 
	Note that $\beta$, and consequently, the choice of $\bar{\delta}$ and the value of $\kappa$ do not depend on $b$ and the specific form of the partition $\Omega = \Omega_{\re} \cup \Omega_0$. 
	Therefore, it follows that $\norm{\cL_v^{\infty} \times \cL_i^{\infty}}{b_{\rr}^{(m)}} \rightarrow 0$ as $m \rightarrow \infty$, and hence, $\phi^{(m)}$ converges to a solution $\phi$ for \eqref{eq:HyperbolicityBar} when $|\Omega_0|< \bar{\delta}$.	
	Moreover, \eqref{eq:32716951}--\eqref{eq:427171054} imply 
	\begin{align*}
		\norm{\cL_v^{\infty} \times \cL_i^{\infty}}{\phi^{(m)}} 
		&\leq \norm{\cL_v^{\infty} \times \cL_i^{\infty}}{\phi^{(0)}}
		+ \norm{\cL_v^{\infty} \times \cL_i^{\infty}}{\phi_{\rr}^{(1)}}
		+ \sum_{l=2}^m \norm{\cL_v^{\infty} \times \cL_i^{\infty}}{\phi_{\rr}^{(l)}}\\
		& \leq \left[ \alpha_0 + \alpha_1 + \eta \beta_1 |\Omega_0|^{\frac{1}{2}}  \sum_{l=2}^m  \kappa^{(l-2)}  \right] \norm{\cL_v^{\infty} \times \cL_i^{\infty}}{b}, 
	\end{align*}
	and hence, taking the limit as $m \rightarrow \infty$, there exists $\bar{\alpha}>0$, independent of the form of the partition, such that 
	\begin{equation} \label{eq:327161027}
		\norm{\cL_v^{\infty} \times \cL_i^{\infty}}{\phi}	\leq \bar{\alpha} \norm{\cL_v^{\infty} \times \cL_i^{\infty}}{b}.
	\end{equation}
	
	To prove the solution constructed above for \eqref{eq:HyperbolicityBar} is unique, first note that by Assumption (i) the operator $A$ becomes a scalar operator given by
	$A = (-\frac{3}{2} \Delta + \Lambda_{\rE \rE}^2I)$. 
	Then, considering the structure of the matrix parameters given by \eqref{eq:Parameters} and reinspecting the expanded form \eqref{eq:Model}, 
	the system of equations \eqref{eq:HyperbolicityBar} can be transformed to a system composed of five algebraic equations and one partial differential equation by pre-multiplying the second equation in \eqref{eq:HyperbolicityBar} by the elementary matrix
	\begin{equation*}
		\left[ 
		\begin{array}{c|c}
			\hspace{-0.35cm}
			\begin{array}{c c}
				1 & 0\\
				- \frac{M_{\rE \rI}}{M_{\rE \rE}}& 1 \vspace{0.1cm}
			\end{array} & 0_{2 \times 2} \\  \hline
			0_{2 \times 2} & I_{2 \times 2}
		\end{array} \right].
	\end{equation*}
	This follows from the fact that the scalar operator $(-\frac{3}{2} \Delta + \Lambda_{\rE \rE}^2 I)^{-1}$ acts only on one of the unknowns, namely, $\phi_{v_{\rE}}$. 
	Now, since $\partial_{(v,i)}P(\bar{v}, \bar{i})$ is nonsingular by Assumption (iii), the five unknowns 
	$\phi_i=(\phi_{i_{\rE \rE}}, \phi_{i_{\rE \rI}}, \phi_{i_{\rI \rE}}, \phi_{i_{\rI \rI}})$ 
	and $\phi_{v_{\rI}}$ can be uniquely determined in terms of $\phi_{v_{\rE}}$ by elementary algebraic operations. 
	Consequently, \eqref{eq:HyperbolicityBar} is reduced to a scalar partial differential equation of the form
	\begin{equation*}
		p(\bar{v}, \bar{i}) \phi_{v_{\rE}} - (-\tfrac{3}{2} \Delta + \Lambda_{\rE \rE}^2 I)^{-1} \Lambda_{\rE \rE}^2 \rM_{\rE \rE} \partial_{v_{\rE}} f(\bar{v}_{\rE}) \phi_{v_{\rE}} = \hat{h}, 
	\end{equation*}
	where $\hat{h} \in L_{\rm per}^{\infty}(\Omega, \bbR)$ is given by the same elementary operations on $b$
	and $p(\bar{v}, \bar{i})$ is nonzero almost everywhere in $\Omega$, since elementary operations do not disrupt the nonsingularity of $\partial_{(v,i)}P(\bar{v}, \bar{i})$. 
	
	Next, dividing by $p(\bar{v}, \bar{i})$, the above equation can be written as
	\begin{equation} \label{eq:CompactForm}
		(I - K) \phi_{v_{\rE}} = h,
	\end{equation}
	where 
	$K:= p(\bar{v}, \bar{i})^{-1} \Lambda_{\rE \rE}^2 \rM_{\rE \rE} \partial_{v_{\rE}} f(\bar{v}_{\rE}) (-\tfrac{3}{2} \Delta + \Lambda_{\rE \rE}^2 I)^{-1} $
	and $h := p(\bar{v}, \bar{i})^{-1} \hat{h}$.
	The operator 
	$K: L_{\rm per}^2(\Omega, \bbR) \rightarrow  L_{\rm per}^2(\Omega, \bbR)$ is linear, self-adjoint, and compact by the Rellich-Kondrachov compact embedding theorems\cite[Th. 6.6-3]{Ciarlet:FunctionalAnalysis:2013}. 
	The existence of solutions of \eqref{eq:HyperbolicityBar} proved above guaranteers the existence of a solution $\phi_{v_{\rE}} \in L_{\rm per}^{\infty}(\Omega, \bbR)$ 
	for every $h \in L_{\rm per}^{\infty}(\Omega, \bbR)$, which implies, $L_{\rm per}^{\infty}(\Omega, \bbR) \subset \range(I - K)$. 
	However,   
	$\range(I - K) = \kernel(I - K^{\ast})^{\perp} =  \kernel(I - K)^{\perp}$ 
	by the Fredholm alternative \cite[Th. 5, Appx. D]{Evans:PDE:2010}, and hence,
	$L_{\rm per}^{\infty}(\Omega, \bbR) \cap \kernel(I - K) = \{0\}$. 
	This proves the uniqueness of bounded solutions of \eqref{eq:CompactForm}, and consequently, the uniqueness of solutions of \eqref{eq:HyperbolicityBar} for every $b = (b_v,b_i) \in \cL_v^{\infty} \times \cL_i^{\infty}$.      
	
	\emph{Step 2.}
	We prove Assertion (I) using the implicit function theorem. 
	Note that since $u_{\re}:=(v_{\re}, i_{\re}, 0, w_{\re}, 0)$ 
	is an equilibrium  of \eqref{eq:Voltage}--\eqref{eq:Wave}, we have
	\begin{equation} \label{eq:EquilibriumEquations}
		P_v(v_{\re},i_{\re}) = 0, \quad 
		P_i(v_{\re}, i_{\re}) = J_6 w_{\re} ,\quad
		w_{\re} = A^{-1} \Lambda^2 \rM J_8 f(v_{\re}).
	\end{equation}
	We seek an equilibrium point $u^{\ast}:=(v^{\ast}, i^{\ast}, 0, w^{\ast}, 0)$ 
	such that 
	\begin{equation*}
		v^{\ast} = \bar{v} + \phi_v, \quad
		i^{\ast} = \bar{i} + \phi_i, 
	\end{equation*}
	where $\phi := (\phi_v, \phi_i) \in \cL_v^{\infty} \times \cL_i^{\infty}$ is a small corrector function
	that satisfies 
	\begin{equation} \label{eq:EquilibriumStar}
		P_v(v^{\ast},i^{\ast}) = 0, \quad 
		P_i(v^{\ast}, i^{\ast}) = J_6 w^{\ast} ,\quad
		w^{\ast} = A^{-1} \Lambda^2 \rM J_8 f(v^{\ast}).
	\end{equation}
	Note that \eqref{eq:v0i0}, \eqref{eq:vBariBar}, and \eqref{eq:EquilibriumEquations} imply
	\begin{equation*}
		P_v(\bar{v},\bar{i}) = 0, \quad P_i(\bar{v}, \bar{i}) = J_6 w_{\re}, \quad v_{\re} = \bar{v}-(v_0-v_{\re})\chi_{\Omega_0}.
	\end{equation*}
	Therefore, the system of equations \eqref{eq:EquilibriumStar} is equivalent to
	\begin{align} \label{eq:EquilibriumStarEquation}
		P_v(\bar{v} + \phi_v, \bar{i} + \phi_i) - P_v(\bar{v}, \bar{i}) &= 0,\\
		P_i(\bar{v} + \phi_v, \bar{i} + \phi_i) - P_i(\bar{v}, \bar{i}) 
		&= J_6 A^{-1} \Lambda^2 \rM J_8 \big(f(\bar{v} + \phi_v) - f(\bar{v}-(v_0-v_{\re})\chi_{\Omega_0})\big), \nonumber
	\end{align}
	which, by the implicit function theorem \cite[Th. 7.13-1]{Ciarlet:FunctionalAnalysis:2013}, has a unique solution $\phi \in \cL_v^{\infty} \times \cL_i^{\infty}$ 
	since \eqref{eq:HyperbolicityBar} has a unique solution in $\cL_v^{\infty} \times \cL_i^{\infty}$ for every $b \in \cL_v^{\infty} \times \cL_i^{\infty}$, as proved in Step 1.
	Moreover, it is immediate from the definition of the Fr\'{e}chet derivative of the mappings $P_i$ and $P_v$ that the solution of  \eqref{eq:EquilibriumStarEquation} is arbitrarily close to the solution of \eqref{eq:HyperbolicityBar} with
	\begin{equation*}
		b := (0, J_6 A^{-1} \Lambda^2 \rM J_8 \dv f(\bar{v}) (v_0-v_{\re})) \chi_{\Omega_0},
	\end{equation*}
	provided these solutions are sufficiently small. 
	This is ensured by \eqref{eq:327161027} for small $|\Omega_0|$, since 
	$\norm{\cL_v^{\infty} \times \cL_i^{\infty}}{b} \leq \zeta |\Omega_0|^{\frac{1}{2}}$ for some $\zeta >0$. 
	Therefore, it follows that Assertion (I) holds for some $\delta = \delta(\varepsilon) \leq \bar{\delta}$. 
	
	\emph{Step 3.} We prove Assertion (II) using the fact that $\delta = \delta(\varepsilon) > 0$
	in Assertion (I) is independent of the specific form of the partition $\Omega = \Omega_{\re} \cup \Omega_0$. 
	Figure \ref{fig:Noncompactness} can be used to visualize the arguments of the proof.
	
	Let
	\begin{equation} \label{eq:101161038}
		\varepsilon :=\tfrac{1}{3} \essinf_{x \in \Omega} \norm{\infty}{(v_{\re}(x), i_{\re}(x)) - (v_0(x), i_0(x))} >0
	\end{equation}
	in Assertion (I), and let $\delta = \delta(\varepsilon) > 0$ be the corresponding bound on the size of the partitions that satisfies the result of Assertion (I). Note that $\varepsilon > 0$ by Assumption (ii).
	Moreover, let $\sM(\Omega)$ denote the set of all measurable subsets of $\Omega$ and define
	\begin{equation*}
		\sP_\delta(\Omega):=\set{(\Omega_{\re}, \Omega_0) \in \sM(\Omega) \times \sM(\Omega)}{ \Omega_{\re}= \Omega \setminus \Omega_0,  |\Omega_0|\leq \delta}.
	\end{equation*}
	Let $\Theta_\delta(\Omega) \subset \sP_\delta(\Omega)$ such that for every 
	$\tilde{\theta} = (\tilde{\Omega}_{\re}, \tilde{\Omega}_0) \in \Theta_\delta(\Omega)$ and
	$\hat{\theta} = (\hat{\Omega}_{\re}, \hat{\Omega}_0) \in \Theta_\delta(\Omega)$ we have $|\tilde{\Omega}_0 \bigtriangleup \hat{\Omega}_0| > \frac{1}{2}\delta$. 
	Note that $\Theta_\delta(\Omega)$ is an uncountable set that can be viewed as an index set enumerating all measurable partitions $\Omega = \Omega_{\re} \cup \Omega_0$, $|\Omega_0| \leq \delta$, which are distinct in the sense of measure by a factor of at least $\frac{1}{2}\delta$.
	
	Now, it follows from Assertion (I) that, for every $\tilde{\theta} \neq \hat{\theta} \in \Theta_\delta(\Omega)$, there exist 
	equilibria $u_{\tilde{\theta}}:=(v_{\tilde{\theta}}, i_{\tilde{\theta}}, 0, w_{\tilde{\theta}}, 0)$ 
	and $u_{\hat{\theta}}:=(v_{\hat{\theta}}, i_{\hat{\theta}}, 0, w_{\hat{\theta}}, 0)$ such that
	\begin{align*}
		\esssup_{x \in (\tilde{\Omega}_{\re} \cap  \hat{\Omega}_0)} \norm{\infty}{(v_{\hat{\theta}}(x), i_{\hat{\theta}}(x)) - (v_0(x), i_0(x)) } &\leq \varepsilon, \\	
		\esssup_{x \in (\tilde{\Omega}_{0} \cap \hat{\Omega}_{\re})} \norm{\infty}{(v_{\hat{\theta}}(x), i_{\hat{\theta}}(x)) - (v_{\re}(x), i_{\re}(x)) } &\leq \varepsilon, \\
		\esssup_{x \in (\tilde{\Omega}_{\re} \cap  \hat{\Omega}_0)} \norm{\infty}{(v_{\tilde{\theta}}(x), i_{\tilde{\theta}}(x)) - (v_{\re}(x), i_\re)) } &\leq \varepsilon,\\
		\esssup_{x \in (\tilde{\Omega}_{0} \cap \hat{\Omega}_{\re})} \norm{\infty}{(v_{\tilde{\theta}}(x), i_{\tilde{\theta}}(x)) - (v_0(x), i_0)) } &\leq \varepsilon.
	\end{align*}    
	Therefore, noting that $\tilde{\Omega}_0 \bigtriangleup \hat{\Omega}_0 = (\tilde{\Omega}_{0} \cap \hat{\Omega}_{\re}) \cup (\tilde{\Omega}_{\re} \cap  \hat{\Omega}_0)$ and recalling the definition of $\varepsilon$ given by \eqref{eq:101161038},
	\begin{equation*}
		\esssup_{x \in (\tilde{\Omega}_0 \bigtriangleup \hat{\Omega}_0)} \norm{\infty} {(v_{\tilde{\theta}}(x), i_{\tilde{\theta}}(x)) - (v_{\hat{\theta}}(x), i_{\hat{\theta}}(x))} \geq \varepsilon,
	\end{equation*}
	which further implies
	\begin{equation*}
		\norm{\cL_v^2 \times \cL_i^2} {(v_{\tilde{\theta}}, i_{\tilde{\theta}}) - (v_{\hat{\theta}}, i_{\hat{\theta}})} \geq | \tilde{\Omega}_0 \bigtriangleup \hat{\Omega}_0|^{\frac{1}{2}} 
		\esssup_{x \in (\tilde{\Omega}_0 \bigtriangleup \hat{\Omega}_0)} \norm{\infty} {(v_{\tilde{\theta}}(x), i_{\tilde{\theta}}(x)) - (v_{\hat{\theta}}(x), i_{\hat{\theta}}(x))} 
		> (\tfrac{1}{2} \delta)^{\frac{1}{2}}\varepsilon.
	\end{equation*} 
	Since $\tilde{\theta}$ and $\hat{\theta}$ are arbitrary, it follows that the set 
	$\cE:= \seq{u_{\theta}}_{\theta \in \Theta_\delta(\Omega)}$ 
	composed of the equilibria $u_{\theta}$ constructed as above is an uncountable discrete subset of the equilibrium sets of \eqref{eq:Voltage}--\eqref{eq:Wave} in $\cU_{\rs}$ and $\cU_{\rw}$. This completes the proof. 
	\qquad 
\end{proof}

\begin{remark}[Alternative assumptions for Theorem \ref{th:Noncompactness}] \label{rem:AlternativeAssumptions}
	According to the proof of Theorem \ref{th:Noncompactness}, some of the assumptions of this theorem can be relaxed or replaced by alternative assumptions as follows: 
	\begin{itemize}
		\item Assumption (i) is used to prove the uniqueness of solutions of \eqref{eq:HyperbolicityBar}. 
		Without this assumption, the operator $A$ is not a scalar operator and \eqref{eq:HyperbolicityBar} cannot be reduced to a scalar partial differential equation using elementary algebraic operations. 
		The operator $K$ representing the system of PDE's in this case would not be self-adjoint, and hence, application of the Fredholm alternative would not immediately imply uniqueness of the solutions. 
		However, an alternative assumption to Assumption (i) can be made on the adjoint of the operator $K$, so that the uniqueness of the solutions of \eqref{eq:HyperbolicityBar} is still ensured using the Fredholm alternative. 
		We avoid this complication since the fiber decay scale constants $\Lambda_{\rE \rE}$ and $\Lambda_{\rE \rI}$ are always assumed to be equal in the practical applications of the model \cite{Bojak:PhysRev:2005}.
		\item In Assumption (ii), it suffices to have $\essinf_{x \in \sX} \norm{\infty}{(v_{\re}(x), i_{\re}(x) - (v_0(x), i_0(x))} >0$, where $\sX$ is any measurable subset of $\Omega$ with positive measure. 
		Correspondingly, it suffices that the nonsingularity in Assumption (iii) holds almost everywhere on an open subset $\sY \supset \sX$ of $\Omega$. 
		In this case, the proof is modified by restricting $\sP_\delta(\Omega)$ to its subset consisting of partitions with $\Omega_0 \subset \sX$. 
		The index set $\Theta_\delta(\Omega)$ remains uncountable, and the noncompactness result of the theorem holds with no change. 
	\end{itemize}
\end{remark}

\begin{table}[t]
	\vspace{0.4cm}
	\caption{ A set of biophysically plausible parameter values for the model \eqref{eq:Model} for which Theorem \ref{th:Noncompactness} implies nonexistence of a global attractor \cite[Table VI, Col. 2]{Bojak:PhysRev:2005}. 
		The parameters $\bar{g}_{\rE \rE}$, $\bar{g}_{\rE \rI}$, $\bar{g}_{\rI \rE}$, and $\bar{g}_{\rI \rI}$ are, respectively, the mean values of the physiologically shaped random inputs $g_{\rE \rE}$, $g_{\rE \rI}$, $g_{\rI \rE}$, and $g_{\rI \rI}$ used in \cite{Bojak:PhysRev:2005}.	
	} \label{tb:NoncompactnessParameters}
	\begin{center}\footnotesize
		\renewcommand{\arraystretch}{1.3}
		\begin{tabular}{|{l}*{8}{c}|}\hline
			\rowcolor{LightMaroon}
			\textbf{Parameter} & 
			$\tau_{\rE}$ & $\tau_{\rI}$ & $\rV_{\rE \rE}$ & $\rV_{\rE \rI}$  
			& $\rV_{\rI \rE}$ & $\rV_{\rI \rI}$ & $\gamma_{\rE \rE}$ & $\gamma_{\rE \rI}$ \\
			\textbf{Value} & $11.787${\tiny$\times 10^{-3}$} & $138.25${\tiny $\times 10^{-3}$} & $61.264$ & $51.703$ & $-7.127$ & $-12.679$ & $816.04$ & $261.29$\\
			\rowcolor{LightMaroon}
			\textbf{Parameter} &
			$\gamma_{\rI \rE}$	& $\gamma_{\rI \rI}$ & 
			$\Upsilon_{\rE \rE}$ & $\Upsilon_{\rE \rI}$ & $\Upsilon_{\rI \rE}$ & 
			$\Upsilon_{\rI \rI}$ & $\rN_{\rE \rE}$ & $\rN_{\rE \rI}$ \\
			\textbf{Value} & $219.09$ & $40.575$ & $0.92695$ & $1.3012$ & $0.19053$ & $0.94921$ & $3893.0$ & $3326.8$\\		
			\rowcolor{LightMaroon}	
			\textbf{Parameter} & 
			$\rN_{\rI \rE}$ & $\rN_{\rI \rI}$ &			
			$\nu$ & $\Lambda_{\rE \rE}, \Lambda_{\rE \rI}$ & $\rM_{\rE \rE}$ & $\rM_{\rE \rI}$ & $\rF_{{\rE}}$ & $\rF_{{\rI}}$  \\
			\textbf{Value} & $839.39$ & $682.41$ & $101.78$ & $0.96545$ & $4013.5$ & $1544.3$ & $266.44$ & $300.65$\\
			\rowcolor{LightMaroon}
			\textbf{Parameter} & 
			$\mu_{\rE}$ & $\mu_{\rI}$ & $\sigma_{\rE}$ & $\sigma_{\rI}$ & $\bar{g}_{\rE \rE}$ & $\bar{g}_{\rE \rI}$
			& $\bar{g}_{\rI \rE}$ & $\bar{g}_{\rI \rI}$\\
			\textbf{Value} & $30.628$ & $19.383$ & $5.6536$ & $3.3140$ & $83.190$ & $6407.5$ & $0$ & $0$\\
			\hline
		\end{tabular}
	\end{center}
\end{table}

\begin{remark} [Nonexistence of a Global Attractor] \label{rem:PlausibilityOfEquilibrium}
	Suppose that the assumptions of Theorem \ref{th:Noncompactness} hold for an input $g$ and an equilibrium $u_{\re}$ that further satisfy $i_{\re}, w_{\re} >0$ almost everywhere in $\Omega$ and $g \in \cD_g$, where $\cD_g$ is given by \eqref{eq:InputPositiveRegion}.
	Note that $u_{\re}$ then belongs to $\cD_{\rs}$.
	Then, the equation $P_i(v_{\re}, i_{\re}) = J_6 w_{\re}$ in the equilibrium equations \eqref{eq:EquilibriumEquations} implies that $P_i(v_{\re}, i_{\re}) \geq 0$, and hence, $P_i(v_0, i_0) \geq 0$ in \eqref{eq:v0i0}. 
	Therefore, it follows from the definition of $P_i$ given by \eqref{eq:PvPiDefinition} that every solution $i_0$ of \eqref{eq:v0i0} is positive almost everywhere in $\Omega$.
	Then, by definition of $(\bar{v}, \bar{i})$, given by \eqref{eq:vBariBar}, all equilibria $u^{\ast}$ constructed by Assertion (I) of Theorem \ref{th:Noncompactness} satisfy $i^{\ast} >0$ almost everywhere in $\Omega$ when $\delta$ is sufficiently small. 
	Also, the equilibrium equations $w_{\re} = A^{-1} \Lambda^2 \rM J_8 f(v_{\re})$ and $w^{\ast} = A^{-1} \Lambda^2 \rM J_8 f(v^{\ast})$ imply that
	\begin{equation*}
		\norm{\cL_w^{\infty}}{w^{\ast} - w_{\re}} \leq \beta_1 \norm{\cH_w^2}{w^{\ast} - w_{\re}}
		\leq \beta |\Omega|^{\frac{1}{2}} \norm{\cL_v^{\infty}}{v^{\ast} - v_{\re}}
	\end{equation*} 
	for some $\beta>0$, and hence, $w^{\ast} >0$ almost everywhere in $\Omega$, when $\delta$ is sufficiently small.
	Therefore, Assertion (II) of Theorem \ref{th:Noncompactness} ensures existence of a biophysically plausible noncompact set of equilibria $\cE \subset \cD_{\rs} \subset \cD_{\rw}$. 
	This, in particular, implies that in the case where the assumptions of Theorem \ref{th:Noncompactness} are satisfied for  some $u_{\re}$ and $g$ as given above, the semigroups $\seq{S_{\rw}(t):\cD_{\rw} \rightarrow \cD_{\rw} }_{t \in [0,\infty)}$ and $\seq{S_{\rs}(t):\cD_{\rs} \rightarrow \cD_{\rs}}_{t \in [0,\infty)}$ are not asymptotically compact, and hence, 
	they do not possess a global attractor.  
\end{remark}

The assumptions of Theorem \ref{th:Noncompactness} are relatively straightforward to check for the space-homogeneous equilibria of \eqref{eq:Voltage}--\eqref{eq:Wave}. 
Consider the set of values given in Table \ref{tb:NoncompactnessParameters} for the parameters of the model, which are suggested in \cite[Table VI, col. 2]{Bojak:PhysRev:2005} as a set of parameter values leading to physiologically reasonable behavior of the model. 
The parameters $\bar{g}_{\rE \rE}$, $\bar{g}_{\rE \rI}$, $\bar{g}_{\rI \rE}$, and $\bar{g}_{\rI \rI}$ are the mean values of the physiologically shaped random signals used in \cite{Bojak:PhysRev:2005} as the subcortical inputs $g_{\rE \rE}$, $g_{\rE \rI}$, $g_{\rI \rE}$, and $g_{\rI \rI}$, respectively. 
Here, we set $g(t, x) = (\bar{g}_{\rE \rE}, \bar{g}_{\rE \rI}, \bar{g}_{\rI \rE}, \bar{g}_{\rI \rI})$ for all $x$ and $t$, and check the assumptions of Theorem \ref{th:Noncompactness} for a space-homogeneous equilibrium of \eqref{eq:Voltage}--\eqref{eq:Wave}.

Assumption (i) holds with $\Lambda_{\rE \rE} = \Lambda_{\rE \rE} = 0.96545$, as given in Table \ref{tb:NoncompactnessParameters}.
Solving the equations
$P_v(v_{\re},i_{\re}) = 0$, 
$P_i(v_{\re}, i_{\re}) = J_6 w_{\re}$ and 
$w_{\re} = \rM J_8 f(v_{\re})$, a space-homogeneous equilibrium is calculated as
\begin{equation*}
	v_{\re}=(1.9629, 6.5150), \;\, i_{\re}=(5.2552, 100.2372, 2.4493, 53.5665), \:\, w_{\re} = (821.7136, 316.1760).
\end{equation*}
Note that the numbers given here should actually be regarded as constant functions over $\Omega$.
Assumption (ii) then holds by finding a solution $(v_0,i_0) \neq (v_{\re},i_{\re})$ for \eqref{eq:v0i0} as
\begin{equation*}
	v_0 = (10.9417, 7.7148), \quad i_0 = (25.9005, 177.5837, 4.0757, 89.1352).
\end{equation*}
Assumption (iii) also holds with the following nonsingular matrix-valued functions
\begin{align*}
	\partial_{(v,i)} P(v_{\re},i_{\re}) &= \left[
	\begin{array}{cc|cccc} 
		1.4294 & 0 & -0.9680 & 0 & 1.2754 & 0\\ 
		0 & 7.1635 & 0 & -0.8740 & 0 & 1.5138\\ \hline
		-199.2222 & 0 & 323.8625 & 0 & 0 & 0\\ 
		-170.2472 & 0 & 0 & 73.8727 & 0 & 0\\ 
		0 & -440.3409 & 0 & 0 & 423.0237 & 0\\ 
		0 & -357.9898 & 0 & 0 & 0 & 15.7254
	\end{array} \right],\\
	\partial_{(v,i)} P(v_0,i_0) &= \left[
	\begin{array}{cc|cccc} 
		1.9946 & 0 & -0.8214 & 0 & 2.5352 & 0\\ 
		0 & 11.4648 & 0 & -0.8508 & 0 & 1.6085\\ \hline 
		-1858.395 & 0 & 323.8625 & 0 & 0 & 0\\ 
		-1588.109 & 0 & 0 & 73.8727 & 0 & 0\\ 
		0 & -730.7260 & 0 & 0 & 423.0237 & 0\\ 
		0 & -594.0680 & 0 & 0 & 0 & 15.7254
	\end{array}\right].
\end{align*}

To check Assumption (iv), note that for every 
$b = (b_v, b_i) \in \cL_v^{\infty} \times \cL_i^{\infty}$, elementary algebraic operations reduce \eqref{eq:HyperbolicityEquilibrium} to 
\begin{alignat}{3} \label{eq:41416656}
	\phi_{v_{\rE}} &= 0.6287 \phi {i_{\rE \rE}} + h_{v_{\rE}}, &\quad
	\phi_{v_{\rI}} &= 0.0521 \phi {i_{\rE \rE}} + h_{v_{\rI}}, && \\ \nonumber
	\phi_{i_{\rE \rI}} &= 2.4834 \phi {i_{\rE \rE}} + h_{i_{\rE \rI}}, &\quad
	\phi_{i_{\rI \rE}} &= 0.0543 \phi {i_{\rE \rE}} + h_{i_{\rI \rE}}, &\quad
	\phi_{i_{\rI \rI}} &= 1.1870 \phi {i_{\rE \rE}} + h_{i_{\rI \rI}}, 
\end{alignat}
and the scalar partial differential equation
\begin{equation} \label{eq:41416649}
	(I-D) \phi_{i_{\rE \rE}} = h_{i_{\rE \rE}},\quad  D:=0.6060  (-\tfrac{3}{2} \Delta + 0.96545^2 I)^{-1},
\end{equation}  
where $h = (h_v, h_i) \in \cL_v^{\infty} \times \cL_i^{\infty}$ is the result of the same algebraic operations on $b$.
Now, note that since $-\Delta$ is a nonnegative operator in $ H_{\rm per}^2(\Omega; \bbR)$,
it follows from the spectral theory of bounded linear self-adjoint operators \cite[Appx. D.6]{Evans:PDE:2010} that the spectrum of the operator $(I-D):L_{\rm per}^2(\Omega; \bbR) \rightarrow L_{\rm per}^2(\Omega; \bbR)$ lies entirely above  $1 - 0.6060  \times 0.96545^{-2} = 0.3498 >0$. 
Therefore, the partial differential equation \eqref{eq:41416649} has a unique solution
$\phi_{i_{\rE \rE}} \in L_{\rm per}^2(\Omega; \bbR)$
for every $h_{i_{\rE \rE}} \in L_{\rm per}^2(\Omega; \bbR) \supset L_{\rm per}^{\infty}(\Omega; \bbR)$, and hence, it follows from \eqref{eq:41416656} that \eqref{eq:HyperbolicityEquilibrium} has a unique solution 
$\phi = (\phi_v, \phi_i) \in \cL_v^{\infty} \times \cL_i^{\infty}$ for every $b \in \cL_v^{\infty} \times \cL_i^{\infty}$. 

It remains to check \eqref{eq:BoundednessOfInverse}.
Using the spectral theory of bounded linear self-adjoint operators and Cauchy-Schwarz inequality we can write
\begin{align*}
	\norm{L_{\rm per}^2(\Omega; \bbR)}{\phi_{i_{\rE \rE}}}^2 
	&\leq \tfrac{1}{0.3498}\inner{L_{\rm per}^2(\Omega; \bbR)}{(I-D) \phi_{i_{\rE \rE}}}{\phi_{i_{\rE \rE}}}
	= \tfrac{1}{0.3498}\inner{L_{\rm per}^2(\Omega; \bbR)}{h_{i_{\rE \rE}}}{\phi_{i_{\rE \rE}}}\\
	&\leq \tfrac{1}{0.3498} \norm{L_{\rm per}^2(\Omega; \bbR)}{h_{i_{\rE \rE}}}  \norm{L_{\rm per}^2(\Omega; \bbR)}{\phi_{i_{\rE \rE}}}.
\end{align*}
Therefore, there exists $\alpha_1 = \frac{1}{0.3498} >0 $ such that
\begin{equation*}
	\norm{L_{\rm per}^2(\Omega; \bbR)}{\phi_{i_{\rE \rE}}} \leq \alpha_1 \norm{L_{\rm per}^2(\Omega; \bbR)}{h_{i_{\rE \rE}}}.
\end{equation*}
Now, using \eqref{eq:41416649} and the Sobolev embedding theorems we can write, for some $\alpha_2, \alpha_3 >0$,
\begin{align*}
	\norm{L_{\rm per}^{\infty}(\Omega; \bbR)}{\phi_{i_{\rE \rE}}} &\leq 
	\norm{L_{\rm per}^{\infty}(\Omega; \bbR)}{h_{i_{\rE \rE}}} + \norm{L_{\rm per}^{\infty}(\Omega; \bbR)}{D \phi_{i_{\rE \rE}}}
	\leq \norm{L_{\rm per}^{\infty}(\Omega; \bbR)}{h_{i_{\rE \rE}}} + \alpha_2 \norm{H_{\rm per}^2(\Omega; \bbR)}{D \phi_{i_{\rE \rE}}} \\
	&\leq \norm{L_{\rm per}^{\infty}(\Omega; \bbR)}{h_{i_{\rE \rE}}} + \alpha_3 \norm{L_{\rm per}^2(\Omega; \bbR)}{ \phi_{i_{\rE \rE}}}
	\leq \norm{L_{\rm per}^{\infty}(\Omega; \bbR)}{h_{i_{\rE \rE}}} + \alpha_1 \alpha_3 \norm{L_{\rm per}^2(\Omega; \bbR)}{ h_{i_{\rE \rE}}} \\
	&\leq (1 + \alpha_1 \alpha_3 |\Omega|^{\frac{1}{2}}) \norm{L_{\rm per}^{\infty}(\Omega; \bbR)}{h_{i_{\rE \rE}}},
\end{align*}
which, along with the algebraic equalities \eqref{eq:41416656}, implies \eqref{eq:BoundednessOfInverse}. 
Hence, Assumption (iv) holds. 

It is now implied by Theorem \ref{th:Noncompactness} that the equilibrium sets of \eqref{eq:Voltage}--\eqref{eq:Wave}  are noncompact in $\cU_{\rs}$ and $\cU_{\rw}$.    
Moreover, it follows immediately from the equilibrium equations \eqref{eq:EquilibriumEquations} and the definition of $P_i$ given by \eqref{eq:PvPiDefinition} that, in general, all space-homogeneous equilibria $i_{\re}$ and $w_{\re}$ are positive and, in particular, belong to $\cD_{\rm Bio} \cap \cD_{\rs}$. 
Therefore, by Remark \ref{rem:PlausibilityOfEquilibrium}, the semigroups 
$\seq{S_{\rw}(t):\cD_{\rw} \rightarrow \cD_{\rw}}_{t \in [0,\infty)}$ and 
$\seq{S_{\rs}(t):\cD_{\rs} \rightarrow \cD_{\rs}}_{t \in [0,\infty)}$ associated with \eqref{eq:Voltage}--\eqref{eq:Wave} with parameter values given by Table \ref{tb:NoncompactnessParameters} do not possess a global attractor.

It can be shown by similar calculations as above that the assumptions of Theorem \ref{th:Noncompactness} are  satisfied by space-homogeneous equilibria of the model for $3$ other sets of parameter values out the $24$ sets available in \cite[Tables V and VI]{Bojak:PhysRev:2005}, namely, the sets given in \cite[Tables V, col. 2]{Bojak:PhysRev:2005} and \cite[Tables VI, col. 10 and col. 12]{Bojak:PhysRev:2005}.
Moreover, it is likely that these assumptions or their possible alternatives suggested in Remark \ref{rem:AlternativeAssumptions} would also hold for other sets of parameter values if we consider equilibria $u_{\re}$ and inputs $g$ that are not homogeneous over $\Omega$. 
Checking the assumptions of Theorem \ref{th:Noncompactness} in this case is, however, not straightforward.

\section{Discussion and Conclusion} \label{sec:Conclusion}
In this paper, we developed basic analytical results to establish a global attractor theory for the mean field model of the electroencephalogram proposed by Liley \emph{et al.}, 2002. 
We showed the boundary-initial value problem associated with the model is well-posed in the weak and strong sense,
and established sufficient conditions for the nonnegativity of the $i(x,t)$ and $w(x,t)$ components of the solution over the entire time horizon. 
Moreover, we proved existence of bounded absorbing sets for semigroups of weak and strong solutions, and discussed challenges towards proving the asymptotic compactness property for these semigroups.
Finally, we showed that the equilibrium sets of the model are noncompact for some physiologically reasonable sets of parameter values which, in particular, implies nonexistence of a global attractor.

The conditions developed in this paper for ensuring nonnegativity of the solution components $i(x,t)$ and $w(x,t)$ over the entire infinite time horizon can be useful in computational analysis of the model. 
Without using such mathematical analysis, it is impossible to ensure that the solutions computed numerically over a finite time horizon are biophysically plausible since, evidently, nonnegativity might occur for time intervals beyond the finite time horizon of numerical computations. 
This fact has been overlooked in most of the available computational analysis of the model.
However, in these computational studies, the initial values are usually set equal to a numerically computed space-homogeneous equilibrium of the model, or equal to zero when no equilibrium is found numerically. 
In both cases, the preset initial values satisfy the sufficient conditions developed in Section \ref{sec:BiophysicalPhaseSpaces} of this paper for biophysical plausibility of the solutions. 
It is perhaps an intractable problem to specify a set of biophysical initial values for a model of the EEG; however, analyzing a more diverse set of reasonable initial values satisfying the sufficient conditions developed in Section \ref{sec:BiophysicalPhaseSpaces} can be beneficial in observing different behaviors of the model.  

Existence of bounded absorbing sets is a desirable global property for a model of electrical activity in the neocortex.
As stated in Remark \ref{rem:ConditionsParameterSpace}, the EEG model investigated in this paper possesses this global property for its entire range of parameter values given in Table \ref{tb:Parameters}.
Moreover, this property holds independently of the parameters of the firing rate functions, number of intracortical and corticocortical connections, mean Nernst potentials, and membrane time constants, as observed in Assumptions (i) and (ii) of Theorems \ref{th:AbsorbingSetD1} and \ref{th:AbsorbingSetD2}.

The lack of space-dissipative terms in the ODE components \eqref{eq:Voltage} and \eqref{eq:Current} of the model is one of the major sources of difficulties towards establishing a global attractor. 
Indeed, as discussed in Section \ref{sec:NonexistenceOfAttractor}, the $v(x,t)$ and $i(x,t)$ components of the solution can evolve discontinuously in space despite continuous evolution of the $w(x,t)$ component. 
Other than disrupting the asymptotic compactness property of the semigroups of solution operators, these space irregularities can predict sharp transitions in the  $v(x,t)$ and $i(x,t)$ components of the solution, which can potentially be problematic in numerical computation of the solutions. 

Slight modifications to the model that result in the presence of additional space-dissipative terms in the ODE's
can improve the regularity of the solutions and can be of particular advantage in numerical computations. 
The fact that part of the equations of the model appears as ODE's is partially due to the simplifying assumption
of \emph{instantaneous} conduction through short-range fibers.  
Removal of such simplifying assumptions, or  considering a singularly perturbed version of \eqref{eq:Voltage} and \eqref{eq:Current} by artificially including additional diffusion terms $\varepsilon \Delta$, 
with sufficiently small $\varepsilon$, can be considered as potential modifications.  
Any such modifications should, however, maintain the neurophysiological plausibility of the model.

The regularization made by appropriate modifications on the model    
may result in the possibility of establishing the asymptotic compactness property. 
However, the analysis in Section \ref{sec:NonexistenceOfAttractor} suggests that the resulting compact attractor would be of very high dimension.
Based on this observation, we speculate that the noncompactness of the attracting sets shown in this paper can provide an explanation for the possibility of having a rich variety of behaviors for this model, 
part of which already shown by computational analysis in the literature; see for example
\cite{Bojak:Neurocomputing:2004, Bojak:Neurocomputing:2007, Frascoli:ProcSPIE:2008, Frascoli:PhysicaD:2011, Dafilis:Chaos:2001, Dafilis:Chaos:2013, Dafilis:JMN:2015, VanVeen:EPJST:2014, VanVeen:PhysRevLett:2006}.
Such diversity of complicated behaviors is indeed what would be expected from a model of the neocortex, 
the part of the brain that is presumed responsible for the extremely complicated perceptual and cognitive functionality of the brain.


\section*{Acknowledgment}
The authors would like to thank Professor Andrzej \'{S}wi\k{e}ch from the School of Mathematics at Georgia Institute of Technology for his helpful suggestions with some of the proofs appearing in this paper.


\bibliographystyle{siam}
\bibliography{References}

\end{document}